\documentclass[12pt]{amsart}

\newcommand{\hidden}[1]{
}

\usepackage{framed}
\usepackage{braket}
\usepackage{amsmath,amssymb,amsthm}
\usepackage{bm}
\usepackage[all]{xy}
\usepackage{amscd}
\usepackage{graphicx}
\usepackage{ascmac}
\usepackage[top=30truemm,bottom=30truemm,left=20truemm,right=20truemm]{geometry}
\usepackage{mathtools}
\mathtoolsset{showonlyrefs=true}

\makeatletter
\@addtoreset{equation}{section}

\makeatother

\newcommand{\Z}{\mathbb{Z}}
\newcommand{\z}{\mathbf{z}}
\newcommand{\Q}{\mathbb{Q}}
\newcommand{\R}{\mathbb{R}}
\newcommand{\F}{\mathbb{F}}
\newcommand{\C}{\mathbb{C}}
\newcommand{\T}{\mathbb{T}}
\newcommand{\kk}{\Bbbk}

\newcommand{\qu}{\mathfrak{q}}

\newcommand{\OO}{\mathcal{O}}

\newcommand{\D}{\mathcal{D}}
\newcommand{\MM}{\mathcal{M}}
\newcommand{\FF}{\mathcal{F}}
\newcommand{\GG}{{\mathcal{G}}}

\newcommand{\rr}{\mathfrak{r}}
\newcommand{\II}{\mathcal{I}}
\newcommand{\JJ}{\mathcal{J}}

\newcommand{\au}{W}

\newcommand{\FUN}{\mathcal{F}}

\newcommand{\m}{\mathfrak{m}}

\newcommand{\LL}{\mathcal{L}}

\newcommand{\AAA}{\mathcal{A}}

\newcommand{\dual}{\vee}

\newcommand{\N}{\mathcal{N}}

\newcommand{\dd}{{d}}
\newcommand{\ddd}{\ttilde{d}}
\newcommand{\rd}{h}

\newcommand{\vsim}{{\rotatebox{90}{$\sim$}}}

\newcommand{\ttilde}{\widetilde}
\newcommand{\overl}{\overline}

\newcommand{\HH}{\mathcal{H}}

\newcommand{\K}{\mathcal{K}}

\newcommand{\RR}{\mathcal{R}}
\newcommand{\ZZ}{\mathcal{Z}}
\newcommand{\CC}{\mathcal{C}}
\newcommand{\PP}{\mathcal{P}}
\newcommand{\DD}{\mathcal{D}}

\newcommand{\SF}[1]{\mathcal{F}^{\langle #1 \rangle}}
\newcommand{\varSF}[1]{\mathcal{F}^{[#1]}}

\newcommand{\step}[2]{{\bf Step #1.} #2}

\newcommand{\pd}{{\rm pd}}

\newcommand{\aalpha}{s}
\newcommand{\bbeta}{t}
\newcommand{\X}{T}

\newcommand{\Rnt}{R_0^{\rm{nt}}}
\newcommand{\RRnt}{\RR^{\rm nt}}

\DeclareMathOperator{\Gal}{Gal}
\DeclareMathOperator{\Tam}{Tam}
\DeclareMathOperator{\GL}{GL}
\DeclareMathOperator{\SL}{SL}

\DeclareMathOperator{\Ker}{Ker}
\DeclareMathOperator{\Image}{Im}
\DeclareMathOperator{\Cok}{Cok}

\DeclareMathOperator{\Ext}{Ext}
\DeclareMathOperator{\Aut}{Aut}

\DeclareMathOperator{\rank}{rank}
\DeclareMathOperator{\corank}{corank}

\DeclareMathOperator{\ord}{ord}

\DeclareMathOperator{\ur}{ur}

\DeclareMathOperator{\CRK}{CRK}

\DeclareMathOperator{\Tr}{Tr}
\DeclareMathOperator{\prim}{prime}
\DeclareMathOperator{\Cole}{Col}
\DeclareMathOperator{\Sel}{Sel}
\DeclareMathOperator{\loc}{loc}
\DeclareMathOperator{\sign}{sign}
\DeclareMathOperator{\Log}{Log}

\DeclareMathOperator{\fin}{fin}
\DeclareMathOperator{\ev}{ev}
\DeclareMathOperator{\LEO}{LEO}
\DeclareMathOperator{\RFX}{RFX}
\DeclareMathOperator{\LOC}{LOC}
\DeclareMathOperator{\can}{can}
\DeclareMathOperator{\per}{per}

\DeclareMathOperator{\tame}{tame}

\DeclareMathOperator{\Fitt}{Fitt}
\DeclareMathOperator{\Hom}{Hom}

\DeclareMathOperator{\id}{id}
\DeclareMathOperator{\ES}{ES}
\DeclareMathOperator{\KS}{KS}
\DeclareMathOperator{\SSS}{SS}
\DeclareMathOperator{\Reg}{Reg}

\DeclareMathOperator{\primitive}{prim}
\DeclareMathOperator{\imp}{imp}

\let\oldenumerate\enumerate
\renewcommand{\enumerate}{
   \oldenumerate
   \setlength{\itemsep}{1pt}
   \setlength{\parskip}{0pt}
   \setlength{\parsep}{0pt}
}
\let\olditemize\itemize
\renewcommand{\itemize}{
   \olditemize
   \setlength{\itemsep}{1pt}
   \setlength{\parskip}{0pt}
   \setlength{\parsep}{0pt}
}

\theoremstyle{plain}
\newtheorem{thm}{Theorem}[section]
\newtheorem{lem}[thm]{Lemma}
\newtheorem{conj}[thm]{Conjecture}
\newtheorem{prop}[thm]{Proposition}
\newtheorem{cor}[thm]{Corollary}

\newtheorem{ass}[thm]{Assumption}
\theoremstyle{definition}
\newtheorem{defn}[thm]{Definition}
\newtheorem{rem}[thm]{Remark}

\title[Equivariant Iwasawa theory for elliptic curves]{Equivariant Iwasawa theory for elliptic curves}
\author[T. Kataoka]{Takenori Kataoka}
\address{Faculty of Science and Technology, Keio University.
3-14-1 Hiyoshi, Kohoku-ku, Yokohama, Kanagawa 223-8522, Japan}
\email{tkataoka@math.keio.ac.jp}
\keywords{Selmer groups, $p$-adic $L$-functions, main conjecture, Mazur-Tate conjecture}
\subjclass[2010]{11R23}

\begin{document}

\begin{abstract}

We discuss abelian equivariant Iwasawa theory for elliptic curves over $\Q$ at good supersingular primes and non-anomalous good ordinary primes.
Using Kobayashi's method, we construct equivariant Coleman maps, which send the Beilinson-Kato element to the equivariant $p$-adic $L$-functions.
Then we propose equivariant main conjectures and, under certain assumptions, prove one divisibility via Euler system machinery.
As an application, we prove a case of a conjecture of Mazur-Tate.

\end{abstract}

\maketitle


\section{Introduction}\label{sec:01}

Iwasawa theory began with the study of the behavior of the ideal class groups along $\Z_p$-extensions.
On the one hand, a variety of equivariant refinements of the study of ideal class groups were developed; for example, see Ritter-Weiss \cite{RW02} and Greither-Popescu \cite{GP15}.
On the other hand, Iwasawa theory for elliptic curves was also developed; see Greenberg \cite{Gree99} as a basic reference.
The purpose of this paper is to propose an equivariant refinement of Iwasawa theory for elliptic curves.

We fix notations as follows.
Let $E$ be an elliptic curve over the rational number field $\Q$ with good reduction at a fixed odd prime number $p$.
Put $a_p = (1+p)-\sharp \ttilde{E}(\F_p)$, where $\ttilde{E}$ denotes the reduction of $E$ modulo $p$.
Let $K$ be a finite abelian extension of $\Q$ where $p$ is unramified.
Put $K_n = K(\mu_{p^{n+1}})$ for integers $n \geq -1$ and put $K_{\infty} = K(\mu_{p^{\infty}})$.
Note that $K_{\infty}$ is the cyclotomic $\Z_p$-extension of $K_0$.
Put $\RR = \Z_p[[\Gal(K_{\infty}/\Q)]]$ and $\Lambda = \Z_p[[\Gal(K_{\infty}/K_0)]]$, the Iwasawa algebras.

This paper concerns the analysis of $E$ along the extension $K_{\infty}/\Q$.
Roughly speaking, previous works usually study the case where $K = \Q$.
In particular, in that case, the works by Kato \cite{Kat04}, Kobayashi \cite{Kob03}, and Sprung \cite{Spr12} establish the one divisibility of the main conjecture in the case where $p \nmid a_p, a_p = 0$, and $p \mid a_p$, respectively.
Moreover, as an application of that divisibility, C.-H.~Kim-Kurihara \cite{KK} proves a certain special case of a conjecture of Mazur-Tate \cite{MT87} (the weak main conjecture).
In this paper, we will generalize these works equivariantly, so that specializing $K = \Q$ recovers them.

Here we recall the general motivation for the equivariant theory.
We have a canonical identification $\RR = \Lambda[G_0]$, where $G_0 = \Gal(K_0/\Q)$.
In the ``non-equivariant'' theory, we decompose modules over $\RR$ into $\chi$-parts, where $\chi$ runs through characters of $G_0$.
This decomposition is harmless if $p \nmid [K:\Q]$, since then $\RR$ is the direct product of the $\chi$-parts of $\RR$.
However, if $p \mid [K:\Q]$, this decomposition loses certain information about the original modules.
Such information should not be lost when we aim at precise results at finite layers, such as the Mazur-Tate conjecture.
Therefore we should work with the original $\RR$-modules directly not taking $\chi$-parts.
This is the idea of the equivariant theory.

In this paper, we are particularly interested in the supersingular case.
One reason is that several results in the ordinary case are already established in previous works.
In fact, the works by Kurihara \cite{Kur03}, \cite{Kur14} also study equivariant refinements of Iwasawa theory for elliptic curves in the ordinary case.
Those works have certain overlaps with this paper.
On the other hand, the supersingular case has more difficulties, and the way to resolve them is particularly innovating in this paper.
Nonetheless, in this paper we treat both the ordinary case and the supersingular case simultaneously.
In fact, this paper certainly contains new results even in the ordinary case.
Moreover, the ordinary case sometimes works as a toy model for the supersingular case.

\subsection*{Algebraic Side}

Let $S$ be a finite set of prime numbers $\neq p$ such that $S$ contains all prime numbers which are ramified in $K/\Q$.
A fundamental algebraic object of study in our context is the $S$-imprimitive ($p$-primary) Selmer group $\Sel_S(E/K_{\infty})$, which is naturally a cofinitely generated $\RR$-module.
See Subsection \ref{subsec:402} for the precise definitions and properties of various Selmer groups.
In the ordinary case (namely $p \nmid a_p$), $\Sel_S(E/K_{\infty})$ is known to be $\Lambda$-cotorsion.
On the other hand, in the supersingular case (namely $p \mid a_p$), $\Sel_S(E/K_{\infty})$ is not $\Lambda$-cotorsion.
In order to resolve this issue, we need to modify the local conditions of Selmer groups at $p$, as follows.

When $a_p=0$, such a modification can be done via an idea of Kobayashi \cite{Kob03}, further developed by Iovita-Pollack \cite{IP06}, Kitajima-Otsuki \cite{KO18}, and others.
The idea is to introduce canonical subgroups $E^{\pm}(K_{\infty} \otimes \Q_p)$ of $E(K_{\infty} \otimes \Q_p)$ and define the corresponding $\pm$-Selmer groups $\Sel_S^{\pm}(E/K_{\infty})$.
They are submodules of $\Sel_S(E/K_{\infty})$ and known to be $\Lambda$-cotorsion.

Consider general $p \mid a_p$.
When $K = \Q$, Sprung \cite{Spr12} developed such a modification and gave rise to $\sharp/\flat$-Selmer groups.
Though the construction is much more complicated than in the $a_p = 0$ case, in this paper we extend the argument by Sprung to general $K$.
As a result, we define subgroups $E_{\infty}^{\sharp/\flat}$ of $E(K_{\infty} \otimes \Q_p) \otimes (\Q_p/\Z_p)$ and the $\sharp/\flat$-Selmer groups $\Sel_S^{\sharp/\flat}(E/K_{\infty})$.
Then it seems reasonable to conjecture that they are $\Lambda$-cotorsion, which we assume throughout this section.
We mention here that $p \geq 5$ and $p \mid a_p$ imply $a_p = 0$ by the Hasse bound, so that $\sharp/\flat$ is concerned essentially only when $p = 3$.

To treat the various Selmer groups simultaneously, we use the notation $\Sel_S^{\bullet}(E/K_{\infty})$ for $\bullet \in \{\emptyset, +, -, \sharp, \flat \}$ ($\emptyset$ means ``no sign'').
As a promise, when $p \nmid a_p$ (resp. $a_p = 0$, resp. $p \mid a_p$), we must have $\bullet = \emptyset$ (resp. $\bullet \in \{+, -\}$, resp. $\bullet \in \{+, -, \sharp, \flat\}$).
Moreover, when $a_p = 0$, we always keep the convention that $(\sharp, \flat) = (-, +)$. 
For example, when $a_p = 0$, we will have $E_{\infty}^{\sharp} = E^-(K_{\infty} \otimes \Q_p) \otimes (\Q_p/\Z_p)$ and $E_{\infty}^{\flat} = E^+(K_{\infty} \otimes \Q_p) \otimes (\Q_p/\Z_p)$ (Corollary \ref{cor:870}).

In equivariant Iwasawa theory, in addition to being $\Lambda$-cotorsion, the finiteness of the projective dimension over $\RR$ (denoted by $\pd_{\RR}$) of the Pontryagin dual is important.
This corresponds to \cite[Theorem 4.6]{GP15}, \cite[Theorem 1]{RW02} in the study of ideal class groups.
If $p \nmid [K: \Q]$, then the finiteness of $\pd_{\RR}$ is trivial since $\RR$ is then a product of regular local rings.
Hence we are mainly concerned in the case where $p \mid [K:\Q]$.
In this paper, we prove the following theorem, under the {\it non-anomalous} condition (Assumption \ref{ass:04}) in the ordinary case.
The non-anomalous condition is necessary whenever we treat the ordinary case in this paper, and we assume it throughout this section.

We denote the Pontryagin dual by the superscript $(-)^{\dual}$.
Recall that $\Sel_S^{\bullet}(E/K_{\infty})^{\dual}$ is known to be (resp. assumed to be) $\Lambda$-torsion for $\bullet \in \{\emptyset, +, -\}$ (resp. $\bullet \in \{ \sharp, \flat\}$).

\begin{thm}[Theorem \ref{thm:978}]\label{thm:76}
For $\bullet \in \{\emptyset, +, -, \sharp, \flat \}$,
we have $\pd_{\RR}(\Sel_S^{\bullet}(E/K_{\infty})^{\dual}) \leq 1$.
\end{thm}

The proof is given in Section \ref{sec:220}.
See Remark \ref{rem:871} for several related works.
In fact, in the ordinary case, the assertion of Theorem \ref{thm:76} follows from a work of Greenberg \cite{Gre10}.
However, our proof, especially the construction of the Coleman map below, will play an important role in the rest of this paper.

Put $R_{n} = \Z_p[\Gal(K_n/\Q)]$ for $n \geq -1$.
Put $\Delta = \Gal(\Q(\mu_p)/\Q)$ and let $N_{\Delta}$ denote the norm element of $\Delta$, namely $N_{\Delta} = \sum_{\delta \in \Delta} \delta \in \Z_p[\Delta]$.
It is an important observation that, since $\sharp \Delta = p-1$, any $\Z_p[\Delta]$-module can be decomposed into $\eta$-parts for characters $\eta$ of $\Delta$.
For example, we can decompose $R_0 = R_0^{\Delta} \times \Rnt$ as a ring, where $R_0^{\Delta} = N_{\Delta}R_0  \simeq R_{-1}$ is the trivial character part and $\Rnt$ denotes the sum of the non-trivial character parts. 

The key ingredient in the proof of Theorem \ref{thm:76} is Theorem \ref{thm:93} below.
It claims the precise $\RR$-module structures of the local conditions which define the concerned Selmer groups.
They can be regarded as refinements of several previous works as in Remark \ref{rem:691}.

\begin{thm}[(1) Theorem \ref{thm:973}, (2) Definition \ref{defn:231}, (3) Theorem \ref{thm:34}]\label{thm:93}
As $\RR$-modules, the following hold.

(1) If $p \nmid a_p$, we have an isomorphism 
\[
(E(K_{\infty} \otimes \Q_p) \otimes (\Q_p/\Z_p))^{\dual} \overset{\sim}{\to} \RR.
\]

(2) If $p \mid a_p$, we have exact sequences
\begin{equation}\label{eq:201}
0 \to (E^{\flat}_{\infty})^{\dual} \to \RR \oplus R_{-1} \to R_{-1} \to 0
\end{equation}
and
\begin{equation}\label{eq:202}
0 \to (E^{\sharp}_{\infty})^{\dual} \to \RR \to \Rnt/(a_p) \to 0.
\end{equation}
(3) If $a_p = 0$, we have exact sequences
\[
0 \to (E^{+}(K_{\infty} \otimes \Q_p) \otimes (\Q_p/\Z_p))^{\dual} \to \RR \oplus R_{-1} \to R_{-1} \to 0,
\]
and 
\[
0 \to (E^{-}(K_{\infty} \otimes \Q_p) \otimes (\Q_p/\Z_p))^{\dual} \to \RR \to \Rnt \to 0.
\]
\end{thm}

The proof is given in Section \ref{sec:32} (in fact, in our treatment, the assertion (2) is almost the definition of $E_{\infty}^{\sharp/\flat}$).
The basic method relies on Kitajima-Otsuki \cite{KO18} and Sprung \cite{Spr12}, which are in turn based on the preceeding works by Kobayashi \cite{Kob03}, Iovita-Pollack \cite{IP06}, etc.
More precisely, we construct a system of local points, namely elements of $E(K_n \otimes \Q_p)$, satisfying certain norm compatibilities.
The construction is done in Section \ref{sec:02} via Honda theory on formal groups.
However, we need more careful construction than in the previous works (see Remark \ref{rem:967}).

In the proof of Theorem \ref{thm:93}, we construct the Coleman map
\[
\Cole^{\bullet}: (E(K_{\infty} \otimes \Q_p) \otimes (\Q_p/\Z_p))^{\dual} \to \RR
\]
for $\bullet \in \{\emptyset, +, -, \sharp, \flat\}$ (Definitions \ref{defn:877}, \ref{defn:400}, and \ref{defn:878}).
When $a_p = 0$, we have $(\Cole^{\sharp}, \Cole^{\flat}) = (\Cole^{-}, \Cole^{+})$, which is consistent with our convention $(\sharp, \flat) = (-, +)$.
Moreover, by composing with a canonical map, $\Cole^{\bullet}$ can be regarded as a map from the local Iwasawa cohomology group,
\[
\Cole^{\bullet}: \varprojlim_n H^1(K_n \otimes \Q_p, T_pE) \to \RR.
\]
Here $T_pE$ denotes the $p$-adic Tate module of $E$.
These Coleman maps generalize the previous works on the case where $K = \Q$.
In that case, such Coleman maps were constructed by Perrin-Riou \cite{Per94} (see \cite[Theorem 16.4]{Kat04}), \cite[Definition 8.22]{Kob03}, and \cite[Definition 5.9]{Spr12}.

\subsection*{Analytic Side}

Let $\alpha \in \overline{\Q_p}$ be a root of $t^2-a_pt+p = 0$ with $\ord_p(\alpha) <1$ (often called an {\it allowable root}).
Here $\ord_p$ denotes the additive $p$-adic valuation map normalized as $\ord_p(p) = 1$.
Then by Amice-Velu \cite{AV75}, Vi{\v{s}}ik \cite{Vis76}, or Mazur-Tate-Teitelbaum \cite{MTT86}, there is a $p$-adic $L$-function $\LL_S(E/K_{\infty}, \alpha)$ which interpolates the $L$-values $L_S(E, \psi, 1)$ for Dirichlet characters $\psi$.
See Subsection \ref{subsec:403} for the precise definitions and properties of various $p$-adic $L$-functions.
Note that our interpolation property does not seem to appear explicitly in the literature; see Remarks \ref{rem:798} and \ref{rem:812} for this issue.
The interpolation property is determined for Theorem \ref{thm:77} below to hold.

In the ordinary case, we have $\LL_S(E/K_{\infty}, \alpha) \in \RR \otimes \Q_p$ for the only one allowable root $\alpha$ (namely, the unit root).
In that case, simply put $\LL_S(E/K_{\infty}) = \LL_S(E/K_{\infty}, \alpha) \in \RR \otimes \Q_p$.
On the other hand, in the supersingular case, $\LL_S(E/K_{\infty}, \alpha)$ is not contained in $\RR \otimes \Q_p(\alpha)$.
This trouble can be regarded as the analytic counterpart to the failure of the torsionness on the algebraic side.
Such a trouble can be resolved by the idea of Pollack \cite{Pol03} ($a_p = 0$ case) and Sprung \cite{Spr12} (general case).
As a result, we obtain $p$-adic L-functions $\LL_S^{\pm}(E/K_{\infty}) \in \RR \otimes \Q_p$ if $a_p = 0$ and $\LL_S^{\sharp/\flat}(E/K_{\infty}) \in \RR \otimes \Q_p$ if $p \mid a_p$.

We will show that these $p$-adic $L$-functions can be obtained by applying our Coleman maps to the Beilinson-Kato element, as follows.
This assertion is a generalization of the results on $K = \Q, S = \emptyset$ by \cite[Theorem 16.6]{Kat04}, \cite[Theorem 6.3]{Kob03}, and \cite[Theorem 6.12]{Spr12}.
Let $\z$ be the Beilinson-Kato element, introduced in Theorem \ref{thm:941}.
This element $\z$ lives in the global Iwasawa cohomology group, and we denote by $\loc(\z)$ its image to the local Iwasawa cohomology group.
We denote by the superscript $\iota$ the involution of a (completed) group ring induced by inverting each of the elements of the group.

\begin{thm}[Theorems \ref{thm:75} and \ref{thm:71}] \label{thm:77}
For $\bullet \in \{\emptyset, +, -, \sharp, \flat \}$, we have
\[
\LL_S^{\bullet}(E/K_{\infty})^{\iota} = \Cole^{\bullet}(\loc(\z))
\]
in $\RR \otimes \Q_p$.
\end{thm}

The proof is given in Section \ref{sec:60}.
 The appearance of $\iota$ in the left hand side is explained in Remark \ref{rem:812}.
 
 \subsection*{Equivariant Main Conjecture}
 
  By Theorem \ref{thm:76}, the (initial) Fitting ideal $\Fitt_{\RR}(\Sel_S^{\bullet}(E/K_{\infty})^{\dual})$ of the Pontryagin dual of each Selmer group is a principal ideal.
Thus it is reasonable to formulate an equivariant main conjecture as a connection between $\Fitt_{\RR}(\Sel_S^{\bullet}(E/K_{\infty})^{\dual})$ and $(\LL_S^{\bullet}(E/K_{\infty}))$.

When $K = \Q$ and $S = \emptyset$, the formulation of $(-)$-main conjecture in \cite[\S 4]{Kob03} requires an explicit auxiliary factor.
Motivated by that, we define auxiliary ideals $\au^{\bullet}$ of $\RR$ by
\[
\au^{\bullet} =
	\begin{cases}
		\RR & (\bullet \in \{ \emptyset, +, \flat\})\\
		\RR & (p \mid a_p \neq 0, \bullet = \sharp)\\
		(N_{\Delta}, \gamma - 1) & (a_p = 0, \bullet \in \{-, \sharp\}).
	\end{cases}
\]
Here, $\gamma$ denotes a fixed generator of $\Gamma$.
Recall that we can decompose $\RR = \RR^{\Delta} \times \RRnt$, where $\RRnt$ denotes the sum of non-trivial character parts with respect to the action of $\Delta$.
Then, in the $a_p=0$ case, we have $\au^- = \au^{\sharp} = \RR^{\Delta} \times (\gamma-1)\RRnt$, which shows that these are in fact principal ideals.

Hopefully the equivariant main conjecture in our context may be formulated as
\begin{equation}\label{eq:241}
\au^{\bullet} \Fitt_{\RR}(\Sel_S^{\bullet}(E/K_{\infty})^{\dual}) = (\LL_S^{\bullet}(E/K_{\infty}))
\end{equation}
for $\bullet \in \{\emptyset, +, -, \sharp, \flat\}$.
When $K = \Q$ and $S = \emptyset$, the assertion \eqref{eq:241} is equivalent to the previous formulations of main conjectures (\cite[Conjecture 17.6]{Kat04}, \cite[\S 4]{Kob03}, and \cite[Conjecture 7.21]{Spr12}).
However, the author is not convinced of the formula \eqref{eq:241} in full generality, so we do not propose it as a conjecture.
Nonetheless, let us refer \eqref{eq:241} as {\it the equivariant main conjecture} in this paper.
 In Proposition \ref{prop:615}, we will show that \eqref{eq:241} is independent from the choice of $S$, similarly as in the work by Greenberg-Vatsal \cite[Theorem 1.5]{GV00}.

\begin{rem}\label{rem:797}
Consider the situation of Kurihara \cite[Theorem 6]{Kur14}.
In particular, $E$ is ordinary at $p$ and $K/\Q$ is a $p$-extension.
As claimed later in Remark \ref{rem:798}, the element $\xi_{K_{\infty}, S}$ in \cite{Kur14} coincides with our $\LL_S(E/K_{\infty})$ up to a unit.
Then \cite[Theorem 6(2)]{Kur14} shows that the classical main conjecture for $E$ over the cyclotomic $\Z_p$-extension of $\Q$ is equivalent to the equivariant main conjecture \eqref{eq:241}.
This provides evidence of our formulation \eqref{eq:241}.
\end{rem}

When $K = \Q$ and $S = \emptyset$, under certain hypotheses, one divisibility of the (non-equivariant) main conjecture is proved in \cite[Theorem 17.4]{Kat04}, \cite[Theorem 4.1]{Kob03}, and \cite[Theorem 7.16]{Spr12} via the Euler system machinery.
In this paper, as a generalization of them, we prove one divisibility of the equivariant main conjecture as follows.
Let $\Sel^0(E/K_{\infty})$ be the fine Selmer group.

\begin{thm}\label{thm:95}
Suppose the following conditions hold.

(a) $H^0(K_{\infty} \otimes \Q_l, E[p]) = 0$ for any prime number $l$ which is ramified in $K/\Q$.

(b) $H^0(K_{\infty} \otimes \Q_l, E[p^{\infty}])$ is a divisible $\Z_p$-module for any prime number $l$.

(c) The Galois representation $\Gal(\overline{\Q}/\Q) \to \Aut(E[p^{\infty}]) \simeq \GL_2(\Z_p)$ is surjective.

(d) $p \geq 5$.

(e) Either $p \nmid [K:\Q]$ or the $\mu$-invariant of $\Sel^0(E/K_{\infty})^{\dual}$ as a $\Lambda$-module is $0$.

(f) $E$ has good reduction at any prime number $l$ which is ramified in $K/\Q$.

Then we have
\[
\au^{\bullet} \Fitt_{\RR}(\Sel_S^{\bullet}(E/K_{\infty})^{\dual}) \supset (\LL_S^{\bullet}(E/K_{\infty}))
\]
for $\bullet \in \{\emptyset, +, -\}$.
\end{thm}

The proof is given in Section \ref{sec:84}.

\begin{rem}\label{rem:a11}
The conditions (a)--(d) are needed for Euler system argument (Theorem \ref{thm:964}).
The condition (b) is equivalent to $p \nmid \Tam(E/K_0)$, the Tamagawa factor of $E$ over $K_0$ (see \cite[p. 74]{Gree99}).
By the open image theorem of Serre \cite{Ser72}, if $E$ does not admit complex multiplication, the condition (c) holds for all but finitely many $p$.
The proof of Theorem \ref{thm:95} first show the inclusion where the right hand side is replaced by $(\LL_S^{\bullet}(E/K_{\infty})^{\iota})$, and the condition (f) (together with (a)) is used to show $(\LL_S^{\bullet}(E/K_{\infty})^{\iota}) = (\LL_S^{\bullet}(E/K_{\infty}))$.

The condition (e) on $\mu = 0$ is conjectured to be true in general by Coates-Sujatha \cite[p.822, Conjecture A]{CS05}.
By \cite[Theorem 3.4]{CS05}, the condition (e) is equivalent to the vanishing of the $\mu$-invariant of the unramified Iwasawa module for $K_{\infty}(E[p])$, which is a very standard conjecture.
However, by our condition (c), $K_{\infty}(E[p])$ is never abelian over $\Q$ and thus the standard conjecture is still unachievable.

It should be remarkable that, if we adopt a ``standard'' method (\cite{BG03}, \cite{GP15}, \cite{Kur03}, etc.) to verify the equivariant main conjecture, we should suppose the stronger hypothesis that $\Sel_S^{\bullet}(E/K_{\infty})^{\dual}$ satisfies $\mu = 0$.
In particular, the work of Kurihara \cite{Kur03}, which treats the ordinary case, is very close to ours, but he assumes the stronger hypothesis.
Roughly speaking, under the stronger hypothesis, the equivariant main conjecture can be deduced from the character-wise main conjecture.
Our method in this paper to prove one divisibility does not rely on the character-wise one.
Instead, we use an equivariant Euler system argument, which is developed recently by Burns, Sano, and Sakamoto (\cite{BSS}, \cite{BS}, \cite{Sak}; we do not need the higher rank theory, though).
The condition (e) is, however, necessary to connect the result from the equivariant Euler system argument and our formulation of the equivariant main conjecture.

In fact, we can remove the condition (e) from Theorem \ref{thm:95}, by developing a certain refinement of the works by Burns, Sano, and Sakamoto.
However, the proof requires detailed discussion on the theory of Euler systems, which is outside the scope of the present paper, so the proof will be given in a forthcoming paper \cite{Kata_11}.
\end{rem}

\subsection*{Application to a Conjecture of Mazur-Tate} 
 
Mazur-Tate \cite{MT87} formulated a conjecture which predicts a relation between Selmer groups and $L$-values.
Let $M$ be a positive integer.
The algebraic object is the (primitive) Selmer group $\Sel(E/\Q(\mu_M))$, whose definition is given in Subsection \ref{subsec:402}.
The analytic object is the Mazur-Tate element $\theta_M \in \Q[\Gal(\Q(\mu_M)/\Q)]$ associated to $E$, whose definition is given in Section \ref{sec:07}.
We often have $\theta_M \in \Z_p[\Gal(\Q(\mu_M)/\Q)]$, that is, the coefficients of $\theta_M$ are $p$-adic integers; see Remark \ref{rem:a03} for a sufficient condition.
 Then a conjecture of Mazur-Tate \cite[Conjecture 3 (weak main conjecture)]{MT87} claims the following (strictly speaking, \cite{MT87} only considers the real part and also assumes the finiteness of the Tate-Shafarevich group).
 
 \begin{conj}\label{conj:202}
 For any positive integer $M$, we have 
 \[
 \theta_M \in \Fitt_{\Z_p[\Gal(\Q(\mu_M)/\Q)]} (\Sel(E/\Q(\mu_M))^{\dual}),
 \]
 as long as $\theta_M \in \Z_p[\Gal(\Q(\mu_M)/\Q)]$.
 \end{conj}
 
When $M$ is a power of $p$, using the validity of non-equivariant ($\pm$-)main conjecture, C.-H. Kim-Kurihara \cite[Theorem 1.14]{KK} proved certain cases of this conjecture (more precisely, they treated only the number fields contained in the $\Z_p$-extension of $\Q$, namely, the $\Delta$-invariant parts).
In this paper, we extend that work to general $M$.
Let $N$ be the conductor of $E$.
For $l \nmid pN$, let $\ttilde{E}_l$ denote the reduction of $E$ modulo $l$.

\begin{thm}\label{thm:203}
Suppose that $p \nmid a_p$ or $a_p = 0$ holds.
Let $m$ be a positive integer relatively prime to $pN$.
Suppose the following holds.
\begin{quote}
$(\star)$ Let $l$ be a prime divisor of $m$ with $l \equiv 1 \mod p$.
If $a_l \equiv 2 \mod p$, then $\sharp \ttilde{E}_l(\F_l)[p] \neq p^2$.
If $a_l \equiv -2 \mod p$ and the residue degree of $l$ in $\Q(\mu_m)/\Q$ is even, then $\sharp \ttilde{E}_l(\F_{l^2})[p] \neq p^2$.
\end{quote}
Then, if 
\[
\au^{\bullet} \Fitt_{\RR}(\Sel_S^{\bullet}(E/K_{\infty})^{\dual}) \supset (\LL_S^{\bullet}(E/K_{\infty}))
\]
holds for $K = \Q(\mu_m)$ and any possible $\bullet \in \{\emptyset, +, -\}$, then Conjecture \ref{conj:202} holds for $M = mp^{n+1}$ with any $n \geq -1$.
\end{thm}

The proof is given in Section \ref{sec:07}.
The very basic idea follows \cite{KK}, but significant difficulties appear.
On the algebraic side, we will have to compare the Fitting ideals of $\Sel_{\prim(m)}^{\bullet}(E/\Q(\mu_{mp^{\infty}}))^{\dual}$ and $\Sel^{\bullet}(E/\Q(\mu_{mp^{\infty}}))^{\dual}$.
We shall do the task by applying the author's work \cite{Kata} on the Fitting ideals.
The condition $(\star)$ is necessary in that computation.
On the analytic side, we will have to compare $\LL_{\prim(m)}(E/\Q(\mu_{mp^{\infty}}), \alpha)$ with the $\alpha$-stabilized Mazur-Tate element $\vartheta_{m, n}^{\alpha}$ (defined in \eqref{eq:150}).
That task is not easy because they have different compatibilities \eqref{eq:905} and \eqref{eq:906} in varying $m$.

The condition (a) in Theorem \ref{thm:95} implies $(\star)$ in Theorem \ref{thm:203} (as long as $(m, pN) = 1$).
Therefore the following is an immediate corollary of Theorems \ref{thm:95} and \ref{thm:203}.

\begin{cor}\label{cor:204}
Let $m$ be a positive integer.
Suppose that the conditions (a)--(f) in Theorem \ref{thm:95} hold for $K = \Q(\mu_m)$.
Then Conjecture \ref{conj:202} holds for $M = mp^{n+1}$ with any $n \geq -1$.
\end{cor}

If we are only concerned with Corollary \ref{cor:204}, by the condition (a), we can avoid the difficulties in the proof of Theorem \ref{thm:203} explained above.
Nonetheless, Theorem \ref{thm:203} itself is valuable.
 
\begin{rem}
Mazur-Tate proposed another conjecture \cite[Conjecture 1]{MT87} (weak vanishing conjecture), concerning the vanishing order of the Mazur-Tate element and the Mordell-Weil rank.
They also show \cite[p. 720, Proposition 3]{MT87} that the weak vanishing conjecture is implied by Conjecture \ref{conj:202} (weak main conjecture).
Hence Corollary \ref{cor:204} proves the weak vanishing conjecture in that case.
On the other hand, Ota \cite{Ota18} proved the weak vanishing conjecture for the trivial character under certain hypotheses.
The nature of his hypotheses is very close to ours, but he does not require the condition (a) or (e).
\end{rem}

 \subsection*{Outline of this Paper}
\begin{itemize}
\item In Section \ref{sec:401}, we recall the definitions of various Selmer groups and $p$-adic $L$-functions.
\item In Section \ref{sec:02}, we construct a system of local points which will be used in the next section.
\item In Section \ref{sec:32}, we construct Coleman maps and prove Theorem \ref{thm:93} (the structures of local conditions).
\item In Section \ref{sec:220}, we prove Theorem \ref{thm:76} (algebraic side: the finiteness of the projective dimension).
\item In Section \ref{sec:60}, we prove Theorem \ref{thm:77} (analytic side: our Coleman maps send the Beilinson-Kato element to the $p$-adic $L$-functions).
\item In Section \ref{sec:84}, we prove Theorem \ref{thm:95} (one divisibility of the equivariant main conjecture).
\item In Section \ref{sec:07}, we prove Theorem \ref{thm:203} (the equivariant main conjecture implies the Mazur-Tate conjecture).
\item Section \ref{sec:108} is devoted to developing auxiliary propositions which are used in Sections \ref{sec:84} and \ref{sec:07}.
\end{itemize}

\subsection*{Notations}\label{subsec:59}

We fix some notations which are used throughout this paper.
Some of them are already introduced, but we restate them here.

Fix an odd prime number $p$.
Let $\overline{\Q}$ and $\overline{\Q_p}$ be fixed algebraic closures of $\Q$ and $\Q_p$, respectively.
We fix embeddings of $\overline{\Q}$ into $\overline{\Q_p}$ and into $\C$, the field of complex numbers.
For a positive integer $M$, we denote by $\mu_M$ the group of $M$-th roots of unity in $\overline{\Q}$.
Fix a system $(\zeta_{M})_{M}$ in $\overline{\Q}$, indexed by positive integers $M$, such that $\zeta_{M}$ is a primitive $M$-th root of unity and $(\zeta_{M})^{M/M'} = \zeta_{M'}$ for $M' \mid M$.
We denote by $\prim(M)$ the set of prime divisors of $M$.

Let $E$ be an elliptic curve over $\Q$ with good reduction at $p$.
Let $N$ be its conductor.
Put $a_p = (1+p)-\sharp \ttilde{E}(\F_p)$, where $\ttilde{E}$ denotes the reduction of $E$ modulo $p$.
Similarly, for a prime number $l \nmid pN$, put $a_l = (1+l) - \sharp \ttilde{E}_l(\F_l)$, where $\ttilde{E}_l$ denotes the reduction of $E$ modulo $l$.
Moreover, if $l \mid N$, put $a_l = +1, -1$, and $0$ when the reduction type at $l$ is split multiplicative, non-split multiplicative, and additive, respectively.

Let $K$ be a finite abelian extension of $\Q$ where $p$ is unramified.
Let $S$ be a finite set of prime numbers $\neq p$ which contains all prime numbers which are ramified in $K/\Q$.
For $n \geq -1$ or $n = \infty$, put $K_n = K(\mu_{p^{n+1}})$ and $G_n = \Gal(K_n/\Q)$.
Put $\Delta = \Gal(K_0/K_{-1}) \simeq \Gal(\Q(\mu_p)/\Q)$ and $\Gamma = \Gal(K_{\infty}/K_0) \simeq \Gal(\Q(\mu_{p^{\infty}})/\Q(\mu_p))$, which are independent of $K$.
Put $\RR = \Z_p[[G_{\infty}]]$ and $\Lambda = \Z_p[[\Gamma]]$.
Fix a topological generator $\gamma$ of $\Gamma$, by which we have an isomorphism $\Lambda \simeq \Z_p[[\X]]$, sending $\gamma$ to $1+\X$.
Note that we have natural identifications $G_{\infty} = G_0 \times \Gamma$, $G_0 = G_{-1} \times \Delta$, and thus $\RR = \Lambda[G_0] = \Lambda[G_{-1}][\Delta]$.

Put $k_n = K_n \otimes \Q_p$, which is a product of local fields.
Let $\OO_n$ be the ring of integers of $k_n$, namely the integral closure of $\Z_p$ in $k_{n}$.
Let $\m_{n}$ be the Jacobson radical of $\OO_{n}$.

For a compact or discrete $\Z_p$-module $X$, we denote by $X^{\dual} = \Hom_{\Z_p}(X, \Q_p/\Z_p)$ the Pontryagin dual.
When $X$ is finitely generated over $\Z_p$, we also define the $\Z_p$-linear dual by $X^* = \Hom_{\Z_p}(X,\Z_p)$.
If $X$ has a left action of a group $G$, then $X^{\dual}$ (resp. $X^*$) also has a left action of $G$ defined by $(gf)(x) = f(g^{-1}x)$ for $g \in G, x \in X$, and $f \in X^{\dual}$ (resp. $f \in X^*$).

For a Galois extension $F'/F$ of fields and a continuous $\Gal(F'/F)$-module $X$, we denote by $H^i(F'/F, X)$ the continuous Galois cohomology $H^i(\Gal(F'/F), X)$.
In particular, if $F' = \overline{F}$ is a separable closure of $F$, then we simply put $H^i(F, X) = H^i(\overline{F}/F, X)$.

Suppose $F$ is an algebraic extension of $\Q$ and $l$ is a prime number which splits into finitely many places in $F$.
Note that we have a natural identification $F \otimes \Q_l =\prod_{v \mid l} F_v$, where $v$ runs over the places of $F$ above $l$.
Here and henceforth, if $F/\Q$ is an infinite extension, $F_v$ denotes the union of the completions of number fields contained in $F$ at the places below $v$.
We put $E(F \otimes \Q_l) = \bigoplus_{v \mid l} E(F_v)$.
Similarly, for a continuous $\Gal(\overline{\Q}/F)$-module $X$, we put
$H^i(F \otimes \Q_l, X) = \bigoplus_{v \mid l} H^i(F_v, X)$.

\section{Definitions of Selmer groups and $p$-adic $L$-functions}\label{sec:401}

In this section, we review the definitions and properties of the various Selmer groups and $p$-adic $L$-functions.
In the $a_p = 0$ case, the basic ideas are due to Kobayashi \cite{Kob03} and Pollack \cite{Pol03}, respectively.
The definitions of the $\sharp/\flat$-objects generalizing the work by Sprung \cite{Spr12} are postponed to Definitions \ref{defn:930} and \ref{defn:73}, because those are comparatively complicated.

\subsection{Selmer groups}\label{subsec:402}

For an integer $n \geq -1$ or $n = \infty$, we define the $S$-imprimitive Selmer group of $E$ over $K_n$ by
\begin{equation}\label{eq:107}
\Sel_S(E/K_n) = 
\Ker \left(H^1(K_n, E[p^{\infty}]) 
\to \prod_{l \not\in S} \frac{H^1(K_n \otimes \Q_l, E[p^{\infty}])}{E(K_n \otimes \Q_l) \otimes (\Q_p/\Z_p)}  \right),
\end{equation}
where $l$ runs over all prime numbers not contained in $S$.
Here $E(K_n \otimes \Q_l) \otimes (\Q_p/\Z_p)$ is embedded in $H^1(K_n \otimes \Q_l, E[p^{\infty}])$ via the Kummer map.
Note that $E(K_{n} \otimes \Q_l) \otimes (\Q_p/\Z_p) = 0$ unless $l = p$.
We define the primitive Selmer group $\Sel(E/K_n)$ by the similar formula allowing $l$ runs over all prime numbers (namely, replacing $S$ by $\emptyset$).
We also define the fine Selmer group of $E$ over $K_n$ by
\begin{equation}
\Sel^0(E/K_n) = 
\Ker \left(H^1(K_n, E[p^{\infty}]) 
\to \prod_{l} H^1(K_n \otimes \Q_l, E[p^{\infty}]) \right),
\end{equation}
where $l$ runs over all prime numbers.

\begin{rem}
Though $\Sel(E/K_{\infty})$ is more fundamental than $\Sel_S(E/K_{\infty})$, the former does not behave well in our equivariant consideration.
In particular, Theorem \ref{thm:76} does not hold in general if we omit the subscript $S$ (though we may still make $S$ smaller as in Remark \ref{rem:871}).
We will compare them in Subsection \ref{subsec:821}.
\end{rem}

Recall that we defined $k_n = K_n \otimes \Q_p$ in Notations in Section \ref{sec:01}.
For integers $n \geq n' \geq -1$, we denote by $\Tr^n_{n'}$ the map $k_n \to k_{n'}$ induced by the trace map $K_n \to K_{n'}$.
By abuse of notation, $\Tr^n_{n'}$ will also denote various trace maps, such as $E(k_{n}) \to E(k_{n'})$.

\begin{defn}\label{defn:898}
When $a_p=0$, the $\pm$-Selmer groups are defined as follows (\cite[Definition 2.1]{Kob03}; more precisely, \cite[Definition 2.1]{KO18} contains our situation).
For $n \geq -1$, define
\begin{equation}\label{eq:105}
E^{\pm}(k_n) = \{ x \in E(k_n) \mid \Tr_{n'+1}^{n}(x) \in E(k_{n'}), -1 \leq \forall n' < n, (-1)^{n'} = \pm 1\}.
\end{equation}
Put $E^{\pm}(k_{\infty}) = \bigcup_n E^{\pm}(k_n)$.
Then, allowing $n = \infty$, we define the $\pm$-Selmer groups by
\begin{equation}\label{eq:109}
\Sel_S^{\pm}(E/K_n) = 
\Ker \left(\Sel_S(E/K_{n}) \to \frac{E(k_{n}) \otimes (\Q_p/\Z_p)}{E^{\pm}(k_{n}) \otimes (\Q_p/\Z_p)}  \right).
\end{equation}
Note that the same way as in \cite[Lemma 8.17]{Kob03} shows that the map $E^{\pm}(k_{n}) \otimes (\Q_p/\Z_p) \to E(k_{n}) \otimes (\Q_p/\Z_p)$ is injective.
We define $\Sel^{\pm}(E/K_n)$ similarly.
\end{defn}

Recall our convention that $\bullet = \emptyset$ (resp. $\bullet \in \{+, -\}$) if $p \nmid a_p$ (resp. $a_p = 0$).

\begin{prop}\label{prop:851}
For $\bullet \in \{\emptyset, +, -\}$, the Selmer group $\Sel_S^{\bullet}(E/K_{\infty})$ is $\Lambda$-cotorsion.
\end{prop}

\begin{proof}
When $K = \Q$ and $S = \emptyset$, the assertion is proved by Kato \cite[Theorem 17.4(1)]{Kat04} ($p \nmid a_p$ case) and Kobayashi \cite[Theorem 2.2]{Kob03} ($a_p = 0$ case).
The general case is also widely known to experts (e.g. \cite[Theorem 1.5]{Gree99} if $p \nmid a_p$, \cite[Remark 1.4(5)]{KO18} if $a_p = 0$).
In fact,  it follows from Proposition \ref{prop:945}, Theorem \ref{thm:77}, and Proposition \ref{prop:101} in this paper.
\end{proof}

\subsection{$p$-adic $L$-functions}\label{subsec:403}

We first recall the definition of the complex $L$-function associated to the elliptic curve $E$ \cite[C.16]{Sil09}.
For a prime number $l$, let $a_l$ be the integer for $E$ defined as usual (see Notations in Section \ref{sec:01}).
Recall that $N$ is the conductor of $E$.
We denote by $\mathbf{1}_N$ the trivial Dirichlet character modulo $N$, namely, for a prime number $l$ we have $\mathbf{1}_N(l) = 1$ if $l \nmid N$ and $\mathbf{1}_N(l) = 0$ otherwise.

Let $\Sigma$ be any finite set of prime numbers and $\psi$ a Dirichlet character.
As a convention, if $l$ does not divide the conductor of $\psi$, we have $\psi(l) = \psi(\sigma_l)$ for the $l$-th power Frobenius $\sigma_l$.
Then the $\Sigma$-imprimitive $L$-function is defined by
\begin{equation}\label{eq:122}
L_{\Sigma}(E, \psi, s) = \prod_{l \not\in \Sigma} (1-a_l \psi(l)l^{-s} + \mathbf{1}_N(l) \psi(l)^2 l^{1-2s})^{-1}
\end{equation}
for complex variable $s$.
This product converges for $s$ with real part greater than $3/2$ and, thanks to the modularity theorem, is known to possess an analytic continuation to the entire complex plane.
We put $L(E, \psi, s) = L_{\emptyset}(E, \psi, s)$.

Let $\Omega^+ \in \R_{>0}$ and $\Omega^- \in \sqrt{-1}\R_{>0}$ denote the real and imaginary N{\'{e}}ron periods of $E$ so that $\Z \Omega^+ \oplus \Z \Omega^-$ contains the N\'{e}ron lattice with index $1$ or $2$ (see e.g. \cite[(1.1)]{MT87}).
Then it is known that the complex number
\[
 \frac{L_{\Sigma}(E,\psi,1)}{\Omega^{\sign(\psi)}}
 \]
is an algebraic number, where $\sign(\psi) \in \{+, -\}$ denotes the sign of $\psi(-1)$.
Therefore we can regard this as an element of $\overline{\Q_p}$, using the fixed embeddings of $\overline{\Q}$ into $\overline{\Q_p}$ and $\C$.

We shall introduce the Gauss sum of a Dirichlet character and its imprimitive variants.
The imprimitive variants will enable us to simplify the interpolation properties (\eqref{eq:71}, \eqref{eq:126}, etc.) of the $p$-adic $L$-functions.
They were unnecessary in the previous works where $S = \emptyset$.

\begin{defn}\label{defn:897}
For a Dirichlet character $\psi$, let $p \nmid m_{\psi}$ and $n_{\psi} \geq -1$ be the integers such that $m_{\psi} p^{n_{\psi}+1}$ is the conductor of $\psi$.
\end{defn}

\begin{defn}\label{defn:119}
For a Dirichlet character $\psi$ of conductor $M = m_{\psi} p^{n_{\psi}+1}$, define the (primitive) Gauss sum of $\psi$ by
\[
\tau(\psi) = \sum_{\sigma \in \Gal(\Q(\mu_M)/\Q)} \psi(\sigma)\zeta_{M}^{\sigma}.
\]
Suppose the set $S$ of prime numbers $\neq p$ satisfies $\prim(m_{\psi}) \subset S$.
Take the minimum integer $m'$ such that $m_{\psi} \mid m'$ and $\prim(m') = S$, and put $M' = m'p^{n_{\psi}+1}$.
Then define the $S$-imprimitive Gauss sum of $\psi$ by
\[
\tau_S(\psi) = \sum_{\sigma \in \Gal(\Q(\mu_{M'})/\Q)} \psi(\sigma)\zeta_{M'}^{\sigma}.
\]
For example, we have $\tau_{S}(\psi) = \tau(\psi)$ if $S = \prim(m_{\psi})$.
\end{defn}

It is well-known and easily verified that 
\begin{equation}\label{eq:91}
\tau_{S'}(\psi) = \prod_{l \in S' \setminus S}(-\psi(\sigma_l)) \tau_S(\psi)
\end{equation}
for another finite set $S'$ of prime numbers $\neq p$ such that $S \subset S'$.

Next we introduce the spaces of functions where the classical $p$-adic $L$-functions live in.
Recall that we fixed a topological generator $\gamma$ of $\Gamma$ and can identify $\RR$ with $\Z_p[[\X]][G_0]$ by sending $\gamma$ to $1+\X$.
For a finite extension $F$ of $\Q_p$ and a real number $h > 0$, we put
\[
\HH_{h, F}(\Gamma) = \Set{ \sum_{n \geq 0} a_{n} \X^n \in F[[\X]] | a_{n} \in F,  \lim_{n \to \infty} \frac{| a_{n}|_p}{n^h} = 0 }.
\]
Here $|-|_p$ denotes the $p$-adic absolute value normalized as $|p|_p = 1/p$.
Note that each power series in $\HH_{h, F}(\Gamma)$ is convergent on the open unit disk $\{x \in \overline{\Q_p} \mid |x|_p < 1\}$.
Put $\HH_{h,F}(G_{\infty}) = \HH_{h,F}(\Gamma) \otimes_{F} F[G_0]$, which contains $\RR \otimes_{\Z_p} F$.
For each character $\psi$ of $G_{\infty}$ of finite order, by abuse of notation, we also denote by $\psi$ the induced map 
\[
\psi: \HH_{h, F}(G_{\infty}) \to \overline{\Q_p}
\]
 defined by $\sum_{n \geq 0} a_{n} \X^n \mapsto \sum_{n \geq 0} a_{n} (\psi(\gamma)-1)^n$ and $\sigma \mapsto \psi(\sigma)$ for $\sigma \in G_0$.
This map is an extension of the homomorphism $\RR \otimes_{\Z_p} F \to \overline{\Q_p}$ induced by $\psi$.

By the works of Amice-Velu \cite{AV75}, Vi{\v{s}}ik \cite{Vis76} or Mazur-Tate-Teitelbaum \cite{MTT86}, we have the classical $p$-adic $L$-functions as follows.

\begin{prop}
Let $\alpha \in \overline{\Q_p}$ be an allowable root of $t^2-a_pt+p = 0$, meaning $\ord_p(\alpha) <1$.
Then there is an element $\LL_S(E/K_{\infty}, \alpha) \in \HH_{1, \Q_p(\alpha)}(G_{\infty})$
 satisfying the following interpolation property:
For any character $\psi$ of $G_{\infty}$ of finite order, we have
\begin{equation}\label{eq:71}
\psi(\LL_S(E/K_{\infty}, \alpha)) = e_p(\alpha, \psi) \tau_S(\psi^{-1}) \frac{L_{S}(E,\psi,1)}{\Omega^{\sign(\psi)}},
\end{equation}
where we put
\[
e_p(\alpha, \psi)
= \begin{cases}
	\alpha^{-(1+n_{\psi})}  & (n_{\psi} \geq 0) \\
	(1-\alpha^{-1}\psi(p))(1-\alpha^{-1}\psi(p)^{-1}) & (n_{\psi} = -1).
\end{cases}
\]
\end{prop}

\begin{rem}\label{rem:798}
As soon as the property \eqref{eq:71} is formulated, the construction of the equivariant $p$-adic $L$-function follows easily from the works \cite{AV75}, \cite{Vis76}, \cite{MTT86}.
However, this interpolation property does not seem to appear explicitly in the literature.
It was determined so that Theorem \ref{thm:77} holds true.

After writing out the first version of this paper, the author was informed that our $p$-adic $L$-function may be related to the element $\xi_{K_{\infty}, S}$ in Kurihara \cite[p. 336]{Kur14}, where the ordinary case is treated.
The author thanks Masato Kurihara for giving this information.
Although the interpolation property of $\xi_{K_{\infty}, S}$ is not written down therein, it can be confirmed that they actually coincide up to a unit, in the situation of \cite{Kur14}.
\end{rem}

\begin{rem}\label{rem:812}
In many works, the $p$-adic $L$-functions of elliptic curves are characterized by the interpolation of $L_S(E, \psi^{-1}, 1)$ rather than $L_S(E, \psi, 1)$ (for example, \cite[Chap I, \S 14]{MTT86}, \cite[Theorem 16.2]{Kat04}, \cite[Theorem 3.1]{Kob03}).
Such a variance is not an essential problem, thanks to the functional equation (see Proposition \ref{prop:641}).
Our convention in this paper has the advantage of being suitable for varying $S$ in the main conjecture as in Proposition \ref{prop:615}.
We are consistent with Greenberg-Vatsal \cite[p. 55]{GV00} (and Kurihara \cite{Kur14} as in Remark \ref{rem:798}), which also concern varying $S$.
\end{rem}

As in Section \ref{sec:01}, if $p \nmid a_p$, we put $\LL_S(E/K_{\infty}) = \LL_S(E/K_{\infty}, \alpha) \in \RR \otimes \Q_p$, where $\alpha$ is the unit root of $t^2-a_pt+p$.
Note that $\alpha \in \Z_p^{\times}$ by the Hensel's lemma.

If $a_p = 0$, we have two $p$-adic $L$-functions $\LL_S(E/K_{\infty}, \alpha), \LL_S(E/K_{\infty}, - \alpha)$, which are not contained in $\RR \otimes \Q_p(\alpha)$.
The idea of Pollack \cite[\S 4.1]{Pol03} is to introduce the $\pm$-logarithm
\[
\log^{\pm} = \frac{1}{p} \prod_{n' \geq 1, (-1)^{n'} = \pm 1} \frac{\Phi_{n'}(1+\X)}{p} \in \HH_{1, \Q_p}(\Gamma),
\]
where
\begin{equation}\label{eq:100}
\Phi_n(1 + \X) = \frac{(1 + \X)^{p^n}-1}{(1 + \X)^{p^{n-1}}-1}
\end{equation}
is the $p^n$-th cyclotomic polynomial.
Moreover, we define elements of $\Lambda = \Z_p[[T]]$ as in \cite[(8.24)]{Kob03} by
\[
\ttilde{\omega}_n^{\pm} = \prod_{1 \leq n' \leq n, (-1)^{n'} = \pm1} \Phi_{n'}(1+\X), \quad \omega_n^{\pm} = T \ttilde{\omega}_n^{\pm}, \quad
\omega_n = (1+T)^{p^n}-1.
\]
Note that we have $\ttilde{\omega}_n^{\pm} \omega_n^{\mp} = \omega_n$.
Our sign convention will follow \cite{Kob03}, which is opposite to \cite{Pol03}.

\begin{prop}
If $a_p = 0$, there are $\LL_S^{\pm}(E/K_{\infty}) \in \RR \otimes \Q_p$ such that
\begin{equation}\label{eq:125}
\LL_S(E/K_{\infty}, \alpha) = \alpha \log^- \LL_S^+(E/K_{\infty}) + \log^+ \LL_S^-(E/K_{\infty})
\end{equation}
holds for any root $\alpha$ of $t^2+p$.
Moreover, they satisfy the interpolation property
\begin{equation}\label{eq:126}
\psi(\LL_S^{\pm}(E/K_{\infty})) =
 	\begin{cases}
		\psi(\ttilde{\omega}_{n_{\psi}}^{\mp})^{-1} (-1)^{[(n_{\psi}+2)/2]} \tau_S(\psi^{-1}) \frac{L_{S}(E,\psi,1)}{\Omega^{\sign(\psi)}} & (n_{\psi} \geq 0, (-1)^{n_{\psi}} = \pm 1)\\
		0 & (n_{\psi} = 0, \pm = -)\\
		(\psi(p)+\psi(p)^{-1}) \tau_S(\psi^{-1}) \frac{L_{S}(E,\psi,1)}{\Omega^{\sign(\psi)}} & (n_{\psi} = -1, \pm = +) \\
		(p-1) \tau_S(\psi^{-1}) \frac{L_{S}(E,\psi,1)}{\Omega^{\sign(\psi)}} & (n_{\psi} = -1, \pm = -) \\
	\end{cases}
\end{equation}
for any character $\psi$ of $G_{ \infty}$ of finite order.
\end{prop}

\begin{proof}
The same proof as in \cite[Theorem 5.1]{Pol03} shows the existences of $\LL_S^{\pm}(E/K_{\infty}) \in \RR \otimes \Q_p$ such that \eqref{eq:125} holds.
Then \eqref{eq:126} follows from applying $\psi$ to both sides of \eqref{eq:125}, together with an easy (but a bit lengthy) case-by-case computation.
When $K = \Q$ and $S = \emptyset$, these formulas are given in \cite[p. 7]{Kob03} (with a variance of the involution $\iota$ as in Remark \ref{rem:812}).
\end{proof}

\section{Construction of Local Points}\label{sec:02}

In this section, we construct a system of local points of $E$, namely points in $E(k_n)$.
The construction is motivated by the works of Kobayashi \cite{Kob03}, Sprung \cite{Spr12}, and Kitajima-Otsuki \cite{KO18}.
However, much more delicate argument is necessary than those works.

Let $\hat{E}$ be the formal group law over $\Z_p$ associated to a minimal Weierstrass model of $E$ and $\log_{\hat{E}}$ its logarithm.
As in Notations in Section \ref{sec:01}, let $\OO_n$ be the integer ring of $k_n = K_n \otimes \Q_p$, and let $\m_n$ be its Jacobson radical.
Then we have an exact sequence
\begin{equation}\label{eq:87}
0 \to \hat{E}(\m_n) \to E(k_n) \to \ttilde{E}(\OO_n/\m_n) \to 0
\end{equation}
by \cite[Propositions 2.1 and 2.2]{Sil09}.

Let $\varphi$ denote the $p$-th power Frobenius.
Therefore we have $\psi(\varphi) = \psi(p)$ for a Dirichlet character $\psi$ with $n_{\psi} = -1$.
Recall that we denote by $m_{\psi}p^{n_{\psi}+1}$ the conductor of $\psi$ (Definition \ref{defn:897}).

Our goal in this section is to prove the following two results.

\begin{thm}\label{thm:116}
Suppose $p \mid a_p$ holds.
Then there exists a unique system of elements $d_n \in \hat{E}(\m_n)$ ($n \geq -1$) satisfying the following.

(1) For $n \geq 0$, we have
\[
\Tr^{n}_{n-1}(\dd_{n}) = 
	\begin{cases}
		a_p \dd_{n-1} - \dd_{n-2} & (n \geq 1) \\
		(a_p - \varphi - \varphi^{-1}) \dd_{-1} & (n = 0).
	\end{cases}
\]

(2) For $n \geq 0$, $\hat{E}(\m_{n})$ is generated by $\dd_{n}, \dd_{n-1}$ as an $R_{n}$-module.
Moreover, $\hat{E}(\m_{-1})$ is generated by $\dd_{-1}$ as an $R_{-1}$-module.

(3) For a character $\psi$ of $G_{\infty}$ of finite order, we have
\[
\sum_{\sigma \in G_{n_{\psi}}} \sigma(\log_{\hat{E}}(\dd_{n_{\psi}})) \psi(\sigma)=
 	\begin{cases}
		\tau_S(\psi) & (n_{\psi} \geq 0)\\
		(1-p^{-1}a_p\psi(p)^{-1}+p^{-1}\psi(p)^{-2})^{-1} \tau_S(\psi) & (n_{\psi} = -1).
	\end{cases}
\]
\end{thm}

\begin{thm}\label{thm:121}
Suppose $p \nmid a_p$ and Assumption \ref{ass:04} below hold.
Let $\alpha, \beta \in \Z_p$ be the roots of $t^2-a_pt+p = 0$ such that $p \nmid \alpha, p \mid \beta$. 
Then there exists a unique system of elements $d_n \in \hat{E}(\m_n)$ $(n \geq -1)$ satisfying the following.

(1) For $n \geq 0$, we have
\[
\Tr^{n}_{n-1}(\dd_{n}) = 
	\begin{cases}
		\alpha \dd_{n-1} & (n \geq 1) \\
		(\alpha - \varphi^{-1}) \dd_{-1} & (n = 0).
	\end{cases}
\]

(2) For $n \geq -1$, $\hat{E}(\m_{n})$ is generated by $\dd_{n}$ as an $R_{n}$-module.

(3) For a character $\psi$ of $G_{\infty}$ of finite order, we have
\[
\sum_{\sigma \in G_{n_{\psi}}} \sigma(\log_{\hat{E}}(\dd_{n_{\psi}})) \psi(\sigma)=
 	\begin{cases}
		\tau_S(\psi) & (n_{\psi} \geq 0)\\
		(1-\beta^{-1}\psi(p)^{-1})^{-1} \tau_S(\psi) & (n_{\psi} = -1).
	\end{cases}
\]
\end{thm}

\begin{rem}
When $K = \Q$, $S = \emptyset$, and $a_p = 0$, the assertion of Theorem \ref{thm:116} is shown by Kobayashi \cite[Definition 8.8, Lemma 8.9, Proposition 8.11, Proposition 8.26]{Kob03}.
There are several further works where $p \mid a_p$, such as \cite{IP06}, \cite{Spr12}, and \cite{KO18}, but the proof of Theorem \ref{thm:116} still needs new ideas.
On the other hand, the author does not know a similar work on the ordinary case as in Theorem \ref{thm:121}.
The proof of Theorem \ref{thm:121} is quite similar to that of Theorem \ref{thm:116}, but a new subtle difficulty will appear in the proof of Proposition \ref{prop:44}.
\end{rem}

\begin{rem}\label{rem:967}
We will see (in Sections \ref{sec:32} and \ref{sec:220}) that we only need the properties (1)(2) in Theorems \ref{thm:116} and \ref{thm:121} in order to prove Theorems \ref{thm:76} and \ref{thm:93}.
This kind of remark is already given in \cite[Remark 3.12]{KO18}.
However, we need to impose the property (3) to establish the connection with the analytic side, stated in Theorem \ref{thm:77}.
In fact, in the $a_p = 0$ case, \cite{KO18} actually constructs a system $(d_n)_n$ satisfying the properties (1)(2).
But their construction is not canonical and, in particular, the property (3) is not necessarily satisfied.
\end{rem}

First we reduce the proofs of Theorems \ref{thm:116} and \ref{thm:121} to a special case.

For $p \nmid m$, we introduce the following notations (meaning that ``subscript $m$ gives corresponding objects for $K = \Q(\mu_m)$'').
For each integer $n \geq -1$ or $n = \infty$, put $K_{m, n} = \Q(\mu_{mp^{n+1}})$ and $G_{m, n} = \Gal(K_{m,n}/\Q)$.
Put $k_{m, n} = K_{m,n} \otimes_{\Q} \Q_p$, let $\OO_{m,n}$ be the integer ring of $k_{m,n}$, and let $\m_{m,n}$ be its Jacobson radical. 
Put $R_{m,n} = \Z_p[G_{m,n}]$ for each integer $n \geq -1$, and put $\RR_m = \Z_p[[G_{m,\infty}]]$.

\begin{lem}\label{lem:946}
Let $m$ be the conductor of $K$.
Then Theorems \ref{thm:116} and \ref{thm:121} follow from the assertions for the case where $K = \Q(\mu_m)$ and $S = \prim(m)$.
\end{lem}

\begin{proof}
First note that, in the $p \nmid a_p$ case, Assumption \ref{ass:04} is unchanged if we change $K$ to $\Q(\mu_m)$.
Since $\log_{\hat{E}}$ is injective by Proposition \ref{prop:98} and Assumption \ref{ass:04} below, the property (3) assures the uniqueness assertion.
To show the existence, let $d_{m,n} \in \hat{E}(\m_{m,n})$ be the claimed system of points for $K = \Q(\mu_m), S = \prim(m)$.
Define 
\[
d_n = \left(\prod_{l \in S, l \nmid m} (-\sigma_l) \right)\Tr_{K_{m,n}/K_n}(d_{m,n}) \in \hat{E}(\m_n),
\]
where $\Tr_{K_{m,n}/K_n}$ denotes the map $\hat{E}(\m_{m,n}) \to \hat{E}(\m_n)$ induced by the trace map $K_{m,n} \to K_n$.
Then the properties (1) -- (3) are preserved.
In fact, (1) is clear; (2) follows from the surjectivity of $\Tr_{K_{m,n}/K_n}: \hat{E}(\m_{m,n}) \to \hat{E}(\m_n)$; (3) follows from \eqref{eq:91}.
\end{proof}

Thus it is enough to prove Theorems \ref{thm:116} and \ref{thm:121} for $K = \Q(\mu_m), S = \prim(m)$.
Those will be done in Subsections \ref{subsec:78} and \ref{subsec:40}, respectively.

\subsection{Non-Anomalous Condition}

In this subsection, we discuss the non-anomalous condition in the ordinary case.
Before that, we recall the following important observation in the supersingular case.

\begin{prop}\label{prop:98}
Suppose $p \mid a_p$ holds.
Then $E(k_{\infty})$ is $p$-torsion free.
Equivalently, $\hat{E}(\m_{\infty})$ is $p$-torsion free.
\end{prop}

\begin{proof}
See \cite[Proposition 3.1]{KO18} and \cite[Proposition 8.7]{Kob03}.
\end{proof}

In the ordinary case, $E(k_{\infty})$ can contain a $p$-torsion in general (we call such a situation {\it anomalous}, due to Mazur).
Since almost all of the method in this paper cannot be applied to the anomalous case, we assume the following.

\begin{ass}\label{ass:04}
When $p \nmid a_p$, for the conductor $m$ of $K/\Q$, $E(k_{m, \infty})$ is $p$-torsion free.
\end{ass}

We mention that Assumption \ref{ass:04} is stronger than $E(k_{\infty})$ being $p$-torsion free, but is necessary since our method uses the corresponding result for $\Q(\mu_m)$; see the proof of Lemma \ref{lem:946}.
An equivalent condition to Assumption \ref{ass:04} will be given in Proposition \ref{prop:05}.
Note that, as shown in Lemma \ref{lem:58}, Assumption \ref{ass:04} is equivalent to that $\hat{E}(\m_{m,\infty})$ is $p$-torsion free.

\begin{lem}\label{lem:58}
Suppose $p \nmid a_p$ holds.
Let $F$ be a finite extension of $\Q_p(\mu_p)$ and $\kappa(F)$ its residue field.
Then we have $E(F)[p]=0$ if and only if $\ttilde{E}(\kappa(F))[p]=0$.
\end{lem}

\begin{proof}
Let $\rho$ (resp. $\rho'$) be the $2$-dimensional (resp. $1$-dimensional) representation of $\Gal(\overline{\Q_p}/F)$ over $\F_p$ defined by the action on $E[p]$ (resp. $\ttilde{E}[p]$).
We denote by $\rho''$ the kernel of the surjective homomorphism $\rho \to \rho'$.
By the Weil pairing, the representation $\rho' \otimes \rho'' \simeq \det(\rho)$ is trivial since $F$ contains $\mu_p$.

Therefore $\ttilde{E}(\kappa(F))[p] \neq 0$ is equivalent to that both $\rho'$ and $\rho''$ are trivial.
On the other hand, $E(F)[p] \neq 0$ is equivalent to the existence of trivial representation contained in $\rho$.
Now it is easy to see that those are equivalent.
\end{proof}

\begin{prop}\label{prop:05}
When $p \nmid a_p$, Assumption \ref{ass:04} holds if and only if $a_p^{\rd_m} \not\equiv 1 \mod p$, where $\rd_m$ is the residue degree of $K_{m,-1} = \Q(\mu_m)$ over $\Q$ at $p$.
\end{prop}

\begin{proof}
By Lemma \ref{lem:58}, Assumption \ref{ass:04} is equivalent to $\ttilde{E}(\OO_{m,-1}/(p))[p] = 0$, namely $\ttilde{E}(\F_{p^{\rd_m}})[p]=0$.
Let $\alpha, \beta$ be the roots of $t^2-a_pt+p$ such that $p \nmid \alpha$ and $p \mid \beta$.
Then $a_p = \alpha + \beta \equiv \alpha \mod p$ and \cite[Theorem 2.3.1]{Sil09} implies
\[
\sharp \ttilde{E}(\F_{p^{\rd_m}}) = (1+p^{\rd_m}) - (\alpha^{\rd_m} + \beta^{\rd_m}) \equiv 1 - \alpha^{\rd_m} \mod p.
\]
This completes the proof.
\end{proof}

In general, if $X$ is a free $\Z_p$-module of finite rank, we have a natural isomorphism
\begin{equation}\label{eq:95}
X^* \simeq (X \otimes (\Q_p/\Z_p))^{\dual}
\end{equation}
sending $f \in X^*$ to $x \otimes p^{-a} \mapsto p^{-a}f(x)$ (recall Notations in Section \ref{sec:01}).
If a group $G$ acts on $X$, then this isomorphism is $G$-equivariant.

\begin{prop}\label{prop:895}
Suppose Assumption \ref{ass:04} holds when $p \nmid a_p$.
Then for $n \geq -1$, $\hat{E}(\m_n)$ is isomorphic to the $p$-part of $E(k_n)$ under the injective map in \eqref{eq:87}.
In particular, by \eqref{eq:95}, we obtain isomorphisms
\begin{equation}
\hat{E}(\m_n)^* 
\simeq (\hat{E}(\m_n) \otimes (\Q_p/\Z_p))^{\dual} 
\simeq (E(k_n) \otimes (\Q_p/\Z_p))^{\dual}
\end{equation}
and, by taking the limit, 
\begin{equation}
\varprojlim_n \hat{E}(\m_n)^*
\simeq (\hat{E}(\m_{\infty}) \otimes (\Q_p/\Z_p))^{\dual} \simeq (E(k_{\infty}) \otimes (\Q_p/\Z_p))^{\dual}.
\end{equation}
\end{prop}

\begin{proof}
It is enough to show $\ttilde{E}(\OO_n/\m_n)[p] = 0$.
When $p \mid a_p$, this is clear.
When $p \nmid a_p$, this follows from Lemma \ref{lem:58}.
\end{proof}

\subsection{Preliminary Computations}\label{subsec:83}

This subsection is a preliminary to the proofs of Theorems \ref{thm:116} and \ref{thm:121}.

\begin{defn}\label{defn:06}
For $p \nmid m$ and $n \geq 0$, we put 
\[
\pi_{m,n} = \zeta_m^{\varphi^{-(n+1)}} (\zeta_{p^{n+1}} ^{\sigma_m^{-1}} - 1) \in \m_{m,n}.
\]
Here $\sigma_m$ denotes the $m$-th power Frobenius map and $\varphi$ the $p$-th power Frobenius map.
If $n \leq -1$, then we put $\pi_{m,n} = 0 \in \m_{m,-1}$.
\end{defn}

\begin{rem}\label{rem:852}
Our $\pi_{m,n}$ corresponds to $\zeta_{p^{n+1}}-1$ in \cite[Definition 8.8]{Kob03} and is motivated especially by the element $\pi_n = \zeta^{\varphi^{-(n+1)}}(\zeta_{p^{n+1}}-1)$ in \cite[p. 9]{KO18}.
Here, $\zeta$ is a generator of the group of roots of unity in a local field.
One of the novel ideas in this paper is to utilize the (global) elements $\zeta_m$ with various $m$ instead of the single (local) element $\zeta$.
That is because the Gauss sums in the right hand sides of Theorems \ref{thm:116} and \ref{thm:121} are defined using various $\zeta_{mp^{n+1}}$.

In order for this idea to work, we will have to modify several assertions in \cite{KO18} by using various $\zeta_m$.
Roughly speaking, assertions in \cite{KO18} on the single element $\zeta$ will be replaced by assertions on $\{\zeta_{m'} \mid m' \in \AAA(m)\}$.
Here $\AAA(m)$ is defined in Definition \ref{defn:18} below.
A typical example is Lemma \ref{lem:19} below, which corresponds to \cite[Lemma 3.9]{KO18}.
\end{rem}

For $m' \mid m$ with $p \nmid m$ and $-1 \leq n' \leq n$, we denote by $\Tr^{m,n}_{m',n'}$ the various maps induced by the trace map $K_{m,n} \to K_{m',n'}$. 
We study the behavior of $\pi_{m,n}$ under the trace maps.

\begin{lem}\label{lem:14}
(1) For $p \nmid m$ and $n \geq 0$, we have
\[
\Tr^{m,n}_{m,n-1}(\zeta_{p^{n+1}}-1) = -p
\]
and
\[
\Tr^{m,n}_{m,n-1}(\pi_{m,n}) = -p \zeta_m^{\varphi^{-(n+1)}}.
\]

(2) Let  $p \nmid m$ and let $l$ be a prime divisor of $m$. 
For $n \geq -1$ we have 
\[
\Tr^{m,n}_{m/l,n} (\zeta_m) = 
	\begin{cases}
		0 & (l^2 \mid m) \\
		(-\sigma_l^{-1}) \zeta_{m/l} & (l^2 \nmid m)
	\end{cases}
\]
and
\[
\Tr^{m,n}_{m/l,n} (\pi_{m,n}) = 
	\begin{cases}
		0 & (l^2 \mid m) \\
		(-\sigma_l^{-1}) \pi_{m/l,n} & (l^2 \nmid m).
	\end{cases}
\]
\end{lem}

\begin{proof}
(1) The minimal polynomial of $\zeta_{p^{n+1}}-1 \in K_{m,n}$ over $K_{m,n-1}$ is $(1+x)^p-\zeta_{p^n}$ (resp. $((1+x)^p - 1)/x$) if $n \geq 1$ (resp. $n = 0$).
The coefficient of $x^{p-1}$ (resp. $x^{p-2}$) is $p$, which implies the first equation.
The second equation follows immediately.

(2) Similarly, this assertion follows from the fact that the minimal polynomial of $\zeta_m \in K_{m,n}$ over $K_{m/l,n}$ is $x^l-\zeta_{m/l}$ (resp. $(x^l-\zeta_{m/l}) / (x-\zeta_{m/l}^{\sigma_l^{-1}})$) if $l^2 \mid m$ (resp. $l^2 \nmid m$).
\end{proof}

\begin{defn}\label{defn:18}
For $p \nmid m$, let $\AAA(m)$ be the set of divisors $m'$ of $m$ such that $\prim(m') = \prim(m)$.
For example, $\AAA(m) = \{m\}$ if and only if $m$ is square-free.
\end{defn}

\begin{lem}\label{lem:19}
Let  $p \nmid m$.
For $n \geq 0$, $\m_{m,n}/\m_{m,n-1}$ is generated by $\{ \pi_{m',n} \mid m' \in \AAA(m)\}$ as an $R_{m,n}$-module.
Moreover, $\m_{m,-1}$ is generated by $\{ p\zeta_{m'} \mid m' \in \AAA(m)\}$ as an $R_{m,-1}$-module.
\end{lem}

\begin{proof}
As mentioned in Remark \ref{rem:852}, this lemma corresponds to \cite[Lemma 3.9]{KO18}, where the module $\m_n/\m_{n-1}$ is generated by the single element $\pi_n$ in their notations.
The proof of our lemma is more delicate since we have to utilize various $m'$.

Firstly we show the second statement.
Since $\m_{m,-1} =p( \Z[\mu_m] \otimes_{\Z} \Z_p)$, it is enough to show that $\Z[\mu_m]$ is generated by $\{ \zeta_{m'} \mid m' \in \AAA(m)\}$ as a $\Z[\Gal(\Q(\mu_m)/\Q)]$-module.
Take any divisor $m''$ of $m$.
Let $m' \in \AAA(m)$ be the smallest element divisible by $m''$.
Then by Lemma \ref{lem:14}(2), we have
\[
\Tr^{m',n}_{m'',n}(\zeta_{m'}) = \left( \prod_{l \mid m', l \nmid m''} (-\sigma_{l}^{-1}) \right) \zeta_{m''}.
\]
This shows that $\zeta_{m''}$ is contained in the module generated by $\zeta_{m'}$.
This proves the claim.

Secondly we show the first statement for $m = 1$, namely, that $\m_{\Q_p(\mu_{p^{n+1}})}/\m_{\Q_p(\mu_{p^{n}})}$ is generated by $\zeta_{p^{n+1}}-1$ as a $\Z_p[\Gal(\Q(\mu_{p^{n+1}})/\Q)]$-module.
Here, $\m_{\Q_p(\mu_{p^{n+1}})}$ denotes the valuation ideal of $\Q_p(\mu_{p^{n+1}})$.
It is clear that $\m_{\Q_p(\mu_{p^{n+1}})}$ is generated by $\{\zeta_{p^{n+1}}^i-1 \mid 1 \leq i < p^{n+1}\}$ over $\Z_p$.
If $p \mid i$ then $\zeta_{p^{n+1}}^i-1 \in \m_{\Q_p(\mu_{p^n})}$ and if $p \nmid i$ then $\zeta_{p^{n+1}}^i-1$ is a Galois conjugate of $\zeta_{p^{n+1}}-1$.
This proves the claim.

Finally we show the first statement in general.
Observe that $\m_{m,n} \simeq \Z[\mu_m] \otimes_{\Z} \m_{\Q_p(\mu_{p^{n+1}})}$ under which $\pi_{m',n}$ corresponds to $\zeta_{m'}^{\varphi^{-(n+1)}} \otimes (\zeta_{p^{n+1}}^{\sigma_{m'}^{-1}}-1)$.
Moreover, $R_{m,n}$ acts on the right hand side through the decomposition $R_{m,n} = \Z[\Gal(\Q(\mu_m)/\Q)] \otimes_{\Z} \Z_p[\Gal(\Q(\mu_{p^{n+1}})/\Q)]$.
Hence the assertion follows from the first and the second claims above.
This completes the proof.
\end{proof}

We shall introduce a set of power series and a topology on it.
The topology is necessary because we will consider infinite sums in the proofs of Propositions \ref{prop:10} and \ref{prop:44}.
In the previous works such as \cite[\S 8.2]{Kob03}, the topology is often neglected, but we give a verification for completeness.

\begin{defn}\label{defn:698}
For $p \nmid m$, we put 
\[
\FUN_{m} = \left\{ f(\X) \in k_{m,-1}[[\X]] \mid f(0) = 0, f'(\X) \in \OO_{m,-1}[[\X]]\right\}.
\]
We have a bijection $\FUN_m \overset{\sim}{\to} \prod_{j \geq 1} \frac{1}{j}\OO_{m,-1}$ defined by $\sum_{j \geq 1} a_j T^j \mapsto (a_j)_j$.
We equip $\FUN_m$ with the topology such that this bijection is homeomorphic, where the target has the product topology of the $p$-adic topology.
\end{defn}

\begin{lem}\label{lem:699}
Let $p \nmid m$ and $n \geq 0$.
Suppose a sequence $\{f_{\nu}(T)\}_{\nu =1}^{\infty}$ in $\FUN_m$ converges to $f(T)$.
Then for $x \in \m_{m,n}$, the sequence $\{f_{\nu}(x)\}_{\nu}$ in $k_{m,n}$ converges to $f(x)$.
\end{lem}

\begin{proof}
Take an arbitrary positive integer $A$.
Since $x \in \m_{m,n}$, we have a positive integer $C$ such that $x^{j}/{j} \in p^A \OO_{m,n}$ holds for any integer $j > C$.
We denote the coefficients of $f_{\nu}$ and $f$ by $f_{\nu}(T) = \sum_{j=1}^{\infty} a_j^{(\nu)}T^j$ and $f(T) = \sum_{j=1}^{\infty} a_j T^j$.
Since $\lim_{\nu \to \infty} a_j^{(\nu)} = a_j$ for each $j$, for sufficiently large $\nu$, we have $\left (a_j^{(\nu)} - a_j \right)x^j \in p^A \OO_{m,-1}$ for any $1 \leq j \leq C$.
Therefore, for sufficiently large $\nu$, we can compute
\[
f_{\nu}(x) - f(x) = \sum_{j=1}^{\infty} \left(a_j^{(\nu)} - a_j \right)x^j \equiv \sum_{j=1}^{C} \left( a_j^{(\nu)} - a_j\right) x^j \equiv 0 \bmod p^A\OO_{m, -1}.
\]
Since $A$ was taken arbitrarily, this shows $\lim_{\nu \to \infty} f_{\nu}(x) = f(x)$.
\end{proof}

The following is a generalization of \cite[Proposition 8.11]{Kob03}.
Though that suffices for the supersingular case, Lemma \ref{lem:20} is necessary for the ordinary case.

\begin{lem}\label{lem:20}
Let $p \nmid m$ and take $f(\X) \in \FUN_m$.
For $n\geq 0$ and $x \in \m_{m,n}$, we have $f(x) \in \m_{m,n} + k_{m,n-1}$.
\end{lem}

\begin{proof}
Without loss of generality, we may suppose $f(\X) = \X^j/j$ for some $j \geq 1$.
Namely, it is enough to show that $x^j \in j\m_{m,n} + k_{m,n-1}$ for $x \in \m_{m,n}$.
By letting $v = \ord_p(j)$, it is enough to show that $x^{p^{v}} \in p^{v}\m_{m,n} + k_{m,n-1}$.
By induction, it is enough to show that $y \in p^{v-1}\m_{m,n} + \m_{m,n-1}$ implies $y^p \in p^v \m_{m,n} + \m_{m,n-1}$ for $v \geq 1$, which we shall show.

It is easy to see that we may assume $y \in p^{v-1}\m_{m,n}$.
Then the assertion is trivial if $v \geq 2$, so we assume $v = 1$.
Write $y = \sum_{i=1}^{p^{n+1}-1} a_i(\zeta_{p^{n+1}}^i-1)$ with $a_i \in \Z[\mu_m] \otimes \Z_p$.
Then 
\[
y^p \in \sum_i a_i^p(\zeta_{p^{n+1}}^i-1)^p + p\m_{m,n}.
\]
Moreover,
\begin{align}
\zeta_{p^n}^i - 1 
&= ((\zeta_{p^{n+1}}^i-1)+1)^p - 1\\
&= (\zeta_{p^{n+1}}^i-1)^p + p\sum_{s = 1}^{p-1} \frac{1}{p}\begin{pmatrix} p \\ s \end{pmatrix} (\zeta_{p^{n+1}}^i-1)^{s}
\end{align}
shows $(\zeta_{p^{n+1}}^i-1)^p \in p\m_{m,n} + \m_{m,n-1}$.
This completes the proof.
\end{proof}

\subsection{Supersingular Case}\label{subsec:78}

Suppose $p \mid a_p$ holds in this subsection.
For each $p \nmid m$, we shall prove Theorem \ref{thm:116} for $K = \Q(\mu_m)$ and $S = \prim(m)$.
The proof can be divided into two steps.
As the first step, by an appropriate use of Honda theory, we construct auxiliary elements $\ddd_{m,n} \in \hat{E}(\m_{m,n})$.
Second we modify $\ddd_{m,n}$ to define
\begin{equation}\label{eq:94}
\dd_{m,n} = \sum_{m' \in \AAA(m)} \frac{m'}{m} \ddd_{m',n}
\end{equation}
and check that the properties (1) -- (3) holds.
Note that the second step particularly accords with the idea explained in Remark \ref{rem:852}.

\subsubsection{First Step}

Fix $p \nmid m$.
We shall construct elements $\ddd_{m,n} \in \hat{E}(\m_{m,n})$ in Proposition \ref{prop:10}.
They give an extension of the elements $c_n$ in \cite[Definition 8.8]{Kob03}, $c_n$ in \cite[Theorem 2.2]{Spr12}, and $d_n$ in \cite[Definition 3.3]{KO18}.
More precisely, \cite{Kob03} (resp. \cite{Spr12}) deals with the case where $m = 1, a_p = 0$ (resp. $m =1, p \mid a_p$).
On the other hand, the argument in \cite{KO18} can deal with $a_p = 0$ and general $m$, though it is not written explicitly (\cite{KO18} only treats $\zeta$ instead of $\zeta_m$ as in Remark \ref{rem:852}).

\begin{defn}\label{defn:08}
For $n \geq -1$, since $\varphi^2-a_p \varphi + p$ acts on $\OO_{m,-1}$ isomorphically by Nakayama's lemma, 
let $\eta_{m,n} \in \m_{m,-1}$ be the unique element such that
\begin{equation}\label{eq:93}
\eta_{m,n}^{\varphi^2} - a_p \eta_{m,n}^{\varphi} + p \eta_{m,n} = p\zeta_m^{\varphi^{-(n+1)}}.
\end{equation}
Note that $\eta_{m,n} = \eta_{m,-1}^{\varphi^{-(n+1)}}$.
Our $\eta_{m,n}$ corresponds to $p / (p+1)$ in \cite{Kob03}, $p / (p+1-a_p)$ in \cite{Spr12}, and $\varepsilon_n$ in \cite{KO18}.
We do not have a direct expression of $\eta_{m,n}$ as in the previous works, but the above characterization suffices for computations.
\end{defn}

\begin{lem}\label{lem:114}
Let $l$ be a prime divisor of $m$. 
For $n \geq -1$, we have 
\[
\Tr^{m,n}_{m/l,n} (\eta_{m,n}) = 
	\begin{cases}
		0 & (l^2 \mid m) \\
		(-\sigma_l^{-1}) \eta_{m/l, n} & (l^2 \nmid m).
	\end{cases}
\]
\end{lem}

\begin{proof}
This follows from Lemma \ref{lem:14}(2).
\end{proof}

\begin{defn}\label{defn:09}
For $j \geq -1$, define $c_j \in \Q_p$ inductively by $c_{-1} = 0$, $c_0 = 1$, and $c_j = p^{-1} (a_pc_{j-1} - c_{j-2})$ for  $j \geq 1$.
Our $c_j$ is denoted by $x_k$ in \cite{Spr12}.
\end{defn}

Recall $\pi_{m,n}$ in Definition \ref{defn:06}.

\begin{prop}\label{prop:10}
For $n \geq -1$, there exists a unique element $\ddd_{m,n} \in \hat{E}(\m_{m,n})$ such that 
\[
\log_{\hat{E}}(\ddd_{m,n}) = \eta_{m,n} + \sum_{j \geq 0} c_j \pi_{m,n-j}.
\]
\end{prop}

\begin{proof}
The uniqueness follows from the injectivity of $\log_{\hat{E}}$.
The existence can be shown by combining the ideas of \cite{KO18} ($a_p = 0$ case) and \cite{Spr12} ($m=1$ case) as follows.

Recall Definition \ref{defn:698} and put
\[
f_m(\X) = \sum_{j \geq 0} c_j ((\X+\zeta_m)^{p^j} - (\zeta_m)^{p^j}) \in \FUN_m.
\]
Namely, $f_m(\X)$ is the element of $\FUN_m$ characterized by 
\[
f_m'(\X) = \sum_{j \geq 0} p^jc_j (\X+\zeta_m)^{p^j-1}.
\]
The convergence follows from $p^jc_j \to 0$, since $c_j \in p^{-[j/2]}\Z_p$.
We also see that $f_m'(0)$ is a $p$-adic unit.

For $n \geq -1$, we have
\[
(f_m^{\varphi^{-(n+1)}})^{\varphi^2}(\X^{p^2}) - a_p (f_m^{\varphi^{-(n+1)}})^{\varphi}(\X^{p}) + p(f_m^{\varphi^{-(n+1)}})(\X) \equiv 0 \mod \m_{m,-1}[[\X]].
\]
In fact, we can assume $n = -1$ and then, by \cite[Lemma 4]{Hon68}, the left hand side is congruent to
\begin{align}
&f_m^{\varphi^2}((\X+\zeta_m)^{p^2} - (\zeta_m)^{p^2}) - a_p f_m^{\varphi}((\X+\zeta_m)^{p} - (\zeta_m)^{p}) + pf_m(\X)\\
&= \sum_{j \geq 0} c_j ((\X+\zeta_m)^{p^{j+2}} - (\zeta_m)^{p^{j+2}}) -a_p\sum_{j \geq 0} c_j ((\X+\zeta_m)^{p^{j+1}} - (\zeta_m)^{p^{j+1}}) + p\sum_{j \geq 0} c_j ((\X+\zeta_m)^{p^{j}} - (\zeta_m)^{p^{j}})\\
&= \sum_{j \geq -1} (c_{j}-a_pc_{j+1} + pc_{j+2})((\X+\zeta_m)^{p^{j+2}} - (\zeta_m)^{p^{j+2}}) + p\X\\
&=p\X.
\end{align}
Thus $f_m^{\varphi^{-(n+1)}}$ has Honda type $t^2-a_pt+p$.
Therefore by Honda theory \cite{Hon70}, there is a unique formal group $\GG_{m,n}$ over $\OO_{m,-1}$ with logarithm $\log_{\GG_{m,n}} = f_m^{\varphi^{-(n+1)}}$.
Moreover, $\exp_{\hat{E}} \circ \log_{\GG_{m,n}}: \GG_{m,n} \to \hat{E}$ is an isomorphism over $\OO_{m,-1}$.

Since $\log_{\GG_{m,n}}: \GG_{m,n}(\m_{m,-1}) \to \m_{m,-1}$ is an isomorphism, we can define $\varepsilon_{m,n} \in \GG_{m,n}(\m_{m,-1})$ as the unique element such that $\log_{\GG_{m,n}}(\varepsilon_{m,n}) = \eta_{m,n}$.
We shall show that 
\[
\ddd_{m,n} = \exp_{\hat{E}} \circ \log_{\GG_{m,n}}(\varepsilon_{m,n} [+]_{\GG_{m,n}} \pi_{m,n}) \in \hat{E}(\m_{m,n})
\]
is the element we want.
In fact, we have
\[
(\pi_{m,n}+\zeta_m^{\varphi^{-(n+1)}})^{p^j} - (\zeta_m^{\varphi^{-(n+1)}})^{p^j} = (\zeta_m^{\varphi^{-(n+1)}})^{p^j} ((\zeta_{p^{n+1}}^{\sigma_m^{-1}})^{p^j}-1) = \pi_{m,n-j}.
\]
Then Lemma \ref{lem:699} implies the final equality of
\[
\log_{\hat{E}}(\ddd_{m,n})= \log_{\GG_{m,n}}(\varepsilon_{m,n} [+]_{\GG_{m,n}} \pi_{m,n})
= \eta_{m,n} + f_m^{\varphi^{-(n+1)}}(\pi_{m,n})
= \eta_{m,n} + \sum_{j \geq 0} c_j\pi_{m,n-j}.
\]
This completes the proof.
\end{proof}

The element $\ddd_{m,n}$ satisfies the following compatibilities.
The assertion (1) is an extension of \cite[Lemma 8.9]{Kob03}, \cite[Theorem 2.2]{Spr12}, and \cite[Proposition 3.4]{KO18}.

\begin{prop}\label{prop:15}
(1) For $n \geq 0$, we have
\[
\Tr^{m,n}_{m,n-1}(\ddd_{m,n}) = 
	\begin{cases}
		a_p \ddd_{m,n-1} - \ddd_{m,n-2} & (n \geq 1) \\
		(a_p - \varphi - \varphi^{-1}) \ddd_{m,-1} & (n = 0).
	\end{cases}
\]

(2) Let $l$ be a prime divisor of $m$. 
For $n \geq -1$ we have 
\[
\Tr^{m,n}_{m/l,n} (\ddd_{m,n}) = 
	\begin{cases}
		0 & (l^2 \mid m) \\
		(-\sigma_l^{-1}) \ddd_{m/l,n} & (l^2 \nmid m).
	\end{cases}
\]
\end{prop}

\begin{proof}
We may apply the injective homomorphism $\log_{\hat{E}}$ to confirm these equations.

(1) When $n \geq 1$, we shall compute
\begin{align}
\Tr^{m,n}_{m,n-1} (\log_{\hat{E}}(\ddd_{m,n}))
 &= \Tr^{m,n}_{m,n-1} (\eta_{m,n}) + \sum_{j \geq 0} c_j\Tr^{m,n}_{m,n-1} (\pi_{m,n-j})\\
 &= p\eta_{m,n} -p\zeta_m^{\varphi^{-(n+1)}} + \sum_{j \geq 1}c_j p \pi_{m,n-j}\\
 &= (a_p \eta_{m,n}^{\varphi} - \eta_{m,n}^{\varphi^2}) + \sum_{j \geq 1} (a_p c_{j-1}-c_{j-2}) \pi_{m,n-j}\\
 &= a_p\left( \eta_{m,n-1} + \sum_{j \geq 0}c_j \pi_{m,n-1-j} \right) - \left(\eta_{m,n-2} + \sum_{j \geq 0}c_j \pi_{m,n-2-j}\right)\\
 &= a_p \log_{\hat{E}}(\ddd_{m,n-1}) - \log_{\hat{E}}(\ddd_{m,n-2}).
\end{align}
Here, the second equality follows from Lemma \ref{lem:14}(1);
 the third follows from Definitions \ref{defn:08} and \ref{defn:09};
  and the fourth follows from the formula in Definition \ref{defn:09}.
Similarly, when $n = 0$, we have
\begin{align}
\Tr^{m,0}_{m,-1} (\log_{\hat{E}}(\ddd_{m,0}))
 &= \Tr^{m,0}_{m,-1} (\eta_{m,0}) + \Tr^{m,0}_{m,-1} (\pi_{m,0})\\
 &= (p-1) \eta_{m,0} -p\zeta_m^{\varphi^{-1}}\\
 &= (a_p \eta_{m,0}^{\varphi} - \eta_{m,0}^{\varphi^2}) - \eta_{m,0}\\
 &= (a_p - \varphi - \varphi^{-1}) \eta_{m,-1}\\
 &= (a_p - \varphi - \varphi^{-1}) \log_{\hat{E}}(\ddd_{m,-1}).
\end{align}

(2) 
The assertion follows from Lemmas \ref{lem:14}(2) and \ref{lem:114} immediately.
\end{proof}

The following is an extension of \cite[Proposition 8.11]{Kob03}.

\begin{prop}\label{prop:21}
For $n \geq 0$, $\hat{E}(\m_{m,n})/\hat{E}(\m_{m,n-1})$ is generated by $\{\ddd_{m',n} \mid m' \in \AAA(m)\}$ over $R_{m,n}$.
Moreover, $\hat{E}(\m_{m,-1})$ is generated by $\{ \ddd_{m',-1} \mid m' \in \AAA(m)\}$ over $R_{m,-1}$.
\end{prop}

\begin{proof}
When $n = -1$, it is known that $\log_{\hat{E}}: \hat{E}(\m_{m,-1}) \to \m_{m,-1}$ is an isomorphism.
Since it sends $\ddd_{m', -1}$ to $\eta_{m', -1}$, the assertion follows from Lemma \ref{lem:19}.

When $n \geq 0$, we use the argument in \cite[Proposition 8.11]{Kob03}.
Consider the injective homomorphism $\log_{\hat{E}}: \hat{E}(\m_{m,n}) \to \m_{m,n}+k_{m,n-1}$ (Lemma \ref{lem:20}).
We obtain the following composite map
\[
\hat{E}(\m_{m,n})/\hat{E}(\m_{m,n-1}) \overset{\log_{\hat{E}}}{\hookrightarrow} (\m_{m,n}+k_{m,n-1}) / k_{m,n-1} \overset{\sim}{\leftarrow} \m_{m,n} / \m_{m,n-1}
\]
where the second map is the isomorphism induced by the inclusion map.
Then the class of $\ddd_{m',n}$ is sent to the class of $\pi_{m',n}$.
Therefore by Lemma \ref{lem:19}, the above injective map has to be an isomorphism and $\hat{E}(\m_{m,n}) / \hat{E}(\m_{m,n-1})$ is generated by $\{ \ddd_{m',n} \mid m' \in \AAA(m)\}$ as an $R_{m,n}$-module.
This completes the proof.
\end{proof}

\subsubsection{Second Step}

Fix $p \nmid m$, and define $\dd_{m,n} \in \hat{E}(\m_{m,n})$ by the formula \eqref{eq:94}.
We shall show that this system of elements satisfies the conditions in Theorem \ref{thm:116}.

\begin{prop}\label{prop:25}
(1) For $n \geq 0$, we have
\[
\Tr^{m,n}_{m,n-1}(\dd_{m,n}) = 
	\begin{cases}
		a_p \dd_{m,n-1} - \dd_{m,n-2} & (n \geq 1) \\
		(a_p - \varphi - \varphi^{-1}) \dd_{m,-1} & (n = 0).
	\end{cases}
\]

(2) Let $l$ be a prime divisor of $m$. 
For $n \geq -1$, we have 
\[
\Tr^{m,n}_{m/l,n} (\dd_{m,n}) = 
	\begin{cases}
		\dd_{m/l,n} & (l^2 \mid m) \\
		(-\sigma_l^{-1}) \dd_{m/l,n} & (l^2 \nmid m).
	\end{cases}
\]
\end{prop}

\begin{proof}
(1) Immediate from Proposition \ref{prop:15}(1).

(2) First we compute $\Tr^{m,n}_{m/l,n} (\ddd_{m',n})$ for each $m' \in \AAA(m)$.
If $m'$ divides $m/l$, then $l^2 \mid m$ and 
\[
\Tr^{m,n}_{m/l,n} (\ddd_{m',n}) = [K_{m,n}:K_{m/l,n}] \ddd_{m',n} = l \ddd_{m',n}.
\]
If $m'$ does not divide $m/l$, then
\[
\Tr^{m,n}_{m/l,n} (\ddd_{m',n}) = \Tr^{m',n}_{m'/l,n} (\ddd_{m',n}) = 
	\begin{cases}
		0 & (l^2 \mid m') \\
		(-\sigma_l^{-1}) \ddd_{m'/l,n} & (l^2 \nmid m')
	\end{cases}
\]
by Proposition \ref{prop:15}(2).
Using these formulas, if $l^2 \mid m$ then
\begin{align}
\Tr^{m,n}_{m/l,n} (\dd_{m,n}) 
&= \sum_{m' \in \AAA(m)} \frac{m'}{m} \Tr^{m,n}_{m/l,n} (\ddd_{m',n})
= \sum_{m' \in \AAA(m), m' | \frac{m}{l}} \frac{m'}{m} l \ddd_{m',n}\\
&= \sum_{m' \in \AAA(m/l)} \frac{m'}{m/l} \ddd_{m',n} = \dd_{m/l,n}.
\end{align}
If $l^2 \nmid m$ then
\begin{align}
\Tr^{m,n}_{m/l,n} (\dd_{m,n}) 
&= \sum_{m' \in \AAA(m)} \frac{m'}{m} \Tr^{m,n}_{m/l,n} (\ddd_{m',n})
= \sum_{m' \in \AAA(m)} \frac{m'}{m} (-\sigma_l^{-1}) \ddd_{m'/l,n}\\
&= (-\sigma_l^{-1})\sum_{m'' \in \AAA(m/l)} \frac{m''}{m/l} \ddd_{m'',n} = (-\sigma_l^{-1})\dd_{m/l,n}.
\end{align}
This completes the proof.
\end{proof}

\begin{prop}\label{prop:26}
For $n \geq 0$, $\hat{E}(\m_{m,n})$ is generated by $\dd_{m,n}, \dd_{m,n-1}$ as an $R_{m,n}$-module.
Moreover, $\hat{E}(\m_{m,-1})$ is generated by $\dd_{m,-1}$ as an $R_{m,-1}$-module.
\end{prop}

\begin{proof}
By Proposition \ref{prop:21}, it is enough to show that, for $n \geq -1$,
\[
\sum_{m' \in \AAA(m)} (\ddd_{m',n})_{R_{m',n}} = (\dd_{m,n})_{R_{m,n}}
\]
as submodules of $\hat{E}(\m_{m,n})$.
The inclusion $\supset$ is clear from the definition \eqref{eq:94}.
For any $m' \in \AAA(m)$, by Proposition \ref{prop:25}(2),
\[
\Tr_{m',n}^{m,n}(\dd_{m,n}) -\ddd_{m',n} = \dd_{m',n} - \ddd_{m',n} = \sum_{m'' \in \AAA(m'), m'' \neq m'} \frac{m''}{m'} \ddd_{m'',n}.
\]
Using induction on $m'$, we obtain $\ddd_{m',n} \in (\dd_{m,n})_{R_{m,n}}$.
This completes the proof.
\end{proof}

\begin{prop}\label{prop:38}
For a character $\psi$ of $G_{m,\infty}$ of finite order, we have
\[
\sum_{\sigma \in G_{m,n_{\psi}}} \sigma(\log_{\hat{E}}(\dd_{m,n_{\psi}})) \psi(\sigma)=
 	\begin{cases}
		\tau_{\prim(m)}(\psi) & (n_{\psi} \geq 0)\\
		(1-p^{-1}a_p\psi(p)^{-1}+p^{-1}\psi(p)^{-2})^{-1} \tau_{\prim(m)}(\psi) & (n_{\psi} = -1).
	\end{cases}
\]
\end{prop}

\begin{proof}
Take $m'$ as in Definition \ref{defn:119}, which depends on $\psi$ and $S = \prim(m)$.
Then we have
\[
\sum_{\sigma \in G_{m,n_{\psi}}} \sigma(\log_{\hat{E}}(\dd_{m,n_{\psi}})) \psi(\sigma)
= \sum_{\sigma \in G_{m',n_{\psi}}} \sigma(\log_{\hat{E}}(\dd_{m',n_{\psi}})) \psi(\sigma)
= \sum_{\sigma \in G_{m',n_{\psi}}} \sigma(\log_{\hat{E}}(\ddd_{m',n_{\psi}})) \psi(\sigma),
\]
where the first equality follows from Proposition \ref{prop:25}(2), and the second follows from \eqref{eq:94}.

If $n_{\psi} \geq 0$, then we can compute
\[
\sum_{\sigma \in G_{m',n_{\psi}}} \sigma(\log_{\hat{E}}(\ddd_{m',n_{\psi}})) \psi(\sigma)
 = \sum_{\sigma \in G_{m',n_{\psi}}} \sigma(\pi_{m',n_{\psi}}) \psi(\sigma)
 = \sum_{\sigma \in G_{m',n_{\psi}}} \sigma(\zeta_{m'p^{n_{\psi}+1}}) \psi(\sigma)
 =\tau_{\prim(m)}(\psi).
\]
Here, in the second equality, we used the fact $\zeta_m^{\varphi^{-(n+1)}} \zeta_{p^{n+1}} ^{\sigma_m^{-1}} = \zeta_{mp^{n+1}}$ in general, which can be shown by taking $p^{n+1}$-th power and $m$-th power of both sides.

If $n_{\psi} = -1$, the definition of $\eta_{m', -1}$ implies
\[
(\psi(p)^{-2}-a_p\psi(p)^{-1}+p)\sum_{\sigma \in G_{m',-1}} \sigma(\eta_{m',-1}) \psi(\sigma) = p\tau_{\prim(m)}(\psi).
\]
Thus we obtain
\[
\sum_{\sigma \in G_{m',-1}} \sigma(\log_{\hat{E}}(\ddd_{m',-1})) \psi(\sigma)
 = \sum_{\sigma \in G_{m',-1}} \sigma(\eta_{m',-1}) \psi(\sigma)
 = (1-p^{-1}a_p\psi(p)^{-1}+p^{-1}\psi(p)^{-2})^{-1} \tau_{\prim(m)}(\psi).
\]
This completes the proof.
\end{proof}

Now Theorem \ref{thm:116} for $K = \Q(\mu_m), S = \prim(m)$ is true by Propositions \ref{prop:25}(1), \ref{prop:26}, and \ref{prop:38}.
By Lemma \ref{lem:946}, this completes the proof of Theorem \ref{thm:116}.

\subsection{Ordinary Case}\label{subsec:40}

We shall prove Theorem \ref{thm:121} in a similar way as in Subsection \ref{subsec:78}.
Suppose $p \nmid a_p$ holds and let $\alpha, \beta \in \Z_p$ be the roots of $t^2-a_pt+p = 0$ such that $p \nmid \alpha$ and $p \mid \beta$.
Fix $p \nmid m$ such that Assumption \ref{ass:04} holds for $K = \Q(\mu_m)$.

\begin{lem}\label{lem:100}
$\varphi-\alpha^{-1}$ and $\varphi-\alpha$ are unit elements of $R_{m,-1}$.
\end{lem}

\begin{proof}
By Nakayama's lemma, it is enough to show that they are unit elements after projection to $R_{m,-1}/(p) = \F_p[\Gal(\Q(\mu_m)/\Q)]$.
As in Proposition \ref{prop:05}, let $\rd_m$ be the residue degree of $K_{m, -1}$ over $\Q$ at $p$.
We have $\alpha^{\rd_m} \not\equiv 1 \mod p$ as shown in Proposition \ref{prop:05}.
Then the assertion follows, since $\varphi^{\rd_m} = 1$.
\end{proof}

\begin{defn}\label{defn:99}
As in Definition \ref{defn:08}, for $n \geq -1$, let $\eta_{m,n}$ be the unique element of $\m_{m, -1}$ such that
\[
\eta_{m,n}^{\varphi}-\beta \eta_{m,n} = -\beta \zeta_{m}^{\varphi^{-(n+1)}}.
\]
\end{defn}

As an analogue of Proposition \ref{prop:10}, we show the following.

\begin{prop}\label{prop:44}
For $n \geq -1$, there exists a unique element $\ddd_{m,n} \in \hat{E}(\m_{m, n})$ such that
\[
\log_{\hat{E}}(\ddd_{m,n}) = \eta_{m,n} + \sum_{j \geq 0} \beta^{-j} \pi_{m,n-j}.
\]
\end{prop}

\begin{proof}
The uniqueness follows from Assumption \ref{ass:04}.

To prove the existence, we wish to define $f_{m}(\X) \in \FUN_m$ by
\[
\sum_{j \geq 0} \beta^{-j} ((\X+\zeta_{m})^{p^j}-(\zeta_{m})^{p^j}).
\]
However, this infinite sum does not converge since the coefficient of $\X$ in each term, $p^j\beta^{-j} \zeta_m^{p^j-1} = \alpha^j \zeta_m^{p^j-1}$, is a $p$-adic unit.
To avoid this trouble, we modify this formula as follows.

By Lemma \ref{lem:100}, there exists a unique element $u_{m} \in \OO_{m, -1}^{\times}$ satisfying $u_{m}^{\varphi} - \alpha^{-1} u_{m} = -\alpha^{-1} \zeta_{m}$.
Put
\[
f_{m}(\X) = \sum_{j \geq 0} \beta^{-j} \left[ (\X+\zeta_{m})^{p^j}-(\zeta_{m})^{p^j} -p^j(\zeta_{m})^{p^j}\log(1+\zeta_{m}^{-1}\X)\right] + u_{m} \log(1+\zeta_{m}^{-1}\X) \in \FUN_m.
\]
This is well-defined since
\begin{align}
f_{m}'(\X) &= \sum_{j \geq 0} \beta^{-j} \left[ p^j(\X+\zeta_{m})^{p^j-1} -p^j(\zeta_{m})^{p^j}(\X+\zeta_{m})^{-1}\right] + u_{m} (\X+\zeta_{m})^{-1}\\
&= (\X+\zeta_{m})^{-1} \left[ \sum_{j \geq 0} \alpha^j \left[(\X+\zeta_{m})^{p^j}-(\zeta_{m})^{p^j}\right] +u_{m} \right]
\end{align}
certainly converges in $\OO_{{m},-1}[[\X]]$.
We also have $f_{m}'(0) = \zeta_{m}^{-1} u_{m} \in \OO_{{m},-1}^{\times}$.
We remark that a formal computation shows that the (divergent) formal sum $\sum_{j \geq 0} \alpha^j \zeta_{m}^{p^j}$ satisfies the defining equation of $u_m$.
This is the motivation to introduce $u_m$. 

\hidden{
In fact, a formal computation yields
\[
\left(\sum_{j \geq 0} \beta^{-j}p^j(\zeta_{m})^{p^j} \right)^{\varphi} - \alpha^{-1} \left(\sum_{j \geq 0} \beta^{-j}p^j(\zeta_{m})^{p^j} \right)
= \sum_{j \geq 0} \alpha^j(\zeta_{m})^{p^{j+1}} - \sum_{j \geq 0} \alpha^{j-1}(\zeta_{m})^{p^j} = - \alpha^{-1} \zeta_{m}.
\]
}

As in the supersingular case, we can compute
\[
(f_{m}^{\varphi^{-(n+1)}})^{\varphi}(\X^p) - \beta f_{m}^{\varphi^{-(n+1)}}(\X) \equiv 0 \mod \m_{m,-1}[[\X]]
\]
for each integer $n$.
This means that $f_{m}^{\varphi^{-(n+1)}}$ has Honda type $t-\beta$.\hidden{
\begin{proof}
We can assume $n = -1$.
Then the left hand side is congruent to
\begin{align}
&f_m^{\varphi}((\X+\zeta_m)^{p} - (\zeta_m)^{p}) -\beta f_m(\X)\\
&= \left\{\sum_{j \geq 0} \beta^{-j} \left[(\X+\zeta_m)^{p^{j+1}} - (\zeta_m)^{p^{j+1}} -p^j(\zeta_m)^{p^{j+1}}\log((1+\zeta_m^{-1}\X)^p) \right] + u_m^{\varphi}\log((1+\zeta_m^{-1}\X)^p)\right\}\\
& \quad -\beta \left\{\sum_{j \geq 0} \beta^{-j} \left[ (\X+\zeta_m)^{p^j}-(\zeta_m)^{p^j} -p^j(\zeta_m)^{p^j}\log(1+\zeta_m^{-1}\X)\right] + u_m \log(1+\zeta_m^{-1}\X)\right\}\\
&= -\beta (\X - \zeta_m\log(1+\zeta_m^{-1}\X)) + (pu_m^{\varphi}-\beta u_m)\log(1+\zeta_m^{-1}\X)\\
&= -\beta \X
\end{align}
where the last equality follows from the choice of $u_m$.
\end{proof}
}
By Honda theory, there is a unique formal group $\GG_{{m},n}$ over $\OO_{{m},-1}$ such that $\log_{\GG_{m,n}} =  f_{m}^{\varphi^{-(n+1)}}$.
Moreover, $\exp_{\hat{E}} \circ \log_{\GG_{{m},n}}: \GG_{{m},n} \to \hat{E}$ is an isomorphism over $\OO_{m,-1}$.
Let $\varepsilon_{{m},n} \in \GG_{{m},n}(\m_{m,-1})$ be the unique element such that $\log_{\GG_{m,n}}(\varepsilon_{m,n}) = \eta_{{m},n}$.
Then, as in the supersingular case,
\[
\ddd_{{m},n} = \exp_{\hat{E}} \circ \log_{\GG_{{m},n}}(\varepsilon_{{m},n} [+]_{\GG_{{m},n}} \pi_{{m},n}) \in \hat{E}(\m_{{m},n})
\]
satisfies the desired property.\hidden{
\begin{proof}
\[
\log_{\hat{E}}(\ddd_{m,n})= \log_{\GG_{m,n}}(\varepsilon_{m,n} [+]_{\GG_{m,n}} \pi_{m,n})
=\eta_{m,n} + f_m^{\varphi^{-(n+1)}}(\pi_{m,n}).
\]
Since 
\[
(\pi_{m,n}+\zeta_m^{\varphi^{-(n+1)}})^{p^j} - (\zeta_m^{\varphi^{-(n+1)}})^{p^j} = (\zeta_m^{\varphi^{-(n+1)}})^{p^j} ((\zeta_{p^{n+1}}^{\sigma_m^{-1}})^{p^j}-1) = \pi_{m,n-j}
\]
and
\[
\log(1+(\zeta_m^{\varphi^{-(n+1)}})^{-1}\pi_{m,n}) = \log(\zeta_{p^{n+1}}^{\sigma_m^{-1}}) = 0,
\]
we obtain
\[
f_m^{\varphi^{-(n+1)}}(\pi_{m,n}) = \sum_{j \geq 0} \beta^{-j}\pi_{m,n-j}.
\]
\end{proof}
}
This completes the proof of Proposition \ref{prop:44}.
\end{proof}

The rest of the argument in Subsection \ref{subsec:78} is valid without serious changes (we omit the detail).
Consequently, defining $d_{m,n}$ by the formula \eqref{eq:94}, we obtain the following variants of  Propositions \ref{prop:25}, \ref{prop:26}, \ref{prop:38}.

\begin{prop}\label{prop:46}
(1) For $n \geq 0$, we have
\[
\Tr^{m,n}_{m,n-1}(\dd_{m,n}) = 
	\begin{cases}
		\alpha \dd_{m,n-1} & (n \geq 1) \\
		(\alpha - \varphi^{-1}) \dd_{m,-1} & (n = 0).
	\end{cases}
\]

(2) Let $l$ be a prime divisor of $m$. 
For $n \geq -1$, we have 
\[
\Tr^{m,n}_{m/l,n} (\dd_{m,n}) = 
	\begin{cases}
		\dd_{m/l,n} & (l^2 \mid m) \\
		(-\sigma_l^{-1}) \dd_{m/l,n} & (l^2 \nmid m).
	\end{cases}
\]
\end{prop}

\begin{prop}\label{prop:47}
For $n \geq -1$, $\hat{E}(\m_{m,n})$ is generated by $\dd_{m,n}$ over $R_{m,n}$.
\end{prop}

\begin{proof}
The same proof as in Proposition \ref{prop:26} shows $\hat{E}(\m_{m,n}) = \hat{E}(\m_{m,n-1}) + (\dd_{m,n})_{R_{m,n}}$ for $n \geq 0$ and $\hat{E}(\m_{m,-1}) = (\dd_{m,-1})_{R_{m,-1}}$.
For $n \geq 0$, we have $d_{m, n-1} \in (d_{m,n})_{R_{m,n}}$ by Proposition \ref{prop:46}(1) and Lemma \ref{lem:100}.
Therefore
\[
\hat{E}(\m_{m,n}) = (\dd_{m,n}, \dd_{m,n-1}, \dots, \dd_{m,-1})_{R_{m,n}} 
= (\dd_{m,n})_{R_{m,n}}.
\]
This completes the proof.
\end{proof}

\begin{prop}\label{prop:48}
For a character $\psi$ of $G_{m,\infty}$ of finite order, we have
\[
\sum_{\sigma \in G_{m,n_{\psi}}} \sigma(\log_{\hat{E}}(\dd_{m,n_{\psi}}) \psi(\sigma)=
 	\begin{cases}
		\tau_{\prim(m)}(\psi) & (n_{\psi} \geq 0)\\
		(1-\beta^{-1}\psi(p)^{-1})^{-1} \tau_{\prim(m)}(\psi) & (n_{\psi} = -1)
	\end{cases}
\]
\end{prop}

\hidden{
\begin{proof}
Take $m'$ as in Definition \ref{defn:119}.
By Proposition \ref{prop:46}(2), we have
\[
\sum_{\sigma \in G_{m,n_{\psi}}} \sigma(\log_{\hat{E}}(\dd_{m,n_{\psi}})) \psi(\sigma)
= \sum_{\sigma \in G_{m',n_{\psi}}} \sigma(\log_{\hat{E}}(\dd_{m',n_{\psi}})) \psi(\sigma)
= \sum_{\sigma \in G_{m',n_{\psi}}} \sigma(\log_{\hat{E}}(\ddd_{m',n_{\psi}})) \psi(\sigma).
\]

If $n_{\psi} \geq 0$, then
\[
\sum_{\sigma \in G_{m',n_{\psi}}} \sigma(\log_{\hat{E}}(\ddd_{m',n_{\psi}})) \psi(\sigma)
 = \sum_{\sigma \in G_{m',n_{\psi}}} \sigma(\pi_{m',n_{\psi}}) \psi(\sigma)
 = \sum_{\sigma \in G_{m',n_{\psi}}} \sigma(\zeta_{m'p^{n_{\psi}+1}}) \psi(\sigma)
 =\tau_{\prim(m)}(\psi).
\]
If $n_{\psi} = -1$, since the definition of $\eta_{m', -1}$ implies
\[
(\psi(p)^{-1}-\beta) \sum_{\sigma \in G_{m',-1}} \sigma(\eta_{m',-1}) \psi(\sigma) = -\beta \tau_{\prim(m)}(\psi),
\]
we have
\[
\sum_{\sigma \in G_{m',-1}} \sigma(\log_{\hat{E}}(\ddd_{m',-1})) \psi(\sigma)
 = \sum_{\sigma \in G_{m',-1}} \sigma(\eta_{m',-1}) \psi(\sigma)
 = (1-\beta^{-1}\psi(p)^{-1})^{-1} \tau_{\prim(m)}(\psi).
\]
This completes the proof.
\end{proof}
}

Now Propositions \ref{prop:46}(1), \ref{prop:47}, and \ref{prop:48} complete the proof of Theorem \ref{thm:121}, by Lemma \ref{lem:946}.

\section{Construction of Coleman Maps}\label{sec:32}

In this section, we construct Coleman maps and prove Theorem \ref{thm:93}.
The construction uses the systems of elements in Theorems \ref{thm:116} and \ref{thm:121}, similarly as in \cite[\S 8.5]{Kob03} and \cite[\S 5]{Spr12}, which treat $K = \Q$, $S = \emptyset$, and $p \mid a_p$.
However, our discussion will be apparently different from \cite{Kob03} and \cite{Spr12}.
That is because, while those works define the Coleman maps as maps from $H^1(k_n, T_pE)$, we will firstly construct the Coleman maps as maps from $\hat{E}(\m_n)^*$ (later in Definition \ref{defn:854} we will compose them with the natural map $H^1(k_n, T_pE) \to \hat{E}(\m_n)^*$).
Our treatment will enable us to determine the precise structure of $\hat{E}(\m_n)^*$, and consequently to prove Theorem \ref{thm:93}.
Note also that, as mentioned in Remark \ref{rem:967}, we only need the properties (1) and (2) of Theorems \ref{thm:116} and \ref{thm:121} in this section.

\begin{rem}\label{rem:691}
We mention here that our results will reprove and refine several previous works (especially in the supersingular case).

\begin{enumerate}
\item For example, Theorem \ref{thm:93}(3) in the case where $K = \Q$ is shown in \cite[Theorem 6.2]{Kob03}.
In fact, in that case, the assertion of Theorem \ref{thm:93}(3) is equivalent to that the $\pm$-Coleman map is an isomorphism
\[
(E^{\pm}(K_{\infty} \otimes \Q_p) \otimes (\Q_p/\Z_p))^{\dual} \overset{\sim}{\to} \au^{\pm}.
\]
Here $\au^{\pm}$ is the ideal of $\RR$ defined in Section \ref{sec:01}.
However, the proof of \cite{Kob03} cannot be directly extended to general $K$.
That is because, as in Remark \ref{rem:879}, the $(+)$-Coleman map does not give an isomorphism in general.

\item Similarly, Theorem \ref{thm:93}(2) in the case where $K = \Q$ is shown in \cite[Propositions 7.3 and 7.6]{Spr12}, but the proof cannot be directly generalized.

\item Another work is \cite[Theorem 1.8]{KO18}, which determines the abstract structure of $(E^{\pm}(k_{\infty}) \otimes (\Q_p/\Z_p))^{\dual}$ as a $\Lambda[\Delta]$-module (not as an $\RR$-module).
Our Theorem \ref{thm:93}(3) refines that result since we give the $\RR$-module structure and moreover give an explicit exact sequence which determines the module structure (the proof in \cite{KO18} relies on the structure theorem of $\Lambda$-modules).
We note here that several other results in \cite{KO18} will be similarly refined in this paper, for example, Lemma \ref{lem:962}(2) below refines \cite[Proposition 1.6]{KO18} in our case.
\end{enumerate}
\end{rem}

We prepare some general notations.
For a finite abelian group $G$, we have a natural isomorphism 
\begin{equation}\label{eq:96}
\Z_p[G]^* \simeq \Z_p[G]
\end{equation}
 as left $\Z_p[G]$-modules, given by $f \mapsto \sum_{g \in G} f(g)g$.

For an $r \times r$ matrix $A \in M_r(\Z_p[G])$, we denote by $\times A: \Z_p[G]^{\oplus r} \to \Z_p[G]^{\oplus r}$ the homomorphism given by the right multiplication by $A$.
Here and henceforth, we always treat vectors as row vectors.
We denote the kernel and the cokernel of the map $\times A$ by $(\Z_p[G]^{\oplus r})[A]$ and $(\Z_p[G]^{\oplus r})/A$, respectively.
Recall that $\iota$ denotes the involution of a group ring.

\begin{lem}\label{lem:693}
Let $G$ be a finite abelian group.

(1) For $a \in \Z_p[G]$ and $f \in \Z_p[G]^*$, we have
\[
\sum_{g \in G} f(ga)g = \left(\sum_{g \in G} f(g)g\right) a^{\iota}.
\]

(2) For $A \in M_r(\Z_p[G])$, the following diagram is commutative.
\[
\xymatrix{
	(\Z_p[G]^{\oplus r})^{*} \ar[r]^{(\times A)^*} \ar[d]_{\vsim}&
	(\Z_p[G]^{\oplus r})^* \ar[d]^{\vsim} \\
	\Z_p[G]^{\oplus r} \ar[r]_{\times (A^{\iota})^T} &
	\Z_p[G]^{\oplus r}
}
\]
Here, the vertical isomorphisms are obtained by \eqref{eq:96} and the superscript $T$ denotes the transpose.
\end{lem}

\begin{proof}
(1) When $a = h$ for some element $h \in G$, the assertion is easy.
Then by $\Z_p$-linearity on $a$, the whole assertion follows.

(2) The assertion for $r = 1$ is nothing but the assertion (1), and the general case also follows from (1).
\end{proof}

Let $X$ be a $\Z_p[G]$-module and $x_1, \dots, x_r$ elements of $X$.
We denote by $\Phi_{x_1, \dots, x_r}: \Z_p[G]^{\oplus r} \to X$ the $\Z_p[G]$-homomorphism given by $\Phi_{x_1, \dots, x_r} \left((a_i)_i \right) = \sum_{i=1}^r a_ix_i$.
We also denote by $\Psi_{x_1, \dots, x_r}: X^* \to \Z_p[G]^{\oplus r}$ the $\Z_p[G]$-homomorphism given by $\Psi_{x_1, \dots, x_r}(f) = \left(\sum_{g \in G} f(gx_i)g \right)_i$.
Then $\Psi_{x_1, \dots, x_r}$ is the $\Z_p$-linear dual of $\Phi_{x_1, \dots, x_r}$ under the identification \eqref{eq:96}.

\subsection{Ordinary Case}\label{subsec:49}

In this subsection, we prove Theorem \ref{thm:93}(1).
Suppose $p \nmid a_p$ and Assumption \ref{ass:04} hold.

\begin{defn}
Using $d_n$ in Theorem \ref{thm:121}, put $\dd_{n}' = \alpha^{-(n+1)}\dd_{n}$ for $n \geq -1$.
Define the Coleman map $\Cole_{n}: \hat{E}(\m_{n})^* \to R_{n}$ by 
\begin{equation}\label{eq:120}
\Cole_{n}(f) = \sum_{\sigma\in G_{n}} f(\sigma(\dd_{n}'))\sigma.
\end{equation}
Thus we have $\Cole_n = \Psi_{d_n'}$.
\end{defn}

\begin{lem}\label{lem:947}
The diagram
\[
\xymatrix{
	\hat{E}(\m_n)^* \ar[r]^-{\Cole_n} \ar@{->>}[d] &
	R_n \ar@{->>}[d] \\
	\hat{E}(\m_{n-1})^* \ar[r]^-{\Cole_{n-1}} &
	R_{n-1}\\
}
\]
is commutative for $n \geq 1$, where the vertical arrows are the natural maps.
\end{lem}

\begin{proof}
Theorem \ref{thm:121}(1) implies $\Tr^n_{n-1}(d_n') = d_{n-1}'$ for $n \geq 1$.
For $\sigma \in G_n$, we denote the projection of $\sigma$ by $\overline{\sigma} \in G_{n-1}$.
Then for $f \in \hat{E}(\m_n)^*$, by \eqref{eq:120}, the projection of $\Cole_n(f) \in R_n$ to $R_{n-1}$ is
\[
\sum_{\sigma \in G_{n}} f(\sigma(\dd_{n}')) \overline{\sigma} 
= \sum_{\tau \in G_{n-1}} \sum_{\sigma \in G_{n}, \overline{\sigma} = \tau} f(\sigma(\dd_{n}')) \tau 
= \sum_{\tau \in G_{n-1}} f( \Tr^n_{n-1}(\dd_{n}')) \tau 
= \sum_{\tau \in G_{n-1}} f( \dd_{n-1}') \tau 
= \Cole_{n-1}(f).
\]
This completes the proof.

\end{proof}

\begin{defn}\label{defn:877}
Define $\Cole: \varprojlim_n \hat{E}(\m_{n})^* \to \RR$ as the limit of $\Cole_n$, which is possible by Lemma \ref{lem:947}.
\end{defn}

\begin{lem}
For $n \geq -1$, the map $\Cole_n: \hat{E}(\m_n)^* \to R_n$ is an isomorphism.
Therefore $\Cole: \varprojlim_n \hat{E}(\m_{n})^* \to \RR$ is also an isomorphism.
\end{lem}

\begin{proof}
By Theorem \ref{thm:121}(2), $\Phi_{d_n'}: R_n \to \hat{E}(\m_n)$ is an isomorphism.
Thus the $\Z_p$-linear dual $\Cole_n = \Psi_{d_n'}$ is also an isomorphism. 
\end{proof}

Using the natural identification in Proposition \ref{prop:895}, we thus obtain Theorem \ref{thm:93}(1).
We restate it here.

\begin{thm}\label{thm:973}
Suppose $p \nmid a_p$ and Assumption \ref{ass:04} hold.
Then $\Cole$ gives an isomorphism 
\[
(E(k_{\infty}) \otimes (\Q_p/\Z_p))^{\dual} \overset{\sim}{\to} \RR.
\]
\end{thm}

\subsection{Supersingular Case}\label{subsec:221}

In this subsection, we suppose $p \mid a_p$ holds and prove Theorem \ref{thm:93}(2).
Letting $K = \Q$ and $S = \emptyset$ in our discussion will recover the results by Sprung \cite{Spr12}.
The particular case where $a_p = 0$, which will similarly recover the results by Kobayashi \cite{Kob03}, will be studied more closely in the next subsection.
The basic idea in this subsection is the same as in the ordinary case, but there are more complications.
Note also that, as remarked in the beginning of this section, our discussion appears different from \cite{Kob03} and \cite{Spr12}, because we stick to $\hat{E}(\m_n)^*$ instead of $H^1(k_n, T_pE)$.

Recall that we fixed a topological generator $\gamma \in \Gamma$, by which $\Lambda = \Z_p[[\Gamma]] \simeq \Z_p[[\X]]$.
For $n \geq 1$, put 
\[
N_{n} = 1+\gamma^{p^{n-1}} + \gamma^{2p^{n-1}} + \dots + \gamma^{(p-1)p^{n-1}} = \frac{\gamma^{p^{n}}-1}{\gamma^{p^{n-1}}-1} \in \Lambda,
\]
which should be regarded as a lift of the norm element of $\Z_p[\Gamma^{p^{n-1}}/\Gamma^{p^{n}}]$.
Then $N_n$ is identified with the cyclotomic polynomial $\Phi_n(1+\X)$ defined in \eqref{eq:100}.

\begin{defn}\label{defn:688}
Put $A_n = \begin{pmatrix} 0 & -N_n \\ 1 & a_p \end{pmatrix} \in M_2(\Lambda)$ as in \cite[Proposition 3.3]{Spr12}.
In general, we denote by $\ttilde{A}$ the adjugate matrix of a square matrix $A$, so that $\ttilde{A_n} = \begin{pmatrix} a_p & N_n \\ -1 & 0 \end{pmatrix}$.
Put $B_n = A_n \dots A_1 \in M_2(\Lambda)$.
\end{defn}

We have
\begin{equation}\label{eq:151}
(\gamma-1)B_n\ttilde{B_n} = (\gamma-1)N_1 \dots N_n = \gamma^{p^n}-1.
\end{equation}

Recall that we always treat vectors as row vectors and that $R_n^{\oplus 2}[(\gamma-1)B_n]$ and $R_n^{\oplus 2}/(\gamma-1)B_n$ denote the kernel and the cokernel of $\times (\gamma -1)B_n: R_n^{\oplus 2} \to R_n^{\oplus 2}$.

\begin{lem}\label{lem:53}
(1) The module $R_{n}^{\oplus 2} / (\gamma-1)B_n$ is a free $\Z_p$-module of rank $[K_0: \Q](p^n+1)$.

(2) The map $\times \ttilde{B_n}$ induces an isomorphism 
\[
R_{n}^{\oplus 2}/(\gamma-1)B_n \overset{\sim}{\to} R_{n}^{\oplus 2}[(\gamma-1)B_n].
\]
\end{lem}

\begin{proof}
(1) We can define $\Lambda^{\oplus 2}/(\gamma-1)B_n$ and $\RR^{\oplus 2}/(\gamma-1)B_n$ similarly.
By \eqref{eq:151}, we have $\det((\gamma-1)B_n) = (\gamma-1)(\gamma^{p^n}-1) \in \Lambda$, which is a distinguished polynomial of degree $p^n+1$.
Thus $\Lambda^{\oplus 2}/(\gamma-1)B_n$ is a free $\Z_p$-module of rank $p^n+1$.
Since $\RR$ is a free $\Lambda$-module of rank $[K_0:\Q]$, we deduce that $\RR^{\oplus 2}/(\gamma-1)B_n$ is a free $\Z_p$-module of rank $[K_0:\Q](p^n+1)$.
By \eqref{eq:151}, the natural map $\RR^{\oplus 2}/(\gamma-1)B_n \to R_{n}^{\oplus 2}/(\gamma-1)B_n$ is an isomorphism.
These prove (1).

(2) This is an abstract algebraic assertion behind \cite[Proposition 3.3]{Spr12}.
We show the surjectivity.
Take any $(\overline{x}, \overline{y}) \in R_{n}^{\oplus 2}$ such that $(\overline{x}, \overline{y})(\gamma-1)B_n =0$.
 Then for a lift $(x, y) \in \RR^{\oplus 2}$ of $(\overline{x}, \overline{y})$, there is $(x', y') \in \RR^{\oplus 2}$ such that $(x,y)(\gamma-1)B_n = (x',y')(\gamma^{p^n}-1)$.
 Applying $ \times \ttilde{B_n}$ and using \eqref{eq:151}, we obtain $(x,y) = (x',y')\ttilde{B_n}$.
 The injectivity can be shown similarly, or alternatively follows from the surjectivity.
\end{proof}

Now we construct the Coleman map.
Let $d_n$ be as in Theorem \ref{thm:116}.
Then Theorem \ref{thm:116}(1) implies
\begin{equation}\label{eq:06}
A_n^T \begin{pmatrix} \dd_{n} \\ \dd_{n-1} \end{pmatrix} = \begin{pmatrix} \dd_{n-1} \\ \dd_{n-2}
\end{pmatrix}
\end{equation}
for $n \geq 1$, where the superscript $T$ denotes the transpose.
It also follows that 
\begin{equation}\label{eq:159}
(\gamma-1)  B_n^T \begin{pmatrix} \dd_{n} \\ \dd_{n-1} \end{pmatrix} = (\gamma-1)\begin{pmatrix} \dd_{0} \\ \dd_{-1} \end{pmatrix} = 0.
\end{equation}
Therefore the homomorphism $\Phi_{d_{n}, d_{n-1}}: R_{n}^{\oplus 2}/(\gamma-1)B_n^T \to \hat{E}(\m_{n})$ is well-defined.

Consider the exact sequence
\[
R_n^{\oplus 2} \overset{\times (\gamma - 1)B_n^T}{\to} R_n^{\oplus 2} \to R_n^{\oplus 2} / (\gamma - 1)B_n^T \to 0.
\]
By Lemma \ref{lem:53}(1), taking the $\Z_p$-linear dual gives the upper exact row of the following diagram
\[
\xymatrix@!C=70pt{
	0  \ar[r]&
	(R_n^{\oplus 2} / (\gamma - 1)B_n^T)^* \ar[r]&
	(R_n^{\oplus 2})^{*} \ar[r]^{(\times (\gamma - 1)B_n^T)^*} \ar[d]_{\vsim}&
	(R_n^{\oplus 2})^* \ar[d]^{\vsim} \\
	&&
	R_n^{\oplus 2} \ar[r]_{\times ((\gamma - 1)B_n)^{\iota}} &
	R_n^{\oplus 2}.
}
\]
By Lemma \ref{lem:693}(2), this diagram is commutative when the vertical arrows are defined by \eqref{eq:96}.
Therefore we have an isomorphism
\[
(R_n^{\oplus 2} / (\gamma - 1)B_n^T)^* \overset{\sim}{\to} R_n^{\oplus 2}[((\gamma - 1)B_n)^{\iota}].
\]

Hence we can construct the composite map
\begin{equation}\label{eq:161}
\hat{E}(\m_{n})^* 
\overset{(\Phi_{d_{n},d_{n-1}})^*}{\to} (R_{n}^{\oplus 2}/(\gamma-1)B_n^T)^* 
\overset{\sim}{\to} R_{n}^{\oplus 2}[((\gamma-1)B_n)^{\iota}] 
\overset{\sim}{\leftarrow} R_{n}^{\oplus 2}/((\gamma-1)B_n)^{\iota}
\end{equation}
where the final arrow is $\times \ttilde{B_n}^{\iota}$ as in Lemma \ref{lem:53}(2).
The composite of the first two maps in \eqref{eq:161} coincides with $\Psi_{d_n, d_{n-1}}$.
Thus we obtain the following.

\begin{defn}\label{defn:55}
Define $\Cole'_{n}: \hat{E}(\m_{n})^* \to R_{n}^{\oplus 2}/((\gamma-1)B_n)^{\iota}$ as the composite map \eqref{eq:161}.
Therefore, $\Cole'_n$ is characterized by 
\begin{equation}\label{eq:16}
\Cole'_{n}(f) \ttilde{B_n}^{\iota} = \left( \sum_{\sigma \in G_{n}} f(\sigma(d_{n}))\sigma, \sum_{\sigma \in G_{n}} f(\sigma(d_{n-1}))\sigma \right)
\end{equation}
for $f \in \hat{E}(\m_{n})^*$.
This is a generalization of \cite[Proposition 5.3]{Spr12}.
See Remark \ref{rem:81} for the reason of the prime symbol.
\end{defn}

\begin{lem}\label{lem:998}
The diagram
\[
\xymatrix{
	\hat{E}(\m_n)^* \ar[r]^-{\Cole_n'} \ar@{->>}[d] &
	R_n^{\oplus 2} / ((\gamma - 1) B_n)^{\iota} \ar@{->>}[d] \\
	\hat{E}(\m_{n-1})^* \ar[r]^-{\Cole_{n-1}'} &
	R_{n-1}^{\oplus 2} / ((\gamma - 1) B_{n-1})^{\iota}
}
\]
is commutative for $n \geq 1$, where the vertical arrows are the natural maps.
\end{lem}

\begin{proof}
This is a generalization of \cite[Corollary 5.6]{Spr12}.
We can directly compute with the expression \eqref{eq:16}, but here we give an alternative proof depending on the more conceptual definition \eqref{eq:161}.
By \eqref{eq:06}, the following diagram is commutative.
\[
\xymatrix@C=50pt{
	R_{n-1}^{\oplus 2} / (\gamma - 1) B_{n-1}^T \ar[r]^-{\Phi_{d_{n-1}, d_{n-2}}} \ar@{^{(}->}[d]_-{\times A_n^T} &
	\hat{E}(\m_{n-1})  \ar@{^{(}->}[d] \\
	R_n^{\oplus 2} / (\gamma - 1) B_n^T  \ar[r]^-{\Phi_{d_{n}, d_{n-1}}} &
	\hat{E}(\m_n)
}
\]
Thus the left square of the following diagram is commutative.
\[\small
 \xymatrix@C=44pt{
    \hat{E}(\m_{n})^* \ar[r]^-{(\Phi_{d_{n},d_{n-1}})^*}\ar@{->>}[dd] &
     (R_{n}^{\oplus 2}/(\gamma-1)B_n^T)^* \ar[r]^-{\sim}\ar[d]_-{(\times A_n^T)^*} & 
     R_{n}^{\oplus 2}[((\gamma-1)B_n)^{\iota}] \ar[d]_{\times A_n^{\iota}} & 
     R_{n}^{\oplus 2}/((\gamma-1)B_n)^{\iota} \ar[l]_-{\sim}^-{\times \ttilde{B_n}^{\iota}} \ar@{->>}[dd] \\
    &
     (R_{n}^{\oplus 2}/(\gamma-1)B_{n-1}^T)^* \ar[r]^-{\sim} &
      R_{n}^{\oplus 2}[((\gamma-1)B_{n-1})^{\iota}] &&\\
    \hat{E}(\m_{n-1})^* \ar[r]_-{(\Phi_{d_{n-1},d_{n-2}})^*} & 
    (R_{n-1}^{\oplus 2}/(\gamma-1)B_{n-1}^T)^* \ar[r]^-{\sim} \ar[u]_-{\vsim} &
    R_{n-1}^{\oplus 2}[((\gamma-1)B_{n-1})^{\iota}] \ar[u]_-{\vsim}^-{\times N_n^{\iota}} &
     R_{n-1}^{\oplus 2}/((\gamma-1)B_{n-1})^{\iota} \ar[l]_-{\sim}^-{\times \ttilde{B_{n-1}}^{\iota}}
   }
\]
Moreover, the other squares in this diagram are also commutative (the middle upper one is by Lemma \ref{lem:693}(2)).
This completes the proof.
\end{proof}

\begin{defn}\label{defn:400}
We define
\[
\Cole': \varprojlim_{n} \hat{E}(\m_{n})^* \to \varprojlim_{n} R_{n}^{\oplus 2}/((\gamma-1)B_n)^{\iota} \simeq \RR^{\oplus 2}
\]
as the limit of $\Cole'_{n}$, which is possible by Lemma \ref{lem:998} (see \cite[Proposition 5.7]{Spr12} for the last isomorphism).
Moreover, define the Coleman map  $\Cole = (\Cole^{\sharp}, \Cole^{\flat}): \varprojlim_n \hat{E}(\m_{n})^* \to \RR^{\oplus 2}$ by
\begin{equation}\label{eq:10}
\Cole(f) = \Cole'(f) \begin{pmatrix} -a_p & -1 \\ 1 & 0 \end{pmatrix}.
\end{equation}
This is a generalization of \cite[Definition 5.9]{Spr12}.
\end{defn}

\begin{rem}\label{rem:81}
The modification from $\Cole'$ to $\Cole$ is necessary to be consistent with the work of Sprung \cite[Definition 3.8]{Spr12}.
However, the author thinks that it is also possible to deal with $\Cole'$ itself throughout this paper.
\end{rem}

Next we study the precise structure of $\hat{E}(\m_{n})$.
The following discussion, which reproves and refines the previous works as in Remark \ref{rem:691}, is one of the novel parts in this paper.

For $n \geq 0$, we define a sequence
\begin{equation}\label{eq:152}
0 \to R_{0} \oplus R_{-1} \overset{\aalpha_n}{\to} R_{n}^{\oplus 2}/(\gamma-1)B_n^T \oplus R_{-1} \overset{\bbeta_n}{\to} \hat{E}(\m_{n}) \to 0
\end{equation}
as follows (the exactness will be claimed in Proposition \ref{prop:57} below).
First $\bbeta_n$ is induced by $\Phi_{\dd_{n}, \dd_{n-1}, \dd_{-1}}$, which is well-defined by \eqref{eq:159}.
We define $\aalpha_0: R_{0} \oplus R_{-1} \to R_{0}^{\oplus 2} \oplus R_{-1}$ by the matrix 
\begin{equation}\label{eq:09}
\begin{pmatrix} 0 & 1 & -1 \\ N_{\Delta} & 0 & \varphi+\varphi^{-1} - a_p \end{pmatrix}.
\end{equation}
Here recall that $\Delta = \Gal(\Q(\mu_p)/\Q)$ and $N_{\Delta}$ is the norm element of $\Z_p[\Delta]$.
Then $\aalpha_n$ is defined inductively by the following commutative diagram
\begin{align}\label{eq:07}
\xymatrix{
	0 \ar[r] &
	R_{0} \oplus R_{-1} \ar[r]^-{\aalpha_{n-1}} \ar@{=}[d] &
	R_{n-1}^{\oplus 2}/(\gamma-1)B_{n-1}^T \oplus R_{-1} \ar[r]^-{\bbeta_{n-1}} \ar@{^{(}->}[d]^{\times A_n^T \oplus \id} &
	\hat{E}(\m_{n-1}) \ar[r] \ar@{^{(}->}[d] &
	0 \\
	0 \ar[r] &
	R_{0} \oplus R_{-1} \ar[r]^-{\aalpha_{n}} &
	R_{n}^{\oplus 2}/(\gamma-1)B_{n}^T \oplus R_{-1} \ar^-{\bbeta_{n}}[r] &
	\hat{E}(\m_{n}) \ar[r] &
	0 
}
\end{align}
for $n \geq 1$, where the right square is commutative by \eqref{eq:06}.

\begin{prop}\label{prop:57}
The sequence \eqref{eq:152} is exact.
\end{prop}

To prove this proposition, we use the following simple lemma.

\begin{lem}\label{lem:61}
Let $0 \to X' \overset{\aalpha}{\to} X \overset{\bbeta}{\to} X'' \to 0$ be a sequence (not known to be exact) of free $\Z_p$-modules of finite ranks.
Suppose that $\aalpha$ is injective with torsion-free cokernel, $\bbeta$ is surjective, $\bbeta \aalpha = 0$, and $\rank(X) = \rank(X') + \rank(X'')$.
Then this sequence is exact.
\end{lem}

\begin{proof}
There is an induced surjective map from the cokernel of $\aalpha$ to $X''$.
Since those are free $\Z_p$-modules of the same rank, the surjective map must be isomorphic. 
\end{proof}

\begin{proof}[Proof of Proposition \ref{prop:57}]
We have $\bbeta_n \aalpha_n = \bbeta_0 \aalpha_0 = 0$ by Theorem \ref{thm:116}(1) and $\bbeta_n$ is surjective by Theorem \ref{thm:116}(2).
By Lemma \ref{lem:53}(1), we have
\[
\rank_{\Z_p}(R_{n}^{\oplus 2}/(\gamma-1)B_{n}^T) = [K_0:\Q](p^n+1) = \rank_{\Z_p}(\hat{E}(\m_{n})) +\rank_{\Z_p}(R_0).
\]
We shall show that $\aalpha_n$ is injective with torsion-free cokernel, which will conclude the proof by Lemma \ref{lem:61}.

We use induction on $n$.
When $n = 0$, consider the composition of $\aalpha_0$ with the projection 
$R_{0}^{\oplus 2} \oplus R_{-1} \to R_{0}^{\oplus 2}$
to the first factor.
Since it is presented by 
$\begin{pmatrix} 0 & 1\\ N_{\Delta} & 0 \end{pmatrix}$,
it is injective and its cokernel is torsion-free.
Hence $\aalpha_0$ itself is injective and its cokernel is torsion-free, as claimed.

When $n \geq 1$, by induction we may suppose that $\aalpha_{n-1}$ is injective with torsion-free cokernel.
A similar proof as in Lemma \ref{lem:53}(1) shows that $R_n^{\oplus 2}/A_n^T$ is a free $\Z_p$-module.
Therefore the commutative left square of \eqref{eq:07} shows that $\aalpha_n$ is injective with torsion-free cokernel.
This completes the proof of Proposition \ref{prop:57}.
\end{proof}

\begin{thm}\label{thm:63}
The following sequences are exact.

(1)
\[
0 \to \varprojlim_n \hat{E}(\m_n)^* \to \RR^{\oplus 2} \oplus R_{-1} \to R_{0} \oplus R_{-1} \to 0,
\]
where the first map is $(\Cole', \Psi_{\dd_{-1}})$ and the second map is presented by
\begin{equation}\label{eq:166}
\begin{pmatrix} 0 & 1 \\ 1 & 0 \\ -N_{\Delta} & \varphi + \varphi^{-1} - a_p \end{pmatrix}.
\end{equation}

(2) 
\begin{equation}\label{eq:104}
0 \to \varprojlim_n \hat{E}(\m_n)^* \to \RR^{\oplus 2} \oplus R_{-1} \to R_{0} \oplus R_{-1} \to 0,
\end{equation}
where the first map is $(\Cole, \Psi_{\dd_{-1}}) = (\Cole^{\sharp}, \Cole^{\flat}, \Psi_{\dd_{-1}})$ and the second map is presented by
\begin{equation}\label{eq:689}
\begin{pmatrix} 1 & 0 \\ -a_p & -1 \\ -N_{\Delta} & \varphi + \varphi^{-1} - a_p \end{pmatrix}.
\end{equation}
\end{thm}

\begin{proof}
(1)
By Proposition \ref{prop:57}, the $\Z_p$-linear dual of \eqref{eq:07} yields the commutative diagram with exact rows
\begin{equation}
\xymatrix{
	0 \ar[r] &
	\hat{E}(\m_{n})^* \ar[r]^-{\bbeta_n^*} \ar@{->>}[d] &
	(R_{n}^{\oplus 2}/(\gamma-1)B_{n}^T \oplus R_{-1})^* \ar[r]^-{\aalpha_{n}^*} \ar@{->>}[d]^-{(\times A_n^T \oplus \id)^*} &
	(R_{0} \oplus R_{-1})^* \ar[r] \ar@{=}[d] &
	0 \\
	0 \ar[r] &
	\hat{E}(\m_{n-1})^* \ar[r]^-{\bbeta_{n-1}^*} &
	(R_{n-1}^{\oplus 2}/(\gamma-1)B_{n-1}^T \oplus R_{-1})^* \ar[r]^-{\aalpha_{n-1}^*} &
	(R_{0} \oplus R_{-1})^* \ar[r] &
	0.
}
\end{equation}
Using the identifications in \eqref{eq:161} and \eqref{eq:96}, this diagram can be rewritten as
\begin{equation}\label{eq:167}
\xymatrix@C=36pt{
	0 \ar[r] &
	\hat{E}(\m_{n})^* \ar[r]^-{(\Cole_n', \Psi_{d_{-1}})} \ar@{->>}[d] &
	R_{n}^{\oplus 2}/(\gamma-1)B_{n}^{\iota} \oplus R_{-1} \ar[r] \ar@{->>}[d] &
	R_{0} \oplus R_{-1} \ar[r] \ar@{=}[d] &
	0 \\
	0 \ar[r] &
	\hat{E}(\m_{n-1})^* \ar[r]^-{(\Cole_{n-1}', \Psi_{d_{-1}})} &
	R_{n-1}^{\oplus 2}/(\gamma-1)B_{n-1}^{\iota} \oplus R_{-1} \ar[r] &
	R_{0} \oplus R_{-1} \ar[r] &
	0,
}
\end{equation}
where the middle vertical arrow is the natural projection by the proof of Lemma \ref{lem:998}.
Moreover, the following diagram is commutative (see Lemma \ref{lem:693}(2)).
\[
\xymatrix{
	(R_0^{\oplus 2} \oplus R_{-1})^* \ar[r]^-{s_0^*} \ar[d]_-{\vsim} &
	(R_0 \oplus R_{-1})^* \ar[d]^-{\vsim} \\
	R_0^{\oplus 2} \oplus R_{-1} \ar[r] &
	R_0 \oplus R_{-1},
	}
\]
where the vertical arrows are given by \eqref{eq:96} and the lower horizontal arrow has the presentation \eqref{eq:166}.
Therefore taking the limit of \eqref{eq:167} shows the assertion.

(2) Consider the diagram
\[
\xymatrix{
	0 \ar[r] &
	\varprojlim_n \hat{E}(\m_n)^* \ar[r] \ar@{=}[d] &
	\RR^{\oplus 2} \oplus R_{-1} \ar[r] \ar[d]_{\vsim}^{\scriptsize \begin{pmatrix} -a_p & -1 \\ 1 & 0 \end{pmatrix} \oplus \id} &
	R_{0} \oplus R_{-1} \ar[r] \ar@{=}[d] &
	0\\
	0 \ar[r] &
	\varprojlim_n \hat{E}(\m_n)^* \ar[r] &
	\RR^{\oplus 2} \oplus R_{-1} \ar[r]&
	R_{0} \oplus R_{-1} \ar[r] &
	0
}
\]
where the upper row is that in (1) and the lower row is that in (2).
The presentation \eqref{eq:689} is defined so that this diagram is commutative.
Thus the exactness of the upper row implies that of the lower row.
\end{proof}

\begin{cor}\label{cor:90}
The following sequences are exact.

(1)
\[
\varprojlim_n \hat{E}(\m_n)^* \to \RR \oplus R_{-1} \to R_{-1} \to 0,
\]
where the first map is $(\Cole^{\flat}, \Psi_{\dd_{-1}})$ and the second map is 
$\begin{pmatrix} -1 \\ \varphi + \varphi^{-1} - a_p \end{pmatrix}$.

(2)
\[
\varprojlim_n \hat{E}(\m_n)^* \to \RR \to \Rnt/(a_p) \to 0,
\]
where the first map is $\Cole^{\sharp}$ and the second map is the natural projection.
\end{cor}

\begin{proof}
(1) This follows from Theorem \ref{thm:63}(2) and the observation that the image of $\RR \oplus 0 \oplus 0$ under the second map of \eqref{eq:104} is $R_0 \oplus 0$.

(2) Similarly, it is enough to show that the image of $0 \oplus \RR \oplus R_{-1}$ under the second map of \eqref{eq:104} is $(N_{\Delta}, a_p)R_{0} \oplus R_{-1}$.
The image is equal to the image of the map $R_0 \oplus R_{-1} \to R_0 \oplus R_{-1}$ given by the matrix 
$\begin{pmatrix} -a_p & -1 \\ -N_{\Delta} & \varphi + \varphi^{-1} - a_p \end{pmatrix}$.
Recall the decomposition $R_0 = R_0^{\Delta} \times \Rnt$.

For the $\Rnt$-part, our claim is that the image of $-a_p: \Rnt \to \Rnt$ is $a_p\Rnt$, which is clear.
For the $R_0^{\Delta}$-part, using the isomorphism $R_{-1} \simeq R_0^{\Delta}$, our claim is that the map $R_{-1} \oplus R_{-1} \to R_{-1} \oplus R_{-1}$ given by the matrix 
$\begin{pmatrix} -a_p & -1 \\ -(p-1) & \varphi + \varphi^{-1} - a_p \end{pmatrix}$
is surjective.
This is true since, by Nakayama's lemma, the determinant is a unit in $R_{-1}$.
\end{proof}

By the isomorphism in Proposition \ref{prop:895}, we can consider $\Cole$ as the map from $(E(k_{\infty}) \otimes (\Q_p/\Z_p))^{\dual}$.
Thus the first terms in the sequences of Corollary \ref{cor:90}(1)(2) can be replaced by $E(k_{\infty}) \otimes (\Q_p/\Z_p)^{\dual}$.

\begin{defn}\label{defn:231}
Define the submodules $E^{\sharp/\flat}_{\infty}$ of $E(k_{\infty}) \otimes (\Q_p/\Z_p)$ to fit into the exact sequences
\begin{equation}\label{eq:11}
0 \to (E^{\flat}_{\infty})^{\dual} \to \RR \oplus R_{-1} \to R_{-1} \to 0
\end{equation}
and
\begin{equation}\label{eq:12}
0 \to (E^{\sharp}_{\infty})^{\dual} \to \RR \to \Rnt/(a_p) \to 0
\end{equation}
induced by Corollary \ref{cor:90}(1)(2).
These generalize \cite[Definition 7.9]{Spr12}.
\end{defn}

\begin{cor}\label{cor:230}
We have $\pd_{\RR}((E^{\sharp/\flat}_{\infty})^{\dual}) \leq 1$ and $\rank_{\Lambda}((E^{\sharp/\flat}_{\infty})^{\dual}) = [K_0:\Q]$.
\end{cor}

\begin{proof}
It is easy to see that $\pd_{\RR}(R_{-1}) = 1$ and $\pd_{\RR}(\Rnt/(a_p)) = 1$ (resp. $2$) if $a_p = 0$ (resp. $a_p \neq 0$).
Hence the claim follows from the exact sequences \eqref{eq:11} and \eqref{eq:12}.
\end{proof}

\subsection{$a_p = 0$ Case}\label{subsec:27}

In this subsection, we assume $a_p = 0$ and consider the $\pm$-theory.

\begin{defn}
As a counterpart to \eqref{eq:105}, we introduce the $\pm$-parts of $\hat{E}(\m_n)$ as follows \cite[Definition 8.16]{Kob03}.
For $n \geq -1$, define
\[
\hat{E}^{\pm}(\m_n) = \{ x \in \hat{E}(\m_{n}) \mid \Tr_{n'+1}^{n}(x) \in \hat{E}(\m_{n'}), -1 \leq \forall n' < n, (-1)^{n'} = \pm 1\}.
\]
Then $\hat{E}^{\pm}(\m_n)$ is precisely the $p$-part of $E^{\pm}(k_n)$ by Proposition \ref{prop:895}.
\end{defn}

Let $d_n$ be as in Theorem \ref{thm:116}.
For $n \geq 0$, put 
\[
\dd_{n}^+ = 
	\begin{cases}
		(-1)^{\frac{n+2}{2}} \dd_{n} & (\text{$n$ is even}) \\
		(-1)^{\frac{n+1}{2}} \dd_{n-1} & (\text{$n$ is odd})
	\end{cases}
\]
and
\[
\dd_{n}^- = 
	\begin{cases}
		(-1)^{\frac{n+1}{2}} \dd_{n} & (\text{$n$ is odd}) \\
		(-1)^{\frac{n}{2}} \dd_{n-1} & (\text{$n$ is even}).
	\end{cases}
\]

\begin{prop}\label{prop:28}
For $n \geq 0$, the following are true.

(1) We have $\hat{E}^{+}(\m_{n}) = (\dd_{n}^+, \dd_{-1})_{R_{n}}$ and $\hat{E}^{-}(\m_{n}) = (\dd_{n}^-)_{R_{n}}$.

(2) We have an exact sequence
\[
0 \to \hat{E}(\m_{-1}) \to \hat{E}^+(\m_{n}) \oplus \hat{E}^-(\m_{n}) \to \hat{E}(\m_{n}) \to 0
\]
where the first map is the diagonal inclusion and the second map sends $(x,y)$ to $x-y$.
\end{prop}

\begin{proof}
This can be shown by the same proof as in \cite[Proposition 3.16]{KO18} and \cite[Proposition 8.12]{Kob03}.
\end{proof}

Recall the elements $\ttilde{\omega}_n^{\pm}$ and $\omega_n^{\pm}$ defined in Subsection \ref{subsec:403}.
Then $\Phi_{d_{n}^{\pm}}$ factors through $R_{n}/\omega_n^{\pm}$, since $\omega_n^{\pm} d_{n}^{\pm} = 0$.
Therefore, similarly as \eqref{eq:161}, we can consider the composite
\[
\hat{E}^{\pm}(\m_{n})^* \overset{(\Phi_{d_{n}^{\pm}})^*}{\to} (R_{n}/\omega_n^{\pm})^* 
\overset{\sim}{\to} R_{n}[(\omega_n^{\pm})^{\iota}] \overset{\sim}{\leftarrow} R_{n}/(\omega_n^{\pm})^{\iota},
\]
where the final arrow is the multiplication by $(\ttilde{\omega}_n^{\mp})^{\iota}$.
Thus we obtain the following.

\begin{defn}
For $n \geq 0$, define the $\pm$-Coleman maps $\Cole_{n}^{\pm}: \hat{E}^{\pm}(\m_{n})^* \to R_{n}/(\omega_n^{\pm})^{\iota}$ characterized by
\begin{equation}\label{eq:233}
\Cole_{n}^{\pm}(f) (\ttilde{\omega}_n^{\mp})^{\iota} = \sum_{\sigma \in G_{n}} f(\sigma(\dd_{n}^{\pm}))\sigma
\end{equation}
for $f \in \hat{E}^{\pm}(\m_{n})^*$.
This is a generalization of \cite[Corollary 8.20]{Kob03}.
\end{defn}

The relation with Definition \ref{defn:55} is given by the following.

\begin{lem}\label{lem:64}
For $n \geq 0$, we have $\Cole'_{n} = (-\Cole_{n}^+, \Cole_{n}^-)$.
\end{lem}

\begin{proof}
Suppose $n$ is even.
By Definition \ref{defn:688}, we have
\[
\ttilde{B_n} = \begin{pmatrix} (-1)^{n/2}\ttilde{\omega}_n^- & 0 \\ 0 & (-1)^{n/2}\ttilde{\omega}_n^+ \end{pmatrix}.
\]
Then for any $f \in \hat{E}(\m_n)^*$, we have
\begin{align}
(-\Cole_n^+(f), \Cole_n^-(f)) \ttilde{B_n}^{\iota} 
&= \left((-1)^{(n+2)/2} \Cole_n^+(f)(\ttilde{\omega}_n^-)^{\iota}, (-1)^{n/2} \Cole_n^-(f)(\ttilde{\omega}_n^+)^{\iota}\right)\\
&= \left( \sum_{\sigma \in G_{n}} f(\sigma(\dd_{n}))\sigma, \sum_{\sigma \in G_{n}} f(\sigma(\dd_{n-1}))\sigma \right)\\
&= \Cole_n'(f) \ttilde{B_n}^{\iota}.
\end{align}
The case where $n$ is odd can be shown similarly.
\end{proof}

\begin{defn}\label{defn:878}
Observe that $\Cole_n^{\pm}$ are compatible along $n$ by Lemmas \ref{lem:998} and \ref{lem:64} (alternatively we can check this property by a direct computation with \eqref{eq:233}).
Define
\[
\Cole^{\pm}: \varprojlim_n \hat{E}^{\pm}(\m_{n})^* \to \varprojlim_n R_{n}/(\omega_n^{\pm})^{\iota} \simeq \RR
\]
as the induced map.
This is a generalization of \cite[Definition 8.22]{Kob03}.
Therefore we have $\Cole' = (-\Cole^+, \Cole^-)$ and
\[
\Cole = (\Cole^{\sharp}, \Cole^{\flat}) = (\Cole^-, \Cole^+).
\]
\end{defn}

Next we investigate the structures of $\hat{E}^{\pm}(\m_n)$ similarly as in the previous subsection.
For $n \geq 0$, we define sequences
\begin{equation}\label{eq:162}
0 \to R_{-1} \overset{\aalpha_n^+}{\to} R_{n} /\omega_n^+ \oplus R_{-1} \overset{\bbeta_n^+}{\to} \hat{E}^+(\m_{n}) \to 0
\end{equation}
and
\begin{equation}\label{eq:163}
0 \to \Rnt \overset{\aalpha_n^-}{\to} R_{n} /\omega_n^- \overset{\bbeta_n^-}{\to} \hat{E}^-(\m_{n}) \to 0
\end{equation}
as follows.
Let $\aalpha_n^+$ be the map presented by $(- N_{\Delta} \ttilde{\omega}_n^+, \varphi + \varphi^{-1})$.
Let $\bbeta_n^+$ be induced by $\Phi_{d_n^+, d_{-1}}$.
Let $\aalpha_n^-$ be the multiplication by $\ttilde{\omega}_n^-$.
Let $\bbeta_n^-$ be induced by $\Phi_{\dd_{n}^-}$.

\begin{prop}\label{prop:32}
For $n \geq 0$, the sequences \eqref{eq:162} and \eqref{eq:163} are exact.
\end{prop}

\begin{proof}
This can be shown by the same argument as in Proposition \ref{prop:57}.
The surjectivities come from Proposition \ref{prop:28}(1).
The rank computation can be done via Proposition \ref{prop:28}(2) as in \cite[Corollary 8.13]{Kob03}.
\end{proof}

Now we obtain the following (Theorem \ref{thm:93}(3)).

\begin{thm}\label{thm:34}
The following sequences are exact.

(1) 
\[
0 \to (E^+(k_{\infty}) \otimes (\Q_p/\Z_p))^{\dual} \to \RR \oplus R_{-1} \to R_{-1} \to 0,
\]
where the first map is $(\Cole^+, \Psi_{d_{-1}})$ and the second map is presented by $\begin{pmatrix} -1 \\ \varphi+\varphi^{-1} \end{pmatrix}$.

(2)
\[
0 \to (E^-(k_{\infty}) \otimes (\Q_p/\Z_p))^{\dual} \to \RR \to \Rnt \to 0,
\]
where the first map is $\Cole^-$ and the second map is the natural projection.
\end{thm}

\begin{proof}
This theorem can be shown as in Theorem \ref{thm:63}.
More concretely, it follows from taking the $\Z_p$-linear dual of the sequences \eqref{eq:162} and \eqref{eq:163} and taking the limit (the identification as in Proposition \ref{prop:895} is also used).
\end{proof}

\begin{rem}\label{rem:879}
Theorem \ref{thm:34}(2) claims that 
$\Cole^-: (E^-(k_{\infty}) \otimes (\Q_p/\Z_p))^{\dual} \to \au^{-}$ is isomorphic.
On the other hand, by Theorem \ref{thm:34}(1), we have equivalences
\begin{align}
& \text{$\Cole^+: (E^+(k_{\infty}) \otimes (\Q_p/\Z_p))^{\dual} \to \RR$ is isomorphic}\\
&\quad\Leftrightarrow \text{$\varphi + \varphi^{-1}$ is a unit of $R_{-1}$} \\
&\quad \Leftrightarrow \text{the residue degree of $p$ in $K/\Q$ is not divisible by $4$}
\end{align}
(see \cite[Lemma 3.6]{KO18} for the final equivalence).
Therefore, $\Cole^+$ often fails to be isomorphic.
Such an obstruction is observed in \cite[Remark 3.5]{KO18}.
M. Kim \cite{Kim11} overlooked this obstruction, as pointed out in \cite{KO18}.
\end{rem}

The following is a consequence of comparing Definition \ref{defn:231} and Theorem \ref{thm:34}.

\begin{cor}\label{cor:870}
We have $E^{\flat}_{\infty} = E^{+}(k_{\infty}) \otimes (\Q_p/\Z_p)$ and $E^{\sharp}_{\infty} = E^{-}(k_{\infty}) \otimes (\Q_p/\Z_p)$ as submodules of $E(k_{\infty}) \otimes (\Q_p/\Z_p)$.
\end{cor}

Here we record a lemma which will be useful in Section \ref{sec:07}.
It is a generalization of the computation in \cite[\S 4]{KK}.
Recall the definition of $\au^{\pm}$ in Section \ref{sec:01}.

\begin{lem}\label{lem:623}
For $n \geq 0$, we have
\[
\left(\frac{E(k_n) \otimes (\Q_p / \Z_p)}{E^{\pm}(k_n) \otimes (\Q_p / \Z_p)}\right)^{\dual} \simeq R_n / \au^{\pm} \ttilde{\omega}_n^{\mp}.
\]
\end{lem}

\begin{proof}
We have
\begin{align}
\left(\frac{\hat{E}(\m_n) \otimes (\Q_p / \Z_p)}{\hat{E}^{\pm}(\m_n) \otimes (\Q_p / \Z_p)}\right)^{\dual} 
&\simeq \Ker \left(\hat{E}(\m_n)^* \to \hat{E}^{\pm}(\m_n)^* \right)\\
&\simeq \Ker \left(\hat{E}^{\mp}(\m_n)^* \to \hat{E}(\m_{-1})^* \right),
\end{align}
where the last isomorphism follows from Proposition \ref{prop:28}(2).
By Proposition \ref{prop:32} and the fact that $\Phi_{d_{-1}}: R_{-1} \to \hat{E}(\m_{-1})$ is an isomorphism by Theorem \ref{thm:116}(2), we can see that
\[
\Cok \left(\hat{E}(\m_{-1}) \to \hat{E}^-(\m_{n}) \right)
\simeq R_n / \ttilde{\omega}_n^- = R_n/\au^+\ttilde{\omega}_n^-
\]
and
\[
\Cok \left(\hat{E}(\m_{-1}) \to \hat{E}^+(\m_{n}) \right)
\simeq R_n / (N_{\Delta}, \gamma - 1) \ttilde{\omega}_n^+ = R_n / \au^{-} \ttilde{\omega}_n^{+}.
\]
Since $(R_n/\au^{\pm} \ttilde{\omega}_n^{\mp})^* \simeq R_n/ (\au^{\pm} \ttilde{\omega}_n^{\mp})^{\iota} = R_n/ \au^{\pm} \ttilde{\omega}_n^{\mp}$, we obtain the result.
\end{proof}

\section{Structures of Selmer Groups}\label{sec:220}

In this section, we prove Theorem \ref{thm:76} using Theorem \ref{thm:93}.
We always assume Assumption \ref{ass:04} in the ordinary case.

First we give the definition of $\sharp/\flat$-Selmer groups, generalizing \cite[Definition 7.11]{Spr12}.

\begin{defn}\label{defn:930}
When $p \mid a_p$, we define the $S$-imprimitive $\sharp/\flat$-Selmer groups by
\begin{equation}\label{eq:200}
\Sel_S^{\sharp/\flat}(E/K_{\infty}) = 
\Ker \left( \Sel_S(E/K_{\infty}) \to \frac{E(k_{\infty}) \otimes (\Q_p/\Z_p)}{E^{\sharp/\flat}_{\infty}}  \right),
\end{equation}
where $E_{\infty}^{\sharp/\flat}$ is defined in Definition \ref{defn:231}.
By Corollary \ref{cor:870}, when $a_p = 0$, this definition is consistent with \eqref{eq:109} and the convention $(\sharp, \flat) = (-, +)$
\end{defn}

We recall here that $\sharp/\flat$-Selmer groups are not known to be $\Lambda$-cotorsion, though they are when $a_p = 0$ by Proposition \ref{prop:851}.

Take a finite set $\Sigma$ of prime numbers containing $S \cup \prim(pN)$.
Let $\Q_{\Sigma}$ denote the maximal algebraic extension of $\Q$ which is unramified outside $\Sigma$.
Then the definition \eqref{eq:107} of $\Sel_S(E/K_{\infty})$ can be rewritten as
\begin{equation}\label{eq:101}
\Sel_S(E/K_{\infty}) = 
\Ker \left( H^1(\Q_{\Sigma} / K_{\infty}, E[p^{\infty}]) \to 
\prod_{l \in \Sigma \setminus S} \frac{H^1(K_{ \infty} \otimes \Q_l, E[p^{\infty}])}{E(K_{ \infty} \otimes \Q_l) \otimes (\Q_p/\Z_p)}  \right).
\end{equation}
Similar formulas for $\Sel_S^{\bullet}(E/K_{\infty})$, etc., also hold.

Let $\kappa: \Gal(\overline{\Q}/\Q) \twoheadrightarrow G_{\infty} \hookrightarrow \RR^{\times}$ be the natural group homomorphism.
Put $\T = T_pE \otimes_{\Z_p} \RR$, which is a free $\RR$-module of rank two.
We equip $\T$ with the action of $\Gal(\overline{\Q}/\Q)$, defined by the natural action on the first component and $\kappa^{-1}$ on the second component.
Then Shapiro's lemma gives natural identifications
\[
H^1(\Q_{\Sigma}/\Q, \T) = \varprojlim_n H^1(\Q_{\Sigma}/K_{n}, T_pE)
\]
and
\[
H^1(\Q_l, \T) = \varprojlim_n H^1(K_{n} \otimes \Q_l, T_pE)
\]
for any prime number $l$, where the transition maps are the corestriction maps.
Namely, these can be regarded as the global and local Iwasawa cohomology groups, respectively.
Moreover, we have $H^1(\Q_{\Sigma}/\Q, \T) = H^1(\Q, \T)$ by \cite[Corollary B.3.6]{Rub00}, so we prefer to adopt the simpler notation $H^1(\Q, \T)$.

By the Tate duality and the Weil pairing, for any prime number $l$, we have natural perfect pairings
\begin{equation}\label{eq:98}
H^1(K_{\infty} \otimes \Q_l, E[p^{\infty}]) \times H^1(\Q_l, \T) \to \Q_p/\Z_p
\end{equation}
and in particular 
\begin{equation}\label{eq:01}
H^1(k_{\infty}, E[p^{\infty}]) \times H^1(\Q_p, \T) \to \Q_p/\Z_p.
\end{equation}

The following proposition is a collection of several facts which are more or less known to experts.

\begin{prop}\label{prop:106}
We have a natural exact sequence
\begin{equation}\label{eq:52}
0 \to H^1(\Q, \T) \to \bigoplus_{l \in \Sigma} H^1(\Q_l, \T) \to H^1(\Q_{\Sigma}/K_{ \infty}, E[p^{\infty}])^{\dual} \to \Sel^0(E/K_{\infty})^{\dual} \to 0.
\end{equation}
Each module which appears in this sequence satisfies the following.

(1) $H^1(\Q, \T)$ is a free $\Lambda$-module of rank $[K_0:\Q]$.

(2) $H^1(\Q_l, \T)$ satisfies $\pd_{\Lambda} \leq 1$ and is $\Lambda$-torsion for $l \neq p$.

(3) $H^1(\Q_p, \T)$ satisfies $\pd_{\RR} \leq 0$ and the $\Lambda$-rank is $2[K_0:\Q]$ (see Remark \ref{rem:999} below).

(4) $H^1(\Q_{\Sigma}/K_{ \infty}, E[p^{\infty}])^{\dual}$ satisfies $\pd_{\RR} \leq 1$ and the $\Lambda$-rank is $[K_0:\Q]$.

(5) $\Sel^0(E/K_{\infty})$ is $\Lambda$-cotorsion.
\end{prop}

\begin{proof}
By \cite[Theorem 12.4(1)]{Kat04} and \cite[Proposition 4.4]{Gree06}, the weak Leopoldt conjecture (i.e. $H^2(\Q_{\Sigma}/K_{ \infty}, E[p^{\infty}]) = 0$) and the assertion (5) are true.
Then the exact sequence \eqref{eq:52} is a consequence of the Poitou-Tate long exact sequence and the definition of $\Sel^0(E/K_{\infty})$.

We show the statements on the ranks.
For any prime number $l$, we have $H^2(K_{\infty} \otimes \Q_l, E[p^{\infty}]) = 0$ and $H^0(K_{\infty} \otimes \Q_l, E[p^{\infty}])$ is cofinitely generated over $\Z_p$.
Then the local Euler-Poincare characteristic formula \cite[Proposition 4.2]{Gree06} shows that 
\[
\corank_{\Lambda} H^1(K_{\infty} \otimes \Q_l, E[p^{\infty}]) = 
	\begin{cases}
		0 & ( l \neq p) \\
		2[K_0:\Q] & (l = p).
	\end{cases}
\]
Thus we obtain the rank parts of (2) and (3).
We can similarly obtain the rank part of (4), using the weak Leopoldt conjecture above and the global Euler-Poincare characteristic formula \cite[Proposition 4.1]{Gree06}.
Now the rank part of (1) also follows.

We show the other parts of the statements.
The freeness part of (1) can be proved as in \cite[Remark 6.5]{Sak} by $E(K_{\infty})[p] = 0$.
The same argument also show that $H^1(\Q_p, \T)$ is a free $\Lambda$-module.
By the local and global ``Tate sequence'' (see \cite[Lemma 4.5]{OV02}), we have
\[
\pd_{\Lambda} H^1(\Q_l, \T) \leq 1, \quad
\pd_{\RR} H^1(\Q_p, \T) \leq 1, \quad
\pd_{\RR} H^1(\Q_{\Sigma}/K_{\infty}, E[p^{\infty}])^{\dual} \leq 1.
\]
Here, for the second and the third inequalities, we used our assumptions $H^0(k_{\infty}, E[p^{\infty}]) = 0$ and $H^0(\Q_{\Sigma}/K_{\infty}, E[p^{\infty}]) = 0$, respectively.
Finally for (3), we use the following general fact.
For a finitely generated $\RR$-module $X$, we have $\pd_{\RR}(X) \leq 0$ if and only if $X$ is free over $\Lambda$ and $\pd_{\RR}(X) \leq 1$.
Thus we have $\pd_{\RR} H^1(\Q_p, \T) \leq 0$ by the above results.
This completes the proof.
\end{proof}

\begin{rem}\label{rem:999}
In general, a finitely generated $\RR$-module $X$ with $\pd_{\RR}(X) \leq 0$ is not necessarily free over $\RR$.
This is simply because $\RR$ is a product of local rings (associated to characters of $G_0$ of order prime to $p$) and the ranks of $X$ of each components may be different.
However, $H^1(\Q_p, \T)$ in Proposition \ref{prop:106}(3) is in fact a free $\RR$-module of rank 2.
This is shown by decomposing the result from the local Euler-Poincare characteristic formula into character-wise statement.
\end{rem}

Recall the convention that $\bullet = \emptyset$ (resp. $\bullet \in \{+, -\}$, resp. $\bullet \in \{+, -, \sharp, \flat\}$) when $p \nmid a_p$ (resp. $a_p = 0$, resp. $p \mid a_p$).

\begin{defn}\label{defn:855}
For $\bullet \in \{\emptyset, +, -, \sharp, \flat\}$, define the submodule $H^1_f(\Q_p, \T)^{\bullet}$ of $H^1(\Q_p, \T)$ as the orthogonal of
\[
	\begin{cases}
		E(k_{\infty}) \otimes (\Q_p/\Z_p) & (\bullet = \emptyset)\\
		E^{\bullet}(k_{\infty}) \otimes (\Q_p/\Z_p) & (\bullet \in \{+, -\})\\
		E_{\infty}^{\bullet} & (\bullet \in \{\sharp, \flat\})
	\end{cases}
\]
under the paring \eqref{eq:01}.
Moreover, put
\begin{equation}
H^1_{/f}(\Q_p, \T)^{\bullet} = H^1(\Q_p, \T) / H^1_{f}(\Q_p, \T)^{\bullet}.
\end{equation}
\end{defn}

\begin{lem}\label{lem:962}
Let $\bullet \in \{\emptyset, +, -, \sharp, \flat\}$.

(1) We have an isomorphism
\begin{equation}
H^1_{/f}(\Q_p, \T)^{\bullet} \simeq 
	\begin{cases}
		(E(k_{\infty}) \otimes (\Q_p/\Z_p))^{\dual} & (\bullet = \emptyset)\\
		(E^{\bullet}(k_{\infty}) \otimes (\Q_p/\Z_p))^{\dual} & (\bullet \in \{+, -\})\\
		(E_{\infty}^{\bullet})^{\dual} & (\bullet \in \{\sharp, \flat\}).
	\end{cases}
\end{equation}

(2) We have $\pd_{\RR} (H^1_{/f}(\Q_p, \T)^{\bullet}) \leq 1$ and $\pd_{\RR} (H^1_{f}(\Q_p, \T)^{\bullet}) \leq 0$.
Moreover, both have the same $\Lambda$-rank $[K_0:\Q]$.
\end{lem}

\begin{proof}
(1) Clear from definition.

(2) The assertions on $H^1_{/f}(\Q_p, \T)^{\bullet}$ follows from (1) and Theorem \ref{thm:93} (see Corollary \ref{cor:230} for the $p \mid a_p$ case). 
Then the assertions on $H^1_{f}(\Q_p, \T)^{\bullet}$ also follows from Proposition \ref{prop:106}(3).
\end{proof}

For $\bullet \in \{\emptyset, +, -, \sharp, \flat\}$, we can consider the composite map
\begin{equation}\label{eq:19}
H^1(\Q, \T) \overset{\loc}{\to} H^1(\Q_p, \T) \to H^1_{/f}(\Q_p, \T)^{\bullet} \overset{\Cole^{\bullet}}{\to} \RR.
\end{equation}
Here, the final map $\Cole^{\bullet}$ is defined using the isomorphism in Lemma \ref{lem:962}(1).
We denote the composite of the first two maps by $\loc_{/f}^{\bullet}: H^1(\Q, \T) \to H^1_{/f}(\Q_p, \T)^{\bullet}$.

The following is a consequence of Proposition \ref{prop:106}.
Such kinds of results are standard in Iwasawa theory; for example, \eqref{eq:02} is a generalization of \cite[Theorem 7.3(i)]{Kob03}.

\begin{prop}\label{prop:945}
For $\bullet \in \{\emptyset, +, -, \sharp, \flat\}$, $\Sel_S^{\bullet}(E/K_{\infty})$ is cotorsion if and only if the map \eqref{eq:19} is injective.
In that case, we have exact sequences
\begin{equation}\label{eq:02}
0 \to H^1(\Q, \T) \to H^1_{/f}(\Q_p, \T)^{\bullet} \oplus \bigoplus_{l \in S} H^1(\Q_l, \T) \to \Sel_S^{\bullet}(E/K_{\infty})^{\dual} \to \Sel^0(E/K_{\infty})^{\dual} \to 0,
\end{equation}
\begin{equation}\label{eq:03}
0 \to H^1_f(\Q_p, \T)^{\bullet} \oplus \bigoplus_{l \in \Sigma, l \not\in S, l \neq p} H^1(\Q_l, \T) \to H^1(\Q_{\Sigma}/K_{ \infty}, E[p^{\infty}])^{\dual} \to \Sel_S^{\bullet}(E/K_{\infty})^{\dual} \to 0,
\end{equation}
and
\begin{equation}\label{eq:04}
0 \to \bigoplus_{l \in S} H^1(\Q_l, \T) \to \Sel_S^{\bullet}(E/K_{\infty})^{\dual} \to \Sel^{\bullet}(E/K_{\infty})^{\dual} \to 0.
\end{equation}
\end{prop}

\begin{proof}
The definition of $\Sel_S^{\bullet}(E/K_{\infty})$ (see \eqref{eq:101}) yields the exact sequence \eqref{eq:03} without the first injectivity.
We shall observe the following equivalences.
\begin{align}
\text{$\Sel_S^{\bullet}(E/K_{\infty})$ is $\Lambda$-cotorsion}
&\Leftrightarrow \text{the map $H^1_f(\Q_p, \T)^{\bullet} \to H^1(\Q_{\Sigma}/K_{ \infty}, E[p^{\infty}])^{\dual}$ is injective}\\
&\Leftrightarrow \text{the map $H^1(\Q, \T) \to H^1_{/f}(\Q_p, \T)^{\bullet} \oplus \bigoplus_{l \in \Sigma \setminus \{p\}} H^1(\Q_l, \T)$ is injective}\\
&\Leftrightarrow \text{the map $H^1(\Q, \T) \to H^1_{/f}(\Q_p, \T)^{\bullet}$ is injective}\\
&\Leftrightarrow \text{the map \eqref{eq:19} is injective}
\end{align}
Here, the first equivalence follows from \eqref{eq:03} (except for the first injectivity), the rank parts of Proposition \ref{prop:106}(2)(4), and Lemma \ref{lem:962}(2);
the second follows from the exact sequence \eqref{eq:52};
the third follows from the facts by Proposition \ref{prop:106}(1)(2) that $H^1(\Q, \T)$ is torsion-free over $\Lambda$ and $H^1(\Q_l, \T)$ is torsion over $\Lambda$;
the fourth follows from the facts that $H^1(\Q, \T)$ is torsion-free over $\Lambda$ and the kernel of $\Cole^{\bullet}: H^1_{/f}(\Q_p, \T)^{\bullet} \to \RR$ is $\Lambda$-torsion (Theorem \ref{thm:93}).

Under the above equivalent conditions, 
it is easy to deduce the exact sequences \eqref{eq:02}, \eqref{eq:03}, \eqref{eq:04} from the exact sequence \eqref{eq:52} and the definitions of Selmer groups.
\end{proof}

Recall that, for a finitely generated $\Lambda$-module $X$, we have $\pd_{\Lambda}(X) \leq 1$ if and only if $X$ does not contain a non-trivial finite submodule (e.g. \cite[Proposition (5.3.19)(i)]{NSW}).

\begin{prop}\label{prop:233}
Let $\bullet \in \{\emptyset, +, -, \sharp, \flat \}$ and suppose that $\Sel_S^{\bullet}(E/K_{\infty})$ is $\Lambda$-cotorsion.
Then $\Sel_S^{\bullet}(E/K_{\infty})^{\dual}$ does not contain a non-trivial finite submodule.
\end{prop}

\begin{proof}
There is a lot of literature on the non-existence of a non-trivial finite submodules.
In fact, \cite[Proposition 4.14]{Gree99} and \cite[Theorem 1.3]{KO18} show our assertion in the case where $p \nmid a_p$ and $a_p = 0$, respectively.

In order to prove the whole case simultaneously, we utilize a quite general work of Greenberg \cite[Proposition 4.1.1]{Gree16}.
We apply that proposition to the situation where the base field is $K_0$, the coefficient ring is $\Lambda$, the representation is the cyclotomic deformation $\D = \Hom(\Lambda, E[p^{\infty}])$ (as in \cite[\S 4.4]{Gree16}).
Take the local conditions $\LL$ for $\D$ so that the Selmer group coincides with $\Sel_S^{\bullet}(E/K_{\infty})$
We check the hypotheses of \cite[Proposition 4.1.1]{Gree16}.

\begin{itemize}
\item As written in \cite[\S 4.4]{Gree16}, the conditions on $\RFX(\D)$, $\LOC_{\eta}^{(1)}(\D)$, and $\LOC_v^{(2)}(\D)$ are automatically satisfied.

\item $\LEO(\D)$ is equivalent to the weak Leopoldt conjecture $H^2(\Q_{\Sigma}/K_{\infty}, E[p^{\infty}]) = 0$, which holds in our case (already remarked in the proof of Proposition \ref{prop:106}).

\item By Proposition \ref{prop:106}, $\CRK(\D, \LL)$ is equivalent to that $\Sel_S^{\bullet}(E/K_{\infty})^{\dual}$ is $\Lambda$-cotorsion, which holds by assumption.

\item We show that the local condition $\LL$ is almost divisible.
For a prime number $l \not\in \{p\} \cup S$, the local condition at $l$ is $0$ and we have nothing to do.
For a prime number $l \in S$, by Proposition \ref{prop:106}(2), the $\Lambda$-module $H^1(\Q_l, \T)$ does not contain a non-trivial finite submodule.
Finally $H^1_{/f}(\Q_p, \T)^{\bullet}$ does not contain a non-trivial finite submodule by Lemma \ref{lem:962}.

\item There is no surjective, $\Gal(\overline{\Q}/K_0)$-equivariant map $E[p] \to \mu_p$, by $E(K_0)[p] = 0$ and the Weil pairing.
\end{itemize}
This completes the proof.
\end{proof}

Now we can prove Theorem \ref{thm:76}.
We restate it here.

\begin{thm}\label{thm:978}
Let $\bullet \in \{\emptyset, +, -, \sharp, \flat \}$ and suppose that $\Sel_S^{\bullet}(E/K_{\infty})$ is cotorsion as a $\Lambda$-module.
Then we have $\pd_{\RR}(\Sel_S^{\bullet}(E/K_{\infty})^{\dual}) \leq 1$.
\end{thm}

\begin{proof}
First we recall that, for any finitely generated torsion $\RR$-module $X$, we have $\pd_{\RR}(X) \leq 1$ if and only if both $\pd_{\RR}(X) \leq 2$ and $\pd_{\Lambda}(X) \leq 1$ hold (see e.g. the text after \cite[Proposition 3.5]{Kata}).
Therefore, by Proposition \ref{prop:233}, we only have to show $\pd_{\RR}(\Sel_S^{\bullet}(E/K_{\infty})^{\dual}) \leq 2$.

We study the first and the second terms in the exact sequence \eqref{eq:03}.
We know $\pd_{\RR}(H^1_{f}(\Q_p, \T)) =0$ and $\pd_{\RR}(H^1(\Q_{\Sigma}/K_{ \infty}, E[p^{\infty}])^{\dual}) \leq 1$ by Lemma \ref{lem:962}(2) and Proposition \ref{prop:106}(4), respectively.
Thus we only have to show that $\pd_{\RR} (H^1(\Q_l, \T)) \leq 2$ for any $l \not\in S \cup \{p\}$ (actually we will conclude $\pd_{\RR} \leq 1$ by Proposition \ref{prop:106}(2) and the above fact).
Since $l$ is unramified in the extension $K_{\infty}/\Q$, if we denote by $M_l$ the decomposition field, then $\Gal(K_{\infty}/M_l)$ is pro-cyclic with its $p$-part isomorphic to $\Z_p$.
Therefore $\Z_p[[\Gal(K_{\infty}/M_l)]]$ is a product of regular local ring of dimension 2.
Since the $\RR$-module $H^1(\Q_l, \T)$ is the induced module of a $\Z_p[[\Gal(K_{\infty}/M_l)]]$-module, we obtain the claim.
This completes the proof.
\end{proof}

\begin{rem}\label{rem:871}
We can make the set $S$ smaller.
Let $\Phi_{K/\Q}$ be the set of prime numbers $l \neq p$ such that $p$ does divide the ramification index of $l$ in $K/\Q$.
Let $S_0$ be any finite set of prime numbers $\neq p$ such that $S_0 \supset \Phi_{K/\Q}$.
Then in the situation of Theorem \ref{thm:978}, we have $\pd_{\RR}(\Sel^{\bullet}_{S_0}(E/K_{\infty})^{\dual}) \leq 1$, where the definition of $\Sel^{\bullet}_{S_0}(E/K_{\infty})$ is obvious.
This fact can be easily shown from the proof of Theorem \ref{thm:978}.

We mention here about some previous works.
In the ordinary case, the results by Greenberg \cite[Proposition 2.4.1, Proposition 3.1.1]{Gre10} show Theorem \ref{thm:978} and moreover the refined assertion above (a proof is also sketched in \cite[Theorem 5]{Kur14}).
The thesis of M. Kim \cite{Kim11} extended the work \cite{Gre10} to the $a_p = 0$ case (the author thanks Chan-Ho Kim for providing this information).
However, \cite{Kim11} contains an important flaw as noted in Remark \ref{rem:879}.

However, we again stress that, while the previous method only concerns the finiteness of the projective dimension, our argument so far is necessary in the rest of this paper to discuss the main conjecture. 

Finally note that those two works \cite{Gre10} and \cite{Kim11} deal with elliptic curves over more general base field and its (not necessarily abelian) finite Galois extension.
In fact, using our argument in this paper, we can reprove the variant of Theorem \ref{thm:978} in such a general situation.
Namely, as mentioned in Remark \ref{rem:967}, a completely local argument can yield a system of points satisfying the properties (1)(2) in Theorems \ref{thm:116} and \ref{thm:121}.
Then we can construct Coleman maps as in Section \ref{sec:32}, and prove the variant of Theorem \ref{thm:978}.
We omit the detail because, at this time, the contents of the following sections cannot be extended to such a general situation at all.
\end{rem}

\section{Relation with $p$-adic $L$-functions}\label{sec:60}

In this section, we prove Theorem \ref{thm:77}.
First we introduce the Beilinson-Kato element, using the results in the seminal paper by Kato \cite{Kat04}.
Let 
\[
\exp_{\omega_E}^*: H^1(k_n, T_pE) \otimes \Q_p \to k_n
\]
be the dual exponential map defined using the N{\'{e}}ron differential $\omega_E$ as in \cite[\S 5]{Rub98} and \cite[\S 8.7]{Kob03}.
Recall that $\T = T_pE \otimes \RR$ is a Galois representation over $\RR$.
For an element of $H^1(\Q_p, \T) \otimes \Q_p$, by attaching the subscript $n$, we denote its natural image to 
$H^1(\Q_p, T_pE \otimes R_n) \otimes \Q_p \simeq H^1(k_n, T_pE) \otimes \Q_p$.
In the next theorem, we do not need Assumption \ref{ass:04}.

\begin{thm}\label{thm:941}
There is an element $\z \in H^1(\Q, \T) \otimes \Q_p$ such that, for any character $\psi$ of $G_{n}$, 
\begin{equation}\label{eq:14}
\sum_{\tau \in G_{n}} \tau(\exp_{\omega_E}^*(\loc(\z)_n))\psi(\tau) = \frac{L_{S \cup \{p\}}(E,\psi,1)}{\Omega^{\sign(\psi)}}
\end{equation}
holds as elements of $\overline{\Q_p}$.
Here the left hand side is regarded as an element of $\overline{\Q_p}$ via the fixed embedding of $\overline{\Q}$ into $\overline{\Q_p}$.
Moreover, if $E[p]$ is irreducible as a $\Gal(\overline{\Q}/\Q)$-module and $H^0(K_0, E[p]) = 0$, then we have $\z \in H^1(\Q, \T)$.
\end{thm}

In the case where $K = \Q$ and $S = \emptyset$, \cite[(12.5)]{Kat04} constructed elements satisfying \eqref{eq:171} below.
Moreover, the formula \eqref{eq:14} is asserted in \cite[Theorem 5.2]{Kob03}.
However, \cite{Kob03} only gives a few words on the deduction of \eqref{eq:14} from \eqref{eq:171}.
In the following proof, we not only extend the result in \cite[(12.5)]{Kat04} to general $K, S$ but also discuss the deduction of \eqref{eq:14} from \eqref{eq:171}.
There is also a work by Delbourgo \cite[Appendix A]{Del08} which generalize \cite[(12.5)]{Kat04} to $K = \Q(\mu_m)$.
The method of \cite{Del08} is similar to (the first half of) the following proof, but the integrality is not asserted in \cite{Del08}.

\begin{proof}[Proof of Theorem \ref{thm:941}]
We denote by $f_E$ the newform of weight $2$ associated to $E$ by the modularity theorem.
Let $V_{\Q_p}(f_E), V_{\C}(f_E)$ be as in \cite[(6.3)]{Kat04} and $V_{\Z_p}(f_E)$ as in \cite[(8.3)]{Kat04}.
Then a modular parametrization $\pi: X_1(N) \to E$ induces an isomorphism $T_pE(-1) \otimes \Q_p \simeq V_{\Q_p}(f_E)$.
If $E[p]$ is irreducible, then we may assume that this isomorphism restricts to $T_pE(-1) \simeq V_{\Z_p}(f_E)$ as in \cite[Proposition 8]{Wut14}.

Let $m$ be the conductor of $K/\Q$.
Let $c, d$ be integers satisfying $\prim(cd) \cap (S \cup \prim(6pN)) = \emptyset$, $c \equiv d \equiv 1 \bmod N$, and $c^2 \neq 1, d^2 \neq 1$.
For an integer $n \geq -1$ and an element $\alpha \in \SL_2(\Z)$, we have the $p$-adic zeta element
\[
{}_{c, d}z_{mp^{n+1}}^{(p)}(f_E, 1, 1, \alpha, S \cup \prim(pN)) \in H^1(K_{m,n}, V_{\Z_p}(f_E)(1)) 
\]
as in \cite[(8.1.3)]{Kat04}.
We denote this element simply by ${}_{c, d}z_{m,n}^{(p)}(\alpha)$.

These elements are related to the $L$-values as follows.
Let $S(f_E)$, $\per_{f_E}: S(f_E) \to V_{\C}(f_E)$, and $\delta(\alpha) = \delta(f_E, 1, \alpha) \in V_{\Q}(f_E)$ be as in \cite[(6.3)]{Kat04}, and $\exp_{f_E}^*: H^1(k_{m,n}, V_{\Q_p}(f_E)(1)) \to S(f_E) \otimes_{\Q} k_{m,n}$ as in \cite[(9.4)]{Kat04}.
Then by \cite[(6.6) and (9.7)]{Kat04}, we have $\exp_{f_E}^*(\loc({}_{c,d}z_{m,n}^{(p)}(\alpha))) \in S(f_E) \otimes_{\Q} K_{m,n}$ and, for any character $\psi$ of $G_{m,n}$,
\begin{align}\label{eq:168}
& \sum_{\tau \in G_{m,n}} \per_{f_E} (\tau(\exp_{f_E}^*(\loc({}_{c,d}z_{m,n}^{(p)}(\alpha)))))^{\sign(\psi)}\psi(\tau) \\
& \qquad \qquad = (c^2 - c\psi^{-1}(c))(d^2 - d\psi^{-1}(d)) L_{S \cup \prim(pN)}(E, \psi, 1) \delta(\alpha)^{\sign(\psi)}.
\end{align}

By \cite[(8.12)]{Kat04}, these elements are compatible with respect to the corestriction maps, yielding
\[
({}_{c, d}z_{m,n}^{(p)}(\alpha))_n \in \varprojlim_n H^1(K_{m,n}, V_{\Z_p}(f_E)(1)).
\]
We denote the image of this element under the corestriction maps $K_{m,n} \to K_n$ by
\[
{}_{c, d}\z^{(p)}(\alpha) \in \varprojlim_n H^1(K_{n}, V_{\Z_p}(f_E)(1)) \simeq H^1(\Q, V_{\Z_p}(f_E)(1) \otimes \RR).
\]

Let $Q(\RR)$ denote the fraction ring of $\RR$.
For $\gamma \in V_{\Q_p}(f_E)$, similarly as in \cite[(13.9)]{Kat04} and \cite[Definition A.1]{Del08}, define 
\[
\z_{\gamma}^{(p)} \in H^1(\Q, V_{\Z_p}(f_E)(1) \otimes \RR) \otimes_{\RR} Q(\RR)
\]
as follows.
Choose elements $\alpha_1, \alpha_2 \in \SL_2(\Z)$ such that $\delta(\alpha_1)^+ \neq 0, \delta(\alpha_2)^- \neq 0$.
Then there are unique elements $b_1, b_2 \in \Q_p$ such that $\gamma = b_1 \delta(\alpha_1) + b_2 \delta(\alpha_2)$.
Define
\[
\z_{\gamma}^{(p)} 
= \left[ (c^2 - c \sigma_c)(d^2 - d\sigma_d) \prod_{l \mid N, l \not\in S \cup \{p\}} (1 - a_l l^{-1}\sigma_l^{-1}) \right]^{-1} 
\left( b_1 \cdot {}_{c, d}\z^{(p)}(\alpha_1)^+ +  b_2 \cdot {}_{c, d}\z^{(p)}(\alpha_2)^- \right).
\]
Here the inverted factor is certainly a non-zero-divisor of $\RR$.
By \eqref{eq:122} and \eqref{eq:168}, for any character $\psi$ of $G_n$, we have
\begin{equation}\label{eq:171}
\sum_{\tau \in G_{n}} \per_{f_E} (\tau(\exp_{f_E}^*(\loc(\z_{\gamma}^{(p)})_n)))^{\sign(\psi)}\psi(\tau)
= L_{S \cup \{p\}}(E, \psi, 1) \gamma^{\sign(\psi)}.
\end{equation}
It follows that $\z_{\gamma}^{(p)}$ is independent of the choices of $\alpha_1, \alpha_2, c, d$.

We consider the integrality of $\z_{\gamma}^{(p)}$.
The key tool is \cite[(12.6)]{Kat04}, generalized in \cite[p.255, Key Claim]{Del08}.
The former treats $K = \Q$ and the latter treats $K = \Q(\mu_m)$, but the following assertion for general $K$ can be deduced from the latter.
Using similar notations as in \cite{Del08}, let 
\[
\ZZ^{\primitive} \subset H^1(\Q, V_{\Z_p}(f_E)(1) \otimes \RR) \otimes_{\RR} Q(\RR)
\]
be the submodule generated by $\z_{\gamma}^{(p)}$ for $\gamma \in V_{\Z_p}(f_E)$.
A submodule $\ZZ^{\imp} \subset H^1(\Q, V_{\Z_p}(f_E)(1) \otimes \RR)$ can be also defined as in \cite{Del08}.
We have  $ \ZZ^{\imp} \subset \ZZ^{\primitive}$ and, moreover, \cite[Key Claim]{Del08} shows that the quotient $\ZZ^{\primitive}/\ZZ^{\imp}$ is finite.
In particular, this implies
\[
\z_{\gamma}^{(p)} \in \ZZ^{\primitive} \subset H^1(\Q, V_{\Z_p}(f_E)(1) \otimes \RR) \otimes_{\Z_p} \Q_p
\simeq H^1(\Q, \T) \otimes_{\Z_p} \Q_p
\]
for $\gamma \in V_{\Z_p}(f_E)$.

In addition, suppose that $H^1(\Q, V_{\Z_p}(f_E)(1) \otimes \RR)$ is free over $\Lambda$.
Then similarly as in \cite[(13.14)]{Kat04}, the finiteness of $\ZZ^{\primitive}/\ZZ^{\imp}$ implies
\[
\z_{\gamma}^{(p)} \in \ZZ^{\primitive} \subset H^1(\Q, V_{\Z_p}(f_E)(1) \otimes \RR)
\]
for $\gamma \in V_{\Z_p}(f_E)$.
Recall that $H^0(K_0, E[p]) = 0$ implies that $H^1(\Q, \T)$ is free over $\Lambda$ as in Proposition \ref{prop:106}(1).
Hence, if $E[p]$ is irreducible and $H^0(K_0, E[p]) = 0$, then we have $\z_{\gamma}^{(p)} \in H^1(\Q, \T)$.

Therefore it remains to show that $\z_{\gamma}^{(p)}$ satisfies \eqref{eq:14} for some $\gamma \in V_{\Z_p}(f_E)$. 
Consider the commutative diagram of parings
\[
\xymatrix{
	H^1(E(\C), \C) \ar[d]_{\pi^*} \ar@{}[r]|-{\times} &
	H_1(E(\C), \Z) \ar[r] &
	\C \ar@{=}[d]\\
	H^1(X_1(N)(\C), \C) \ar@{}[r]|-{\times} &
	H_1(X_1(N)(\C), \Z) \ar[r] \ar[u]^{\pi_*} &
	\C
}
\]
as used in \cite[p. 52]{GV00}.
Let $\per_E(\omega_E) \in H^1(E(\C), \C)$ be the image of $\omega_E$ under the period map.
By the definition of the N\'{e}ron periods $\Omega^{\pm}$, the image of $\per_E(\omega_E)$ as a map $H_1(E(\C), \Z)^{\pm} \to \C$ is $\Omega^{\pm}\Z$.
Then the above diagram implies that there are elements $\ttilde{\gamma}^{\pm} \in H^1(X_1(N)(\C), \Z)^{\pm}$ such that $\pi^*(\per_E(\omega_E)) = \Omega^+ \ttilde{\gamma}^+ + \Omega^- \ttilde{\gamma}^-$.
Let $\gamma^{\pm} \in V_{\Z_p}(f_E)^{\pm}$ be their images.
Put $\z = \z_{\gamma}^{(p)}$ with $\gamma = \gamma^+ + \gamma^-$.
Then, by the compatibilities of the period maps and the dual exponential maps, \eqref{eq:171} implies
\[
\sum_{\tau \in G_{n}} \per_{E} (\tau(\exp_{E}^*(\loc(\z)_n)))^{\sign(\psi)}\psi(\tau)
= \frac{L_{S \cup \{p\}}(E, \psi, 1)}{\Omega^{\sign(\psi)}} \per_E(\omega_E)^{\sign(\psi)},
\]
namely
\[
\sum_{\tau \in G_{n}} \tau(\exp_{E}^*(\loc(\z)_n))\psi(\tau)
= \frac{L_{S \cup \{p\}}(E, \psi, 1)}{\Omega^{\sign(\psi)}}\omega_E.
\]
By the definition of $\exp_{\omega_E}^*$, this is equivalent to \eqref{eq:14}.
\end{proof}

\begin{rem}
Here are a couple of remarks about the final assertion of Theorem \ref{thm:941}.
\begin{enumerate}
\item It is known that $E[p]$ is necessarily irreducible in the supersingular case as in \cite[Remark 5.3(i)]{Kob03}.

\item Thanks to the weaker assumption $H^0(K_0, E[p]) = 0$ than Assumption \ref{ass:04}, the Beilinson-Kato elements sit in the integral Iwasawa cohomology groups for various $K$.
This remark will be necessary when we make use of the Euler system argument in the proof of Theorem \ref{thm:964}.
\end{enumerate}
\end{rem}

\begin{defn}\label{defn:854}
Let $(-,-)_{n}: \hat{E}(\m_{n}) \times H^1(k_{n}, T_pE) \to \Z_p$ be the (sum of) local Tate pairing.
This induces a natural surjective map $H^1(k_{n}, T_pE) \to \hat{E}(\m_{n})^*$.
Using this map, the Coleman maps whose sources are $\hat{E}(\m_{n})^*$ constructed in Section \ref{sec:32} will also be considered to have sources $H^1(k_{n}, T_pE)$.
Therefore we obtain a map
\[
\Cole^{\bullet}: H^1(\Q_p, \T) \to \RR
\]
for $\bullet \in \{ \emptyset, +, -, \sharp, \flat\}$ (here we suppose Assumption \ref{ass:04} in the ordinary case).
This is consistent with $\Cole^{\bullet}$ in \eqref{eq:19}.
\end{defn}

For $x \in \hat{E}(\m_{n})$, we define $P_{n,x}: H^1(k_{n}, T_pE) \to R_{n}$ by 
\[
P_{n,x}(z) = \sum_{\sigma\in G_{n}} (\sigma(x), z)_{n} \sigma
\]
as in \cite[\S 8.5]{Kob03}.
Then we have 
\begin{equation}\label{eq:15}
P_{n,x}(z) = \left( \sum_{\sigma \in G_{n}} \sigma(\log_{\hat{E}}(x))\sigma \right) \left( \sum_{\tau \in G_{n}} \tau(\exp_{\omega_E}^*(z))\tau^{-1} \right)
\end{equation}
by the same proof as in \cite[Proposition 8.25]{Kob03}.

Now we prove Theorem \ref{thm:77} in the ordinary case.

\begin{thm}\label{thm:75}
Suppose $p \nmid a_p$ and Assumption \ref{ass:04} hold.
Then we have $\LL_S(E/K_{\infty})^{\iota} = \Cole(\loc(\z))$.
\end{thm}

\begin{proof}
Let $\alpha, \beta$ be the roots of $t^2-a_pt+p = 0$ with $p \nmid \alpha$.
Recall that $\LL_S(E/K_{\infty}) = \LL_S(E/K_{\infty}, \alpha)$ by the very definition.
By \eqref{eq:71}, it is enough to show that
\begin{equation}\label{eq:123}
 \psi(\Cole(\loc(\z))) = e_p(\alpha, \psi) \tau_S(\psi) \frac{L_S(E, \psi^{-1}, 1)}{\Omega^{\sign(\psi)}}
\end{equation}
for any character $\psi$ of $G_{ \infty}$ of finite order.
For any $n \geq \max \{0, n_{\psi} \}$, by \eqref{eq:120} we have
\[
\Cole_n(\loc(\z)_n)
= \sum_{\sigma\in G_{n}} (\sigma(\dd_{n}'), \loc(\z)_n)\sigma
 = P_{n, \dd_n'}(\loc(\z)_n).
\]
Therefore \eqref{eq:14} and \eqref{eq:15} show
\begin{equation}\label{eq:121}
\psi(\Cole(\loc(\z)))
= \left( \sum_{\sigma \in G_{n}} \sigma(\log_{\hat{E}}(\dd_n')) \psi(\sigma) \right) \frac{L_{S \cup \{p\}}(E,\psi^{-1},1)}{\Omega^{\sign(\psi)}}.
\end{equation}

If $n_{\psi} \geq 0$, then letting $n = n_{\psi}$ in \eqref{eq:121} gives \eqref{eq:123} by Theorem \ref{thm:121}(3).
When $n_{\psi} = -1$, let $n = 0$ in \eqref{eq:121}.
We have by Theorem \ref{thm:121}(1)(3)
\begin{align}
\sum_{\sigma \in G_{0}} \sigma(\log_{\hat{E}}(d_{0}'))\psi(\sigma) 
&= \sum_{\sigma \in G_{-1}} \sigma(\log_{\hat{E}}(\Tr^{0}_{-1}(d_{0}'))) \psi(\sigma)\\
&= (1 - \alpha^{-1}\psi(p)) \sum_{\sigma \in G_{-1}} \sigma(\log_{\hat{E}}(d_{-1})) \psi(\sigma)\\
&= (1-\alpha^{-1}\psi(p)) (1-\beta^{-1}\psi(p)^{-1})^{-1} \tau_S(\psi).
\end{align}
Moreover, by \eqref{eq:122}
\begin{align}
L_{S \cup \{p\}}(E,\psi^{-1},1) 
&= (1-a_p p^{-1} \psi(p)^{-1} + p^{-1}\psi(p)^{-2})L_{S}(E,\psi^{-1},1)\\
&=(1-\alpha^{-1}\psi(p)^{-1})(1-\beta^{-1}\psi(p)^{-1})L_{S}(E,\psi^{-1},1).
\end{align}
These prove \eqref{eq:123}.
\end{proof}

In the rest of this section, we consider the case where $p \mid a_p$.
The discussion in the following will generalize the results in \cite{Spr12}, where $K = \Q$ and $S = \emptyset$.
Recall $B_n \in M_2(\Lambda)$ in Definition \ref{defn:688}.
A similar computation as in Theorem \ref{thm:75} yields the following.

\begin{thm}\label{thm:72}
Suppose $p \mid a_p$ holds.
For any character $\psi$ of $G_{\infty}$ of finite order, we have
\[
\psi(\Cole'(\loc(\z))) =
 	\begin{cases}
		\tau_S(\psi) \frac{L_{S}(E,\psi^{-1},1)}{\Omega^{\sign(\psi)}}(1, 0) \psi(\ttilde{B_{n_{\psi}}}^{\iota})^{-1} & (n_{\psi} \geq 0)\\
		\tau_S(\psi) \frac{L_{S}(E,\psi^{-1},1)}{\Omega^{\sign(\psi)}}(a_p - (\psi(p)+\psi(p)^{-1}), p-1) & (n_{\psi} = -1)
	\end{cases}
\]
in $\overline{\Q_p}^{\oplus 2}$.
\end{thm}

\begin{proof}
This is a generalization of \cite[Proposition 6.5]{Spr12}.
Note that, if $n_{\psi} \geq 0$, then $\psi(\gamma)^{p^{n_{\psi}}} \neq 1$ and \eqref{eq:151} show that $\psi(\ttilde{B_{n_{\psi}}}^{\iota})$ is certainly invertible.
For any $n \geq \max \{0, n_{\psi} \}$, we have by \eqref{eq:16}
\[
\Cole'_n(\loc(\z)_n) \ttilde{B_n}^{\iota} = (P_{n,\dd_{n}}(\loc(\z)_n), P_{n,\dd_{n-1}}(\loc(\z)_n)).
\]
Therefore, by \eqref{eq:14} and \eqref{eq:15},
\begin{equation}\label{eq:124}
\psi(\Cole'(\loc(\z))) \psi(\ttilde{B_n}^{\iota}) = \left( \sum_{\sigma \in G_{n}} \sigma(\log_{\hat{E}}(\dd_{n})) \psi(\sigma), \sum_{\sigma \in G_{n}} \sigma(\log_{\hat{E}}(\dd_{n-1})) \psi(\sigma) \right) \frac{L_{S \cup \{p\}}(E,\psi^{-1},1)}{\Omega^{\sign(\psi)}}.
\end{equation}

When $n_{\psi} \geq 0$, letting $n = n_{\psi}$ in \eqref{eq:124} gives the assertion by Theorem \ref{thm:116}(3).
When $n_{\psi} = -1$, let $n=0$ in \eqref{eq:124}.
By Theorem \ref{thm:116}(1)(3), we compute
\begin{align}
& \left( \sum_{\sigma \in G_{0}} \sigma(\log_{\hat{E}}(\dd_{0})) \psi(\sigma), \sum_{\sigma \in G_{0}} \sigma(\log_{\hat{E}}(\dd_{-1})) \psi(\sigma) \right)\\
 &= \left( \sum_{\sigma \in G_{-1}} \sigma(\log_{\hat{E}}(\Tr^{0}_{-1}(\dd_{0}))) \psi(\sigma), \sum_{\sigma \in G_{-1}} \sigma(\log_{\hat{E}}(\Tr^{0}_{-1}(\dd_{-1}))) \psi(\sigma) \right)\\
 &= \left( \sum_{\sigma \in G_{-1}} \sigma(\log_{\hat{E}}(\dd_{-1})) \psi(\sigma) \right) (a_p - (\psi(p)+\psi(p)^{-1}), p-1)\\
 &= (1-p^{-1}a_p \psi(p)^{-1} + p^{-1}\psi(p)^{-2})^{-1} \tau_S(\psi)(a_p - (\psi(p)+\psi(p)^{-1}), p-1).
 \end{align}
Now
\[
L_{S \cup \{p\}}(E, \psi^{-1}, 1) = (1-p^{-1}a_p\psi(p)^{-1}+p^{-1}\psi(p)^{-2})L_{S}(E, \psi^{-1},1)
\]
by \eqref{eq:122} implies the result.
\end{proof}

By comparing the result in Theorem \ref{thm:72} and the characterization \eqref{eq:126} of the $\pm$-$p$-adic $L$-functions, we can now prove Theorem \ref{thm:77} for $\bullet \in \{+, -\}$.
However, we omit the detail since we will prove the more general $p \mid a_p$ case in Theorem \ref{thm:71}.

We continue to suppose $p \mid a_p$.
Let $\alpha, \beta$ be the roots of $t^2-a_pt+p = 0$.
The following is introduced by \cite[Definition 6.8]{Spr12}.

\begin{defn}\label{defn:92}
We put 
\[
\Log'_{\alpha, \beta} = \lim_{n\to \infty} \left[ \ttilde{B_n} 
\begin{pmatrix} a_p & p \\ -1 & 0 \end{pmatrix}^{-(n+2)} \right] \begin{pmatrix} \alpha & \beta \\ -1 & -1 \end{pmatrix}
\]
and 
\[
\Log_{\alpha, \beta}  = \begin{pmatrix} \log_{\alpha}^{\sharp} & \log_{\beta}^{\sharp} \\ \log_{\alpha}^{\flat} & \log_{\beta}^{\flat}
 \end{pmatrix}
 = \begin{pmatrix} 0 & 1 \\ -1 & -a_p \end{pmatrix}\Log'_{\alpha, \beta}
\]
(see Remark \ref{rem:81} for this modification).
These are matrices in $M_2(\HH_{1, \Q_p(\alpha)}(\Gamma))$.
Equivalently, we may define
\[
\Log_{\alpha, \beta} = \lim_{n \to \infty} 
\left[ \begin{pmatrix} a_p & 1 \\ -N_1 & 0 \end{pmatrix} \begin{pmatrix} a_p & 1 \\ -N_2 & 0 \end{pmatrix} \dots
\begin{pmatrix} a_p & 1 \\ -N_n & 0 \end{pmatrix} \begin{pmatrix} a_p & 1 \\ -p & 0 \end{pmatrix}^{-(n+2)} \right] 
\begin{pmatrix} -1 & -1 \\ \beta & \alpha \end{pmatrix}
\]
as in \cite{Spr17}.
\end{defn}

The following is a generalization of \cite[Theorem 6.12]{Spr12}.

\begin{defn}\label{defn:73}
Define the $\sharp/\flat$-$p$-adic $L$-functions $\LL_S^{\sharp/\flat}(E/K_{\infty}) \in \RR \otimes \Q_p$ by the formula
\[
(\LL_S(E/K_{\infty}, \alpha), \LL_S(E/K_{\infty}, \beta)) 
= (\LL_S^{\sharp}(E/K_{\infty}), \LL_S^{\flat}(E/K_{\infty})) \Log_{\alpha, \beta}.
\]
Namely, we have
\[
\begin{cases}
\LL_S(E/K_{\infty}, \alpha) = \log_{\alpha}^{\sharp} \LL_S^{\sharp}(E/K_{\infty}) + \log_{\alpha}^{\flat} \LL_S^{\flat}(E/K_{\infty})\\
\LL_S(E/K_{\infty}, \beta) = \log_{\beta}^{\sharp} \LL_S^{\sharp}(E/K_{\infty}) + \log_{\beta}^{\flat} \LL_S^{\flat}(E/K_{\infty}).
\end{cases}
\]
The existence will be shown in Theorem \ref{thm:71}.
\end{defn}

\begin{rem}\label{rem:68}
Suppose $a_p = 0$ holds.
We show that this definition is compatible with our convention $(\sharp, \flat) = (-, +)$.
In fact, we have
\[
\lim_{n\to \infty} \left[ \ttilde{B_n} 
\begin{pmatrix} 0 & p \\ -1 & 0 \end{pmatrix}^{-(n+2)} \right] 
= \begin{pmatrix} - \log^- & 0 \\ 0 & -\log^+ \end{pmatrix}.
\]
Therefore
\begin{align}
(\LL_S^{-}(E/K_{\infty}), \LL_S^{+}(E/K_{\infty})) \Log_{\alpha, -\alpha}
&= (\LL_S^{-}(E/K_{\infty}), \LL_S^{+}(E/K_{\infty}))
\begin{pmatrix} 0 & 1 \\ -1 & 0 \end{pmatrix} \begin{pmatrix} - \log^- & 0 \\ 0 & -\log^+ \end{pmatrix} \begin{pmatrix} \alpha & -\alpha \\ -1 & -1 \end{pmatrix}\\
&= (\LL_S^+(E/K_{\infty}), \LL_S^-(E/K_{\infty})) \begin{pmatrix} \log^- & 0 \\ 0 & \log^+ \end{pmatrix} \begin{pmatrix} \alpha & -\alpha \\ 1 & 1 \end{pmatrix}\\
&= (\LL_S(E/K_{\infty}, \alpha), \LL_S(E/K_{\infty}, -\alpha)),
\end{align}
where the final equality follows from \eqref{eq:125}.
\end{rem}

The following (cf. \cite[Definition 6.1]{Spr12}) is Theorem \ref{thm:77} in the supersingular case.

\begin{thm}\label{thm:71}
Suppose $p \mid a_p$ holds.
We have $\Cole^{\sharp/\flat}(\loc(\z)) = \LL_S^{\sharp/\flat}(E/K_{\infty})^{\iota}$.
\end{thm}

\begin{proof}
Definitions \ref{defn:400} and \ref{defn:92} imply $\Cole(\loc(\z)) \Log_{\alpha, \beta} = \Cole'(\loc(\z)) \Log'_{\alpha, \beta}$.
Thus the assertion is equivalent to 
\[
(\LL_S(E/K_{\infty}, \alpha)^{\iota}, \LL_S(E/K_{\infty}, \beta)^{\iota}) = \Cole'(\loc(\z)) (\Log'_{\alpha, \beta})^{\iota},
\]
which we shall prove.
Let $\psi$ be a character of $G_{\infty}$ of finite order and we shall compare the evaluation by $\psi$ with \eqref{eq:71} (this suffices by \cite[Lemma 6.11]{Spr12}).
If $n_{\psi} \geq 0$, we have
\begin{equation}\label{eq:130}
\psi(\ttilde{B_{n_{\psi}}})^{-1} \psi(\Log'_{\alpha, \beta})
= \begin{pmatrix} a_p & p \\ -1 & 0 \end{pmatrix}^{-(n_{\psi}+2)} \begin{pmatrix} \alpha & \beta \\ -1 & -1 \end{pmatrix}
= \begin{pmatrix} \alpha^{-(1+n_{\psi})} & \beta^{-(1+n_{\psi})} \\ -\alpha^{-(2+n_{\psi})} & -\beta^{-(2+n_{\psi})}\end{pmatrix}
\end{equation}
(see \cite[Lemma 6.7]{Spr12} for the second equality).
Thus by Theorem \ref{thm:72}, we have
\[
\psi \left( \Cole'(\loc(\z)) (\Log'_{\alpha, \beta})^{\iota} \right) 
= \tau_S(\psi) \frac{L_S(E,\psi^{-1},1)}{\Omega^{\sign(\psi)}} \left( \alpha^{-(1+n_{\psi})}, \beta^{-(1+n_{\psi})} \right).
\]
If $n_{\psi} = -1$, similarly we have
\[
\psi(\Log'_{\alpha, \beta})
= \begin{pmatrix} a_p & p \\ -1 & 0 \end{pmatrix}^{-2} \begin{pmatrix} \alpha & \beta \\ -1 & -1 \end{pmatrix}
= \begin{pmatrix} \alpha^{-1} & \beta^{-1} \\ -\alpha^{-2} & -\beta^{-2}\end{pmatrix}
\]
and thus, by Theorem \ref{thm:72},
\[
\psi \left( \Cole'(\loc(\z)) (\Log'_{\alpha, \beta})^{\iota} \right) 
 = \tau_S(\psi) \frac{L_S(E,\psi^{-1},1)}{\Omega^{\sign(\psi)}} (a_p - (\psi(p)+\psi(p)^{-1}), p-1)
 \begin{pmatrix} \alpha^{-1} & \beta^{-1} \\ -\alpha^{-2} & -\beta^{-2} \end{pmatrix}.
\]
An easy computation shows
\[
(a_p - (\psi(p)+\psi(p)^{-1}), p-1)\begin{pmatrix} \alpha^{-1} \\ -\alpha^{-2} \end{pmatrix} = (1-\alpha^{-1}\psi(p))(1-\alpha^{-1}\psi(p)^{-1}) = e_p(\alpha, \psi)
 \]
 and a similar formula for $\beta$.
This completes the proof.
\end{proof}

Before closing this section, we state a non-vanishing results of the $p$-adic $L$-functions.
Recall that we have a decomposition $\RR \otimes \Q_p \simeq \bigoplus_{\chi} (\RR\otimes \Q_p)^{\chi}$ where $\chi$ runs over the characters of $G_0$ and each $(\RR\otimes \Q_p)^{\chi}$ is an integral domain.
An element $\xi$ of $\RR \otimes \Q_p$ is a non-zero-divisor if and only if $\xi^{\chi} \in (\RR \otimes \Q_p)^{\chi}$ is nonzero for any $\chi$.

\begin{prop}\label{prop:101}
If $p \nmid a_p$, then $\LL_S(E/K_{\infty})$ is a non-zero-divisor of $\RR \otimes \Q_p$.
If $p \mid a_p$, for each character $\chi$ of $G_0$, at least one of $\LL_S^{\sharp}(E/K_{\infty})^{\chi}, \LL_S^{\flat}(E/K_{\infty})^{\chi}$ is nonzero.
If $a_p = 0$, then both $\LL_S^{\pm}(E/K_{ \infty})$ are non-zero-divisors of $\RR \otimes \Q_p$.
\end{prop}

\begin{proof}
This proposition is shown by the result of Rohrlich \cite{Roh84} on non-vanishing of $L$-values.
See \cite[Proposition 6.14]{Spr12} (resp. \cite[Corollary 5.11]{Pol03}) for the $p \mid a_p$ (resp. $a_p = 0$) case.
\end{proof}

\section{Equivariant Main Conjecture}\label{sec:84}

The goal of this section is to prove Theorem \ref{thm:95} by applying the Euler system argument to the Beilinson-Kato elements.
As usual, suppose Assumption \ref{ass:04} holds in the ordinary case.

\subsection{Remarks on Equivariant Main Conjecture}
In this subsection, we collect several remarks around the equivariant main conjecture \eqref{eq:241}.

\subsubsection{Independence from $S$}
In Proposition \ref{prop:615} below, we shall show that the equivariant main conjecture \eqref{eq:241} is independent from $S$.
This is a generalization of Greenberg-Vatsal \cite[Theorem (1.5)]{GV00}, where the non-equivariant, ordinary case is treated.
The proof of Proposition \ref{prop:615} traces that of \cite{GV00}.
More concretely, Lemma \ref{lem:992} below, the exact sequence \eqref{eq:04}, and Lemma \ref{lem:a02} below, respectively, correspond to \cite[the formula in p. 25]{GV00}, \cite[Proposition (2.1)]{GV00}, and \cite[Proposition (2.4)]{GV00}.

\begin{defn}\label{defn:991}
For a prime number $l \neq p$ which is unramified in $K/\Q$, put
\begin{equation}\label{eq:140}
P_l = 1-a_l l^{-1}\sigma_l + \mathbf{1}_N(l) l^{-1}\sigma_l^2 \in \RR,
\end{equation}
where $\sigma_l$ is the $l$-th power Frobenius map.
\end{defn}

\begin{lem}\label{lem:992}
Let $S' \supset S$ be a finite set of prime numbers $\neq p$.
We have 
\[
\LL_{S'}(E/K_{\infty}, \alpha) = \left( \prod_{l \in S' \setminus S} (-\sigma_l^{-1} P_l) \right) \LL_{S}(E/K_{\infty}, \alpha)
\]
for any allowable root $\alpha$, and
\begin{equation}\label{eq:904}
\LL_{S'}^{\bullet}(E/K_{\infty}) = \left( \prod_{l \in S' \setminus S} (-\sigma_l^{-1} P_l) \right) \LL_{S}^{\bullet}(E/K_{\infty})
\end{equation}
for $\bullet \in \{\emptyset, +, -, \sharp, \flat \}$.

\end{lem}

\begin{proof}
The first formula follows from \eqref{eq:122}, \eqref{eq:91}, and \eqref{eq:71}.
Then, by Definition \ref{defn:73} in the $p \mid a_p$ case, the second formula follows.
\end{proof}

For a finitely generated $\RR$-module $X$, let $\FF(X) = \Fitt_{\RR}(X)$ be the initial Fitting ideal of $X$.
See Section \ref{sec:108} for properties of $\FF$.

\begin{lem}\label{lem:a02}
Let $l \neq p$ be a prime number.

(1) We have an isomorphism
\begin{equation}
H^1(\Q_l, \T) \simeq H^0(K_{\infty} \otimes \Q_l, T_pE)
\end{equation}
of $\RR$-modules.

(2) Suppose $l$ is unramified in $K/\Q$.
Then we have $\FF(H^0(K_{\infty} \otimes \Q_l, T_pE)) = (P_l)$.
\end{lem}

\begin{proof}
(1) This is shown by the argument in the proof of \cite[Proposition (2.4)]{GV00}.
\hidden{
Let $T = T_pE$.
Since $H^1(\Q_l, \T) \simeq H^1(K_{\infty} \otimes \Q_l, T^{\dual}(1))^{\dual}$ by the Tate duality, it is enough to show 
$H^1(K_{\infty, \lambda}, T^{\dual}(1))^{\dual} \simeq H^0(K_{\infty, \lambda}, T)$ for a place $\lambda$ of $K_{\infty}$ above $l$.
Let $L^{\ur}$ (resp. $L^{\tame}$) be the maximal unramified (resp. tame) extension of $K_{\infty, \lambda}$.
Then we have
\begin{align}
H^1(L^{\ur}, T^{\dual}(1)) 
& \simeq H^1(L^{\tame}/L^{\ur}, H^0(L^{\tame}, T^{\dual}(1))) & \text{the degree of $\overline{L}/L^{\tame}$ is relatively prime to $p$}\\
& \simeq H^1(L^{\tame}/L^{\ur}, H_0(L^{\tame}, T^{\dual}(1))) & \text{the same reason}\\
& \simeq H^1(L^{\tame}/L^{\ur}, H_0(L^{\ur}, T^{\dual}(1))) & \text{the extension $L^{\tame}/L^{\ur}$ is pro-cyclic}\\
& \simeq \Hom(\Z_p(1), H_0(L^{\ur}, T^{\dual}(1))) & \text{$\Gal(L^{\tame}/L^{\ur}) \simeq \Z_{l'}(1)$ as $\Gal(L^{\ur}/\Q_l)$-modules}\\
& \simeq H_0(L^{\ur}, T^{\dual}).
\end{align}
Hence $H^1(K_{\infty, \lambda}, T^{\dual}(1))^{\dual} \simeq H^0(K_{\infty, \lambda}, T)$.
}

(2) 
Let $\lambda$ be a place of $K_{\infty}$ above $l$, 
and let $L^{\ur}$ be the maximal unramified extension of $K_{\infty, \lambda}$.

First suppose $l$ is a good prime for $E$.
Let $\begin{pmatrix} x_1 & x_2 \\ x_3 & x_4 \end{pmatrix} \in \GL_2(\Z_p)$ be the presentation matrix of the action of the Frobenius $\sigma_l$ on $T_pE$ with respect to a basis of $T_pE$ over $\Z_p$.
Then we have an exact sequence
\[
0 \to \Z_p[[\Gal(L^{\ur}/\Q_l)]]^{\oplus 2} \overset{\times D}{\to} \Z_p[[\Gal(L^{\ur}/\Q_l)]]^{\oplus 2} \to T_pE \to 0
\]
where $D$ is the matrix $\begin{pmatrix} \sigma_l - x_1 & -x_2 \\ -x_3 & \sigma_l - x_4 \end{pmatrix}$.
Since the degree of the infinite extension $L^{\ur}/K_{\infty, \lambda}$ is relatively prime to $p$, taking $\Gal(L^{\ur}/K_{\infty, \lambda})$-coinvariant yields an exact sequence
\[
0 \to \Z_p[[\Gal(K_{\infty, \lambda}/\Q_l)]]^{\oplus 2} \overset{\times D}{\to} \Z_p[[\Gal(K_{\infty, \lambda}/\Q_l)]]^{\oplus 2} \to H_0(K_{\infty, \lambda}, T_pE)  \to 0.
\]
We also have $H^0(K_{\infty, \lambda}, T_pE) \simeq H_0(K_{\infty, \lambda}, T_pE)$.
Thus we obtain an exact sequence
\[
0 \to \RR^{\oplus 2} \overset{\times D}{\to} \RR^{\oplus 2} \to H^0(K_{\infty} \otimes \Q_l, T_pE) \to 0.
\]
Therefore $\FF(H^0(K_{\infty} \otimes \Q_l, T_pE)) = (\det(D)) = (P_l)$.

If $l$ is additive for $E$, then the assertion is trivial since $H^0(L^{\ur}, T_pE) = 0$ and $P_l = 1$.
Suppose that $l$ is multiplicative for $E$.
Then $H^0(L^{\ur}, T_pE)$ is a free $\Z_p$-module of rank one, on which $\sigma_l$ acts as $a_l l$.
Thus a similar (but simpler) computation as in the good case shows the assertion.
\end{proof}

\begin{prop}\label{prop:615}
For $\bullet \in \{\emptyset, +, -, \sharp, \flat \}$, the cotorsionness of $\Sel_S^{\bullet}(E/K_{\infty})$ and the equality \eqref{eq:241} are independent from $S$.
\end{prop}

\begin{proof}
By Proposition \ref{prop:945}, the cotorsionness is independent from $S$.
Under the cotorsionness, Proposition \ref{prop:945} yields an exact sequence
\[
0 \to \bigoplus_{l \in S' \setminus S} H^1(\Q_l, \T) \to \Sel_{S'}^{\bullet}(E/K_{\infty})^{\dual} \to \Sel_{S}^{\bullet}(E/K_{\infty})^{\dual} \to 0.
\] 
Then by Theorem \ref{thm:76} and Proposition \ref{prop:613}, we have
\[
\FF(\Sel_{S'}^{\bullet}(E/K_{\infty})^{\dual}) 
= \left( \prod_{l \in S' \setminus S} \FF(H^1(\Q_l, \T)) \right) \FF(\Sel_{S}^{\bullet}(E/K_{\infty})^{\dual}).
\]
For any $l \not \in S$, by Lemma \ref{lem:a02}, we obtain $\FF(H^1(\Q_l, \T)) = (P_l)$.
By Lemma \ref{lem:992}, this completes the proof.
\end{proof}

\subsubsection{Behavior of $p$-adic $L$-functions under $\iota$}
We study the behavior of our $p$-adic $L$-function when we apply the involution $\iota$, using the functional equations.
The behavior is well-known when $K = \Q$, $S = \emptyset$, and $p \nmid a_p$.
Moreover, when $a_p = 0$, an analogue of that for the $\pm$-$p$-adic $L$-function is given by Pollack \cite[Theorem 5.13]{Pol03}.
The following generalizations of them require harder computations.

\begin{defn}\label{defn:801}
For a prime number $l \neq p$, let $K_{(l)}$ be the inertia field of $l$ in $K/\Q$.
Put $K_{(l), \infty} = (K_{(l)})_{\infty}$, which is the inertia field of $l$ in $K_{\infty}/\Q$.
Let $\nu_{K, (l)} \in \RR$ denote the norm element of $\Gal(K_{\infty}/K_{(l), \infty})$.
We also put $\RR_{(l)} = \Z_p[[\Gal(K_{(l), \infty}/\Q)]]$, which is a quotient of $\RR$.

For a (possibly empty) subset $T$ of $S$, put $\nu_{K, (T)} = \prod_{l \in T} \nu_{K,(l)} \in \RR$.
Let $K_{(T)}$ be the intersection of $K_{(l)}$ for $l \in T$.
Put $K_{(T), \infty} = (K_{(T)})_{\infty}$ and $\RR_{(T)} = \Z_p[[\Gal(K_{(T), \infty}/\Q)]]$, which is a quotient of $\RR$.
Since $\nu_{K,(T)}$ is a multiple of the norm element of $\Gal(K_{\infty}/K_{(T),\infty})$, multiplying $\nu_{K, (T)}$ defines a map $\RR_{(T)} \to \RR$.
\end{defn}

Since the rational number $(1-l^{-1})/[K:K_{(l)}]$ is a $p$-adic integer, the coefficients in the next formula are $p$-adically integral.

\begin{prop}\label{prop:641}
Suppose $S \cap \prim(N) = \emptyset$ holds.
Then, for $\bullet \in \{ \emptyset, +, -, \sharp, \flat\}$, we have an equality
\[
\LL_S^{\bullet}(E/K_{\infty})^{\iota} = w_E \sigma_{-N} 
\sum_{T \subset S} \nu_{K, (T)}\left(\prod_{l \in T} \frac{1-l^{-1}}{[K:K_{(l)}]} (\sigma_l^{-1} - \sigma_l)\right) \LL_{S\setminus T}^{\bullet}(E/K_{(T), \infty}),
\]
where $w_E \in \{\pm1\}$ is the sign of the functional equation.
\end{prop}

\begin{proof}
By Definition \ref{defn:73} in the $p \mid a_p$ case, it is enough to show the same relation for $\LL_S(E/K_{\infty}, \alpha)$ for any allowable root $\alpha$.
We evaluate at arbitrary character $\psi$ of $G_{\infty}$ of finite order.
Recall that, for any Dirichlet character $\psi$ of conductor relatively prime to $N$, the functional equation says
\begin{equation}\label{eq:72}
\tau(\psi)L(E, \psi^{-1}, 1) = w_E \psi(-N) \tau(\psi^{-1})L(E, \psi, 1).
\end{equation}
By $S \cap \prim(N) = \emptyset$, for $l \in S$, we can directly check that
\begin{equation}\label{eq:54}
- \sigma_l P_l^{\iota} = (1-l^{-1})(\sigma_l^{-1} - \sigma_l) + (- \sigma_l^{-1} P_l)
\end{equation}
in $\RR_{(l)}$.

We shall compute
\begin{align}
&\psi(\LL_S(E/K_{\infty}, \alpha)^{\iota})
= \left(\prod_{l \in S, l \nmid m_{\psi}} \psi(-\sigma_l P_l^{\iota}) \right) e_p(\alpha, \psi^{-1}) \tau(\psi) \frac{L(E, \psi^{-1}, 1)}{\Omega^{\sign(\psi)}}\\
&= w_E \psi(-N) \left(\prod_{l \in S, l \nmid m_{\psi}} \left[(1-l^{-1}) \psi(\sigma_l^{-1} - \sigma_l) + \psi(-\sigma_l^{-1} P_l) \right]\right) e_p(\alpha, \psi) \tau(\psi^{-1}) \frac{L(E, \psi, 1)}{\Omega^{\sign(\psi)}}\\
&= w_E \psi(-N) \sum_{T \subset S \setminus \prim(m_{\psi})} \left(\prod_{l \in T} (1-l^{-1})\psi(\sigma_l^{-1} - \sigma_l)\right) 
\left(\prod_{l \in S \setminus T, l \nmid m_{\psi}} \psi(-\sigma_l^{-1} P_l) \right) 
e_p(\alpha, \psi) \tau(\psi^{-1}) \frac{L(E, \psi, 1)}{\Omega^{\sign(\psi)}}\\
&= w_E \psi(-N) \sum_{T \subset S \setminus \prim(m_{\psi})} \psi\left(\prod_{l \in T} \nu_{K, (l)} \frac{1-l^{-1}}{[K:K_{(l)}]} (\sigma_l^{-1} - \sigma_l)\right) 
\psi(\LL_{S \setminus T}(E/K_{(T), \infty}, \alpha))\\
&= \psi\left[w_E \sigma_{-N} \sum_{T \subset S \setminus \prim(m_{\psi})} \nu_{K, (T)}\left(\prod_{l \in T} \frac{1-l^{-1}}{[K:K_{(l)}]} (\sigma_l^{-1} - \sigma_l)\right) \LL_{S \setminus T}(E/K_{(T), \infty}, \alpha)\right].
\end{align}
Here, the first equality follows from Lemma \ref{lem:992};
the second follows from \eqref{eq:54} and \eqref{eq:72};
the third is an expansion of the product;
the fourth follows from Lemma \ref{lem:992} and $\psi(\nu_{K,(l)}) = [K:K_{(l)}]$ for $l \nmid m_{\psi}$.

In the final formula, we can replace the range of $T$ by $T \subset S$.
This is because, if $T \subset S$ and $T \cap \prim(m_{\psi}) \neq \emptyset$, then $\psi(\nu_{K, (T)}) = 0$.
This completes the proof.
\end{proof}

\begin{cor}\label{cor:913}
Suppose the conditions (a) and (f) in Theorem \ref{thm:95} hold.
Suppose that $S$ is the set of prime numbers which are ramified in $K/\Q$.
Then, for $\bullet \in \{ \emptyset, +, -, \sharp, \flat\}$, the elements $\LL_S^{\bullet}(E/K_{\infty})^{\iota}$ and $\LL_S^{\bullet}(E/K_{\infty})$ coincide up to a unit of $\RR$.
\end{cor}

\begin{proof}
For any $T \subset S$, Lemma \ref{lem:992} shows 
\[
\nu_{K, (T)} \left( \prod_{l \in T} (-\sigma_l^{-1} P_l) \right)\LL_{S \setminus T}^{\bullet}(E/K_{(T), \infty}) 
= \nu_{K, (T)} \LL_S^{\bullet}(E/K_{\infty}).
\]
Under the condition (a), by Lemma \ref{lem:a02}, $P_l$ is a unit of $\RR_{(T)}$ for any $l \in T$.
Therefore the above equality shows that $\nu_{K, (T)} \LL_{S \setminus T}^{\bullet}(E/K_{(T), \infty}) \in (\LL_S^{\bullet}(E/K_{\infty}))_{\RR}$.
Then we can apply Proposition \ref{prop:641} by (f) and obtain $\LL_S^{\bullet}(E/K_{\infty})^{\iota} \in (\LL_S^{\bullet}(E/K_{\infty}))_{\RR}$.
By taking the involution, the inverse divisibility also holds.
\end{proof}

\subsubsection{equivariant main conjecture without $p$-adic $L$-functions}

It is a common phenomenon in Iwasawa theory that we can formulate both the main conjecture {\it with} $p$-adic $L$-functions and the main conjecture {\it without} $p$-adic $L$-functions.
In the non-equivariant theory, the two formulations are often known to be equivalent.
In the situation of this paper, when $K = \Q$ and $S = \emptyset$, such equivalences are established in \cite[Theorem 17.4]{Kat04}, \cite[Theorem 7.4]{Kob03}, and \cite[Conjecture 7.21]{Spr12}.

The equivariant main conjecture \eqref{eq:241} in this paper is a conjecture {\it with} $p$-adic $L$-functions.
In Proposition \ref{prop:911} below, under certain conditions, we shall formulate a variant \eqref{eq:89} {\it without} $p$-adic $L$-functions, and show that \eqref{eq:241} and \eqref{eq:89} are equivalent.

Let $\FF(X) = \Fitt_{\RR}(X)$ be the Fitting ideal and recall the definition of $\varSF{n}$ in Theorem \ref{thm:614}.
The following lemma explains the definition of $\au^{\sharp}$ in Section \ref{sec:01}.

\begin{lem}\label{lem:933}
When $p \mid a_p$, we have $(\au^{\sharp})^{-1} = \varSF{1}(\Rnt/(a_p))$.
\end{lem}

\begin{proof}
We have $\pd_{\RR}(\Rnt) \leq 1$.
If $a_p = 0$, Theorem \ref{thm:614} (applied to $Y=0, P_1 = X$) shows
\[
\varSF{1}(\Rnt) = \FF(\Rnt)^{-1} = (\au^{\sharp})^{-1}.
\]
Suppose $a_p \neq 0$ holds.
Then applying Theorem \ref{thm:614} to the sequence $0 \to \Rnt \overset{a_p}{\to} \Rnt \to \Rnt/(a_p) \to 0$ shows $\varSF{1}(\Rnt/(a_p)) = (1) = (\au^{\sharp})^{-1}$.
This completes the proof.
\end{proof}

Recall the Beilinson-Kato element $\z$ in Theorem \ref{thm:941}.
In this section we write $\z = \z_S$ to clarify the choice of $S$.
We have $\z_S \in H^1(\Q, \T)$ if $E[p]$ is irreducible as a $\Gal(\overline{\Q}/\Q)$-module.

\begin{prop}\label{prop:711}
Suppose $\z_S \in H^1(\Q, \T)$ holds.
Let $\bullet \in \{ \emptyset, +, -, \sharp, \flat\}$ and suppose $\LL_S^{\bullet}(E/K_{\infty})$ is a non-zero-divisor of $\RR$.
Then we have
\[
\au^{\bullet} \FF \left(H^1_{/f}(\Q_p, \T)^{\bullet} / (\loc_{/f}^{\bullet}(\z_S))_{\RR}\right) = (\LL_S^{\bullet}(E/K_{\infty})^{\iota})
\]
as ideals of $\RR$.
\end{prop}

\begin{proof}
Recall the identification in Lemma \ref{lem:962}(1) and that we have $\Cole^{\bullet}: H_{/f}(\Q_p, \T)^{\bullet} \to \RR$ as in \eqref{eq:19}.
Then, by Theorem \ref{thm:77} and the assumption on $\LL_S^{\bullet}(E/K_{\infty})$, the $\RR$-submodule $(\loc_{/f}^{\bullet}(\z_S))_{\RR}$ of $H^1_{/f}(\Q_p, \T)^{\bullet}$ is a free $\RR$-module of rank one.
Hence, combining with Lemma \ref{lem:962}(2), we see that $\pd_{\RR} \left(H^1_{/f}(\Q_p, \T)^{\bullet} / (\loc_{/f}^{\bullet}(\z_S))_{\RR} \right) \leq 1$ and this module is torsion over $\Lambda$.

If $p \nmid a_p$, the assertion follows from Theorems \ref{thm:93}(1) and \ref{thm:77}.
Suppose $p \mid a_p$ holds.
By Theorems \ref{thm:93}(2) and \ref{thm:77}, we obtain a diagram with exact row and column
\[
\xymatrix{
	& & R_{-1} \ar@{^{(}->}[d] & & \\
	0 \ar[r] & H^1_{/f}(\Q_p, \T)^{\flat} / (\loc_{/f}^{\flat}(\z_S))_{\RR} \ar[r] & (\RR \oplus R_{-1}) / ((\LL_S^{\flat}(E/K_{\infty})^{\iota}, \ast))_{\RR} \ar[r] \ar@{->>}[d] & R_{-1} \ar[r] & 0 \\
	& & \RR/(\LL_S^{\flat}(E/K_{\infty})^{\iota}) & &
}
\]
where $\ast$ denotes an unspecified element of $R_{-1}$.
Hence, by Proposition \ref{prop:613}, we obtain
\[
\FF\left(H^1_{/f}(\Q_p, \T)^{\flat} / (\loc_{/f}^{\flat}(\z_S))_{\RR} \right) 
= \FF \left(\RR/(\LL_S^{\flat}(E/K_{\infty})^{\iota}) \right)
= (\LL_S^{\flat}(E/K_{\infty})^{\iota} ).
\]
Similarly, Theorems \ref{thm:93}(2) and \ref{thm:77} yield an exact sequence
\[
0 \to H^1_{/f}(\Q_p, \T)^{\sharp} / (\loc_{/f}^{\sharp}(\z_S))_{\RR} \to \RR/(\LL_S^{\sharp}(E/K_{\infty})^{\iota}) \to \Rnt/(a_p) \to 0.
\]
By Theorem \ref{thm:614} and Lemma \ref{lem:933}, we obtain
\begin{align}
\FF\left(H^1_{/f}(\Q_p, \T)^{\sharp} / (\loc_{/f}^{\sharp}(\z_S))_{\RR} \right) 
& = \FF \left(\RR/(\LL_S^{\sharp}(E/K_{\infty})^{\iota}) \right) \varSF{1} \left(\Rnt/(a_p) \right)\\
& = (\au^{\sharp})^{-1}(\LL_S^{\sharp}(E/K_{\infty})^{\iota}).
\end{align}
This completes the proof.
Note that, when $a_p = 0$, we do not have to invoke Theorem \ref{thm:614} but just Proposition \ref{prop:613} works enough.
\end{proof}

See Definitions \ref{defn:995} and \ref{defn:110} for the definitions of $\subset_{\fin}$ and $\sim_{\fin}$, and Definition \ref{defn:977} for the definition of $E^1(X)$.

\begin{prop}\label{prop:911}
Suppose the conditions (a) and (f) in Theorem \ref{thm:95} hold.
Suppose that $S$ is the set of prime numbers which are ramified in $K/\Q$ and that $\z_S \in H^1(\Q, \T)$.
Let $\bullet \in \{ \emptyset, +, -, \sharp, \flat\}$ and suppose $\Sel_S^{\bullet}(E/K_{\infty})$ is $\Lambda$-cotorsion.
Then the equality \eqref{eq:241} implies
\begin{equation}\label{eq:89}
\FF \left(\Sel^0(E/K_{ \infty})^{\dual}\right) \subset_{\fin} \FF \left(E^1(H^1(\Q, \T)/(\z_S)_{\RR}) \right).
\end{equation}
The converse is also true under the condition (e) in Theorem \ref{thm:95}.
\end{prop}

\begin{proof}
Since $H^1(\Q_l, \T) = 0$ for $l \in S$ by (a), the exact sequence \eqref{eq:02} yields an exact sequence
\[
0 \to H^1(\Q, \T)/(\z_S)_{\RR} \to H^1_{/f}(\Q_p, \T)^{\bullet}/(\loc_{/f}^{\bullet}(\z_S))_{\RR} \to \Sel_S^{\bullet}(E/K_{\infty})^{\dual} \to \Sel^0(E/K_{\infty})^{\dual} \to 0.
\]
Put $\II = \FF\left(\Sel^0(E/K_{ \infty})^{\dual}\right)$ and $\JJ = \FF\left(E^1(H^1(\Q, \T)/(\z_S)_{\RR})\right)$.
Then Proposition \ref{prop:112} implies
\begin{equation}\label{eq:67}
\FF \left(H^1_{/f}(\Q_p, \T)^{\bullet}/(\loc_{/f}^{\bullet}(\z_S))_{\RR} \right) \II \subset_{\fin} \FF(\Sel_S^{\bullet}(E/K_{ \infty})^{\dual}) \JJ.
\end{equation}
By Corollary \ref{cor:913} and Proposition \ref{prop:711}, it follows that
\begin{equation}\label{eq:85}
\LL_S^{\bullet}(E/K_{\infty}) \II \subset_{\fin} \au^{\bullet} \FF(\Sel_S^{\bullet}(E/K_{ \infty})^{\dual}) \JJ.
\end{equation}

If the equality \eqref{eq:241} holds, then \eqref{eq:85} implies $\II \subset_{\fin} \JJ$.
Conversely, suppose that $\II \subset_{\fin} \JJ$ holds and the condition (e) is true.
Then \eqref{eq:85} implies
\[
\au^{\bullet} \FF(\Sel_S^{\bullet}(E/K_{ \infty})^{\dual}) \II \subset_{\fin} \au^{\bullet} \FF(\Sel_S^{\bullet}(E/K_{ \infty})^{\dual}) \JJ \supset_{\fin} \LL_S^{\bullet}(E/K_{\infty}) \II.
\]
Since the condition (e) implies $\II \RR_{p\Lambda} = \RR_{p\Lambda}$ if $p \mid [K: \Q]$, the equality \eqref{eq:241} follows from Lemma \ref{lem:109}.
\end{proof}

\subsection{One Divisibility of Equivariant Main Conjecture}

The goal of this subsection is to prove Theorem \ref{thm:95}.
Using a similar proof as in Proposition \ref{prop:911}, in the final paragraph of this subsection, we will deduce Theorem \ref{thm:95} from the following.

\begin{thm}\label{thm:674}
Suppose the conditions (a) -- (d) in Theorem \ref{thm:95} hold.
Suppose that $S$ is the set of prime numbers which are ramified in $K/\Q$.
Then there is a non-zero-divisor $u \in \RR$ such that 
\begin{equation}\label{eq:68}
u\FF(\Sel^0(E/K_{ \infty})^{\dual}) \subset_{\fin} \FF\left(E^1(H^1(\Q, \T)/(\z_{S})_{\RR})\right).
\end{equation}
\end{thm}

To prove Theorem \ref{thm:674}, we first prove Theorem \ref{thm:964} below.
The proof of Theorem \ref{thm:964} is a (nearly direct) application of the theory of Euler, Kolyvagin, and Stark systems developed in \cite{BSS}, \cite{BS}, \cite{Sak}, though we have the task to check various hypotheses.
On the other hand, the assertion of Theorem \ref{thm:674} does not appear in those previous works.
Thus the deduction of Theorem \ref{thm:674} from Theorem \ref{thm:964} is a novel part of this paper.

Let $\kk$ be the quotient of $\RR$ by its Jacobson radical.
Then $\kk$ is the group ring over $\F_p$ of the Galois group of the maximal extension of $\Q$ contained in $K_0$ with degree prime to $p$.
We remark here that, precisely speaking, the results in \cite[\S\S 3--5]{BSS} and \cite{Sak} are stated only for {\it local} coefficient ring, while we will apply them to our {\it semilocal} ring $\RR$ (and its quotients).
This is harmless as we can decompose our ring $\RR$ into the product of local rings associated to characters of order prime to $p$.
We do not try to explain the precise formulations because the notation would be cumbersome.

Let $\FF_{\Lambda}$ be the Selmer structure on $\T$ defined by $H^1_{\FF_{\Lambda}}(\Q_l, \T) = H^1(\Q_l, \T)$ for all prime numbers $l$ including $l = p$ (see \cite[\S 6]{Sak}; our case is mentioned in \cite[Example 6.3]{Sak}).
For each integer $n \geq 0$, put $J_n = (p^n, \gamma^{p^n}-1) \subset \RR$, where $\gamma$ is the fixed generator of $\Gamma$, and put
\[
T_n = \T \otimes_{\RR} \RR/ J_n = T_pE \otimes_{\Z_p} (\Z_p/p^n)[\Gal(K_n/\Q)],
\]
which is a representation of $\Gal(\overline{\Q}/\Q)$ over $\RR/J_n$.
We also put $\overline{T} = \T \otimes_{\RR} \kk$.
By abuse of notation, let $\FF_{\Lambda}$ also denote the propagated Selmer structure on $T_n$ in the sense of \cite[Example 1.1.2]{MR04}.
Furthermore, $\FF_{\Lambda}$ also denotes the dual Selmer structure on $T_n^{\dual}(1)$ as in \cite[Definition 1.3.1]{MR04}.

For a prime number $l \neq p$, let $\Q_l^{\ur}$ be the maximal unramified extension of $\Q_l$.
If $X$ is a continuous $\Gal(\overline{\Q_l}/\Q_l)$-module, we define
\[
H^1_{\ur}(\Q_l, X) = \Ker(H^1(\Q_l, X) \to H^1(\Q_l^{\ur}, X)).
\]

\begin{lem}\label{lem:997}
The following are true.

(1) The four natural maps
\[
H^1(\Q_p, \T) \to H^1(\Q_p, T_pE \otimes R_n) \to H^1(\Q_p, T_n) \to H^1(\Q_p, T_n/p) \to H^1(\Q_p, \overline{T})
\]
are all surjective.

(2) Suppose the condition (b) in Theorem \ref{thm:95} holds.
For a prime number $l \neq p$ which is unramified in $K/\Q$, the four natural maps
\[
H^1_{\ur}(\Q_l, \T) \to H^1_{\ur}(\Q_l, T_pE \otimes R_n) \to H^1_{\ur}(\Q_l, T_n) \to H^1_{\ur}(\Q_l, T_n/p) \to H^1_{\ur}(\Q_l, \overline{T})
\]
are all surjective.
\end{lem}

\begin{proof}
(1) The assumption $H^0(K_0 \otimes \Q_p, E[p]) = 0$ and the Tate duality show $H^2(\Q_p, T_n) = 0$.
Taking the limit, we obtain $H^2(\Q_p, \T) = 0$.
Therefore, for any finitely generated $\RR$-module $X$, we have $H^2(\Q_p, T_pE \otimes X) = 0$.
This implies the claim.

(2) First we show that the four natural maps
\[
H^0(\Q_l \otimes K_{\infty}, \T) \to H^0(\Q_l \otimes K_{\infty}, T_pE \otimes R_n) \to H^0(\Q_l \otimes K_{\infty}, T_n) \to H^0(\Q_l \otimes K_{\infty}, T_n/p) \to H^0(\Q_l \otimes K_{\infty}, \overline{T})
\]
are all surjective.
Observe that we have $H^0(\Q_l \otimes K_{\infty}, \T) = H^0(\Q_l \otimes K_{\infty}, T_pE) \otimes \RR$, etc.
Hence the first and the fourth maps are surjective and, for the second and the third map, it is enough to show that the two maps
\[
H^0(\Q_l \otimes K_{\infty}, T_pE) \to H^0(\Q_l \otimes K_{\infty}, T_pE/p^n) \to H^0(\Q_l \otimes K_{\infty}, T_pE/p)
\]
are surjective.
This follows from (b).

Take a place $\lambda$ of $K_{\infty}$ above $l$.
Then we have
\[
H^1_{\ur}(\Q_l, \T) = \Ker(H^1(\Q_l, \T) \to H^1(K_{\infty, \lambda}, \T)) \simeq H^1(K_{\infty, \lambda}/\Q_l, H^0(K_{\infty, \lambda}, \T)),
\]
etc., where the last isomorphism is obtained by the inflation-restriction exact sequence.
Since $\Gal(K_{\infty, \lambda}/\Q_l)$ has $p$-cohomological dimension one, the above claim proves the lemma.
 \end{proof}

\begin{lem}\label{lem:910}
Suppose $S$ is the set of prime numbers which are ramified in $K/\Q$.
Suppose the conditions (a) and (b) in Theorem \ref{thm:95} hold.

(1) We have
\[
H^1_{\FF_{\Lambda}}(\Q_l, T_n) = 
	\begin{cases}
		H^1(\Q_p, T_n) & (l = p)\\
		H^1_{\ur}(\Q_l, T_n) & (l \not\in S \cup \{p\})\\
		0 & (l \in S)
	\end{cases}
\]
Moreover, it coincides with the Selmer structure propagated by the canonical Selmer structure $\FF_{\can}$ on $T_pE \otimes _{\Z_p} R_n$ (see \cite[\S 6.2]{BSS} or \cite[Example 3.4]{Sak} for the definition of $\FF_{\can}$).

(2) This Selmer structure $\FF_{\Lambda}$ on $T_n$ is cartesian (in the sense of \cite[Definition 3.8]{Sak}).

(3) This Selmer structure $\FF_{\Lambda}$ on $T_n$ has core rank one (in the sense of \cite[Definition 3.19]{Sak}).
\end{lem}

\begin{proof}
The statements (1)(2) are more or less explained in \cite[Example 5.3]{Sak}, but we give a detailed proof for convenience (the author thanks Ryotaro Sakamoto for providing the detail).

(1) For $l \in S$, the condition (a) is equivalent to $H^0(\Q_l, T_n) = 0$. 
Then the Tate duality shows $H^2(\Q_l, T_n) = 0$ and, in turn, the local Euler-Poincare characteristic formula shows $H^1(\Q_l, T_n) = 0$.
Thus the assertions for $l \in S$ are trivial.

For $l = p$, Lemma \ref{lem:997}(1) shows $H^1_{\FF_{\Lambda}}(\Q_p, T_n) = H^1(\Q_p, T_n) = H^1_{\FF_{\can}}(\Q_p, T_n)$.

Let $l \not \in S \cup \{p\}$.
By Lemma \ref{lem:997}(2) and $H^1_{\ur}(\Q_l, \T) = H^1(\Q_l, \T)$ \cite[Proposition B.3.4]{Rub00}, we have $H^1_{\FF_{\Lambda}}(\Q_l, T_n) = H^1_{\ur}(\Q_l, T_n)$.
Take a place $\lambda$ of $K_{\infty}$ above $l$.
Then the definitions of the local conditions yield the following diagram
\[
\xymatrix{
	0 \ar[r] &
	H^1_{\ur}(\Q_l, T_pE \otimes R_n) \ar[r]\ar[d] &
	H^1(\Q_l, T_pE \otimes R_n) \ar[r]\ar@{=}[d] &
	H^1(K_{\infty, \lambda}, T_pE \otimes R_n) \ar[d] \\
	0 \ar[r] &
	H^1_{\FF_{\can}}(\Q_l, T_pE \otimes R_n) \ar[r] &
	H^1(\Q_l, T_pE \otimes R_n) \ar[r] &
	H^1(K_{\infty, \lambda}, T_pE \otimes R_n \otimes \Q_p).
}
\]
with exact rows.
Since (b) implies that $H^1(\Q_l \otimes K_{\infty}, T_pE) \to H^1(\Q_l \otimes K_{\infty}, T_pE \otimes \Q_p)$ is injective, the right vertical arrow is injective.
Therefore the left vertical arrow is an equality.

(2) 
For $l \in S$, we have nothing to say more.
For $l = p$, Lemma \ref{lem:997}(1) shows that $H^1(\Q_p, \T) \to H^1(\Q_p, \overline{T})$ is surjective.
Thus the cartesian condition at $p$ is trivial.

Let $l \not \in S \cup \{p\}$.
Take injective homomorphisms $\kk \to R_n/p \to \RR/J_n$, which induces $\overline{T} \to T_n/p \to T_n$.
By Lemma \ref{lem:997}(2), the cartesian condition is equivalent to the injectivity of the induced map $H^1_{/ \ur}(\Q_l, \overline{T}) \to H^1_{/ \ur}(\Q_l, T_n)$.
Thus it is enough to show that the map $H^1(\Q_l \otimes K_{\infty}, \overline{T}) \to H^1(\Q_l \otimes K_{\infty}, T_n)$ is injective.
The map $H^1(\Q_l \otimes K_{\infty}, \overline{T}) \to H^1(\Q_l \otimes K_{\infty}, T_n/p)$ is clearly injective, and the map $ H^1(\Q_l \otimes K_{\infty}, T_n/p) \to  H^1(\Q_l \otimes K_{\infty}, T_n)$ is injective by (b).

(3) The definition of the core rank $\chi(\FF_{\Lambda})$ of $\FF_{\Lambda}$ on $T_n$ is given by
\[
\chi(\FF_{\Lambda}) = \dim_{\kk} H^1_{\FF_{\Lambda}}(\Q, \overline{T}) -  \dim_{\kk} H^1_{\FF_{\Lambda}}(\Q, \overline{T}^{\dual}(1)).
\]
Here, since $\kk$ is not necessarily a field but instead a product of fields, we understand $\dim_{\kk}$ as the vector of the ranks after decomposing into components.
The assertion means that $\chi(\FF_{\Lambda}) = 1$, the vector consisting of $1$ in every component.
By \cite[Proposition 2.3.5]{MR04} applied to $\overline{T}^{\dual}(1)$, we have
\begin{align}
\chi(\FF_{\Lambda}) = & \dim_{\kk} H^0(\Q, \overline{T}) -  \dim_{\kk} H^0(\Q, \overline{T}^{\dual}(1)) \\
&+ \sum_l \left( \dim_{\kk} H^0(\Q_l, \overline{T}^{\dual}(1)) - \dim_{\kk} H^1_{\FF_{\Lambda}}(\Q_l, \overline{T}^{\dual}(1)) \right) + \dim_{\kk} H^0(\R, \overline{T}^{\dual}(1)),
\end{align}
where $l$ runs over all prime numbers.

We know $H^0(\Q, \overline{T}) = 0, H^0(\Q, \overline{T}^{\dual}(1)) = 0$ by assumption.
For $l =p$, we have $H^0(\Q_p, \overline{T}^{\dual}(1)) = 0$ and also $H^1_{\FF_{\Lambda}}(\Q_p, \overline{T}^{\dual}(1)) = 0$ since $H^1_{\FF_{\Lambda}}(\Q_p, \overline{T}) = H^1(\Q_p, \overline{T})$.
For $l \not \in S \cup \{p\}$, since $H^1_{\FF_{\Lambda}}(\Q_l, \overline{T}) = H^1_{\ur}(\Q_l, \overline{T})$, we have $H^1_{\FF_{\Lambda}}(\Q_p, \overline{T}^{\dual}(1)) = H^1_{\ur}(\Q_l, \overline{T}^{\dual}(1))$.
Since $\Q_l^{\ur}/\Q_l$ is a pro-cyclic extension, we have an exact sequence
\[
0 \to H^0(\Q_l, \overline{T}^{\dual}(1)) \to H^0(\Q_l ^{\ur}, \overline{T}^{\dual}(1)) \to H^0(\Q_l ^{\ur}, \overline{T}^{\dual}(1)) \to H^1_{\ur}(\Q_l, \overline{T}^{\dual}(1)) \to 0,
\]
where the middle map is defined as ``the Frobenius minus 1''.
This shows 
\[
\dim_{\kk} H^0(\Q_l, \overline{T}^{\dual}(1)) - \dim_{\kk} H^1_{\FF_{\Lambda}}(\Q_l, \overline{T}^{\dual}(1)) = 0.
\]
Finally, since $\dim_{\F_p} H^0(\R, E[p]) = 1$ while $\dim_{\F_p} E[p] = 2$, we have $\dim_{\kk} H^0(\R, \overline{T}^{\dual}(1)) = 1$.
This completes the proof of $\chi(\FF_{\Lambda}) = 1$.
\end{proof}

In general, if $X$ is a module over a commutative ring $R$ and $x \in X$, then $\ev^R_{x \in X}: \Hom_R(X, R) \to R$ denotes the evaluation map at $x$, and $\Image\left(\ev^R_{x \in X}\right)$ its image.

\begin{thm}\label{thm:964}
Suppose the conditions (a) -- (d) in Theorem \ref{thm:95} hold.
Suppose that $S$ is the set of prime numbers which are ramified in $K/\Q$.
Then there is an element $u \in \RR$ such that
\begin{equation}\label{eq:76}
u \Fitt_{\RR/J_n}(H^1_{\FF_{\Lambda}}(\Q, T_n^{\dual}(1))^{\dual}) = \Image \left(\ev^{\RR/J_n}_{\overline{\z_S} \in H^1(\Q, T_n)}\right)
\end{equation}
for any $n \geq 0$ (we do not claim here that $u$ is a non-zero-divisor).
Here, $\overline{\z_S} \in H^1(\Q, T_n)$ denotes the image of $\z_S \in H^1(\Q, \T)$.
\end{thm}

\begin{proof}
We apply the results in \cite{BSS}, where very general Galois representations are treated, to our case.
To ease the notation, we write $\z = \z_S$.

\step{1}{Euler system of the Beilinson-Kato elements.}

Let $K, \RR, T$ in \cite[\S 6]{BSS} correspond to our $\Q, R_0, T_pE \otimes R_0$, respectively.
Let $\PP$ be the set of prime numbers $l$ such that $l \nmid pN$, $l \not\in S$, and $l \equiv 1 \mod p$.
Let $\N(\PP)$ be the set of square-free products of $l \in \PP$ (by convention, $1 \in \N(\PP))$.
For each $\rr \in \N(\PP)$, let $\Q(\rr)$ be the maximal $p$-extension of $\Q$ contained in $\Q(\mu_{\rr})$.
Now we take $\K$ in \cite[\S 6]{BSS} as the composite of $\Q(\rr)$ for $\rr \in \N(\PP)$ and the cyclotomic $\Z_p$-extension of $\Q$.
Then \cite[Hypothesis 6.1]{BSS} holds by our assumption $H^0(K_0, E[p]) = 0$.
Also \cite[Hypothesis 6.7]{BSS} holds (by an appropriate choice of $S$).

For $\rr \in \N(\PP)$, we have $H^0(\Q(\rr)K_0, E[p]) = H^0(K_0, E[p]) = 0$.
Thus by Theorem \ref{thm:941}, we have an element 
\[
\z_{\rr} \in H^1(\Q, \T \otimes_{\Z_p} \Z_p[\Gal(\Q(\rr)/\Q)]) =  H^1(\Q(\rr), \T).
\]
Note that $\z = \z_1$ by definition.
By the norm relation of the Beilinson-Kato elements \cite[Proposition 8.12]{Kat04}, this system $(\z_{\rr})_{\rr}$ constitutes an Euler system of rank one, in the sense of \cite[Definition 6.4]{BSS}.
We denote by $\ES(\T)$ the module of Euler systems for $\T$; we have $(\z_{\rr})_{\rr} \in \ES(\T)$.

\step{2}{From Euler system to Kolyvagin system.}

Let $\KS(\T) = \varprojlim_n \KS(T_n)$ denote the module of Kolyvagin systems for $\T$ of rank one (\cite[\S5.1]{BSS}).
Here we always equip $T_n$ with the Selmer structure $\FF_{\Lambda}$, which coincides with $\FF_{\can}$ due to Lemma \ref{lem:910}(1).
Since \cite[Hypothesis 6.11]{BSS} holds, by \cite[Theorem 6.12]{BSS} (see also Corollaries 6.13 and 6.18(ii) therein), we have the Kolyvagin derivative homomorphism
\[
\DD: \ES(\T) \to \KS(\T).
\]

\step{3}{From Kolyvagin system to Stark system.}

Observe that, since the Galois representation  $\Gal(\overline{\Q}/\Q) \to \Aut(E[p^{\infty}]) \simeq \GL_2(\Z_p)$ is surjective by (c), the restriction $\Gal(\overline{\Q}/\Q(\mu_{p^{\infty}})) \to \SL_2(\Z_p)$ is also surjective.
Because $K/\Q$ is abelian and the group $\SL_2(\Z_p)$ is perfect under $p \geq 5$, it follows that the restriction $\Gal(\overline{\Q}/K_{\infty}) \to \SL_2(\Z_p)$ is also surjective.

In order to apply \cite[Theorem 5.25]{BSS}, we check the assumptions.
By the above observation, \cite[Hypothesis 4.7]{BSS} (and simultaneously \cite[Hypothesis 3.12]{Sak}) is true.
Then \cite[Proposition 3.22]{Sak} gives us a core vertex (in the sense of \cite[Definition 4.3]{Sak}; see \cite[Proposition 4.4]{Sak}).
Also we know that the Selmer structure $\FF_{\Lambda}$ on $T_n$ is cartesian by Lemma \ref{lem:910}(2).
Therefore \cite[Lemma 4.6]{Sak} implies that \cite[Hypothesis 4.2]{BSS} is true in our setting.

Let $\SSS(\T) = \varprojlim_n \SSS(T_n)$ be the module of Stark systems for $\T$.
This is a free $\RR$-module of rank one by \cite[Theorem 4.6(i)]{BSS} or \cite[Theorem 5.4(1)]{Sak}.
Now by \cite[Theorem 5.25]{BSS}, we have the regulator map
\[
\Reg: \SSS(\T) \overset{\sim}{\to} \KS(\T),
\]
which is isomorphic.

\step{4}{Application of the theory of Stark systems.}

Define the Stark system $\varepsilon = (\varepsilon_n)_n \in \SSS(\T)$ arising from the Beilinson-Kato Euler system by $\Reg(\varepsilon) = \DD((\z_{\rr})_{\rr})$.
By the definition of a Stark system, in particular this element involves
\[
(\varepsilon_n)_1 \in H^1(\Q, T_n).
\]
(Here the subscript $1$ is in the place of the index $\rr$.)
By the constructions of $\DD$ and $\Reg$, this element $(\varepsilon_n)_1$ coincides with $\overline{\z} \in H^1(\Q, T_n)$.

Let $\varepsilon^0 = (\varepsilon^0_n)_{n}$ be any basis of $\SSS(\T)$ as an $\RR$-module.
Then $\varepsilon^0_n$ is a basis of $\SSS(T_n)$ as an $\RR/J_n$-module.
Similarly as $\varepsilon$, we have an element $(\varepsilon^0_n)_1 \in H^1(\Q, T_n)$, and the main theorem of the theory of Stark systems \cite[Theorem 4.6(ii)]{BSS} implies
\[
\Fitt_{\RR/J_n}(H^1_{\FF_{\Lambda}}(\Q, T_n^{\dual}(1))^{\dual}) = \Image \left(\ev^{\RR/J_n}_{(\varepsilon^0_n)_1 \in H^1(\Q, T_n)}\right).
\]
Let $u \in \RR$ be the element such that $\varepsilon = u \varepsilon^0$.
Then \eqref{eq:76} follows, which completes the proof of Theorem \ref{thm:964}.
\end{proof}

\begin{lem}\label{lem:781}
Let $X$ be a finitely generated $\RR$-module which is free over $\Lambda$.
Then the natural map 
\[
\Hom_{\RR}(X, \RR) \to \Hom_{\RR/J_n}(X/J_n, \RR/J_n)
\]
is surjective.
\end{lem}

\begin{proof}
Put $I_n = (p^n, \gamma^{p^n}-1) \subset \Lambda$ so that $\RR/J_n \simeq \RR \otimes_{\Lambda} \Lambda/I_n$.
Since $\RR = \Lambda[G_0]$, we can construct an isomorphism $\Hom_{\Lambda}(\RR, \Lambda) \simeq \RR$ similarly as \eqref{eq:96}.
Applying $(-) \otimes_{\Lambda} \Lambda/I_n$ to this isomorphism gives an isomorphism $\Hom_{\Lambda/I_n}(\RR/J_n, \Lambda/I_n) \simeq \RR/J_n$.
Then we obtain a natural commutative diagram
\[
\xymatrix{
	\Hom_{\RR}(X, \RR) \ar[r] \ar[d]_{\vsim} &
	\Hom_{\RR/J_n}(X/J_n, \RR/J_n) \ar[d]^{\vsim} \\
	\Hom_{\RR}(X, \Hom_{\Lambda}(\RR, \Lambda)) \ar[r] \ar[d]_{\vsim} &
	\Hom_{\RR/J_n}(X/J_n, \Hom_{\Lambda/I_n}(\RR/J_n, \Lambda/I_n)) \ar[d]^{\vsim}\\
	\Hom_{\Lambda}(\RR \otimes_{\RR} X, \Lambda) \ar[r] \ar[d]_{\vsim} &
	\Hom_{\Lambda/I_n}(\RR/J_n \otimes_{\RR/J_n} X/J_n, \Lambda/I_n) \ar[d]^{\vsim}\\
	\Hom_{\Lambda}(X, \Lambda) \ar[r] &
	\Hom_{\Lambda/I_n}(X/J_n, \Lambda/I_n)\\
}
\]
Since $X$ is free over $\Lambda$, the bottom horizontal arrow is surjective.
Therefore the top horizontal arrow is also surjective.
\end{proof}

\begin{proof}[Proof of Theorem \ref{thm:674}]
We compute the both sides of \eqref{eq:76}, applying several arguments of \cite[\S\S 5, 6]{Sak} (in particular \cite[Proposition 6.11]{Sak}, where a non-equivariant situation is treated).
We continue to write $\z = \z_S$.

For the left hand side of \eqref{eq:76}, recall that $H^1_{\FF_{\Lambda}}(\Q, T_n^{\dual}(1)) \simeq H^1_{\FF_{\Lambda}}(\Q, \T^{\dual}(1))[J_n]$ (\cite[Lemma 3.14]{Sak}) and $H^1_{\FF_{\Lambda}}(\Q, \T^{\dual}(1)) = \Sel^0(E/K_{ \infty})$ by definition.
Hence we have 
\begin{equation}\label{eq:78}
\Fitt_{\RR/J_n}(H^1_{\FF_{\Lambda}}(\Q, T_n^{\dual}(1))^{\dual}) = \II (\RR/J_n),
\end{equation}
where we put $\II = \FF(\Sel^0(E/K_{ \infty})^{\dual})$.

We compute the right hand side of \eqref{eq:76}.
We use the exact sequence
\begin{equation}\label{eq:77}
0 \to H^2(\Q_{\Sigma}/\Q, \T)[J_n] \to H^1(\Q, \T)/J_n \to H^1_{\FF_{\Lambda}}(\Q, T_n)
\end{equation}
(\cite[Lemma 6.9]{Sak}).
When $n$ is enough large (which we assume in the following), the first term coincides with $H^2(\Q_{\Sigma}/\Q, \T)_{\fin}$, the maximal finite submodule of $H^2(\Q_{\Sigma}/\Q, \T)$.
In the sequence \eqref{eq:77}, the element $\z \bmod J_n \in H^1(\Q, \T)/J_n$ goes to $\overline{\z}$ in the final module.
Since $\RR/J_n$ is a zero-dimensional Gorenstein ring, we thus have
\[
\Image \left(\ev^{\RR/J_n}_{\overline{\z} \in H^1(\Q, T_n)}\right)
 = \{ h(\z \bmod J_n) \in \RR/J_n \mid h \in \Hom_{\RR/J_n}(H^1(\Q, \T)/J_n, \RR/J_n), h |_{H^2(\Q_{\Sigma}/\Q, \T)_{\fin}} = 0 \}.
 \]
Let $\ZZ \subset \RR$ be the annihilator ideal of $H^2(\Q_{\Sigma}/\Q, \T)_{\fin}$.
Then it follows that
\begin{equation}\label{eq:81}
\ZZ \Image\left(\ev^{\RR/J_n}_{\z \bmod J_n \in H^1(\Q, \T)/J_n} \right) 
\subset \Image \left(\ev^{\RR/J_n}_{\overline{\z} \in H^1(\Q, T_n)} \right) 
\subset \Image \left(\ev^{\RR/J_n}_{\z \bmod J_n \in H^1(\Q, \T)/J_n} \right).
\end{equation}
Since $H^1(\Q, \T)$ is free over $\Lambda$ (Proposition \ref{prop:106}(1)), Lemma \ref{lem:781} shows that the natural map
\[
\Hom_{\RR}(H^1(\Q, \T), \RR) \to \Hom_{\RR/J_n}(H^1(\Q, \T)/J_n, \RR/J_n)
\]
is surjective.
It follows that
\begin{equation}\label{eq:80}
\Image \left(\ev^{\RR/J_n}_{\z \bmod J_n \in H^1(\Q, \T)/J_n}\right) = \JJ (\RR/J_n),
\end{equation}
where we put $\JJ = \Image\left(\ev^{\RR}_{\z \in H^1(\Q, \T)}\right)$.

Now \eqref{eq:76}, \eqref{eq:78}, \eqref{eq:81}, and \eqref{eq:80} imply
\[
\ZZ \JJ (\RR/J_n) \subset u\II (\RR/J_n) \subset \JJ (\RR/J_n),
\]
namely
\[
\ZZ \JJ + J_n \subset u\II +J_n \subset \JJ +J_n
\]
as ideals of $\RR$.
Since $n$ is arbitrarily large, we obtain $\ZZ \JJ \subset u\II \subset \JJ$ and in particular $u\II \subset_{\fin} \JJ$.

By Theorem \ref{thm:77} and Proposition \ref{prop:101}, $\JJ$ contains a non-zero-divisor.
It follows that $u$ is a non-zero-divisor.
The trivial exact sequence $0 \to (\z)_{\RR} \to H^1(\Q, \T) \to H^1(\Q, \T)/(\z)_{\RR} \to 0$ induces an exact sequence
\[
0 \to \Hom_{\RR}(H^1(\Q, \T), \RR) \to \Hom_{\RR}((\z)_{\RR}, \RR) \to E^1(H^1(\Q, \T)/(\z)_{\RR}) \to 0.
\]
Since the evaluation map $\ev_{\z}: \Hom_{\RR}((\z)_{\RR}, \RR) \to \RR$ is an isomorphism, the definition of Fitting ideals yields 
\begin{equation}\label{eq:82}
\JJ = \FF(E^1(H^1(\Q, \T)/(\z)_{\RR})).
\end{equation}
This completes the proof of Theorem \ref{thm:674}.
\end{proof}

Now we can finish the proof of Theorem \ref{thm:95}.

\begin{proof}[Proof of Theorem \ref{thm:95}]
By (the proof of) Proposition \ref{prop:615}, we may and do suppose that $S$ is the set of prime numbers which are ramified in $K/\Q$.
Put $\II = \FF(\Sel^0(E/K_{ \infty})^{\dual})$ and $\JJ = \FF(E^1(H^1(\Q, \T)/(\z_S)_{\RR}))$.
By Theorem \ref{thm:674}, there is a non-zero-divisor $u \in \RR$ such that $u \II \subset_{\fin} \JJ$.
Then, as in the proof of Proposition \ref{prop:911}, we have
\[
u \au^{\bullet} \FF(\Sel_S^{\bullet}(E/K_{ \infty})^{\dual}) \II \subset_{\fin} \au^{\bullet} \FF(\Sel_S^{\bullet}(E/K_{ \infty})^{\dual}) \JJ \supset_{\fin} \LL_S^{\bullet}(E/K_{\infty}) \II.
\]
Thus the condition (e) and Lemma \ref{lem:109} show $u \au^{\bullet} \FF(\Sel_S^{\bullet}(E/K_{ \infty})^{\dual}) = (\LL_S^{\bullet}(E/K_{\infty}))$.
This completes the proof of Theorem \ref{thm:95}.
\end{proof}

\section{Application to Mazur-Tate Conjecture}\label{sec:07}

In this section, we prove Theorem \ref{thm:203}, which roughly states that one divisibility of the equivariant main conjecture implies Mazur-Tate conjecture.
This is a vast generalization of the work of C.-H. Kim-Kurihara \cite[Theorem 1.14]{KK}, where $m = 1$ (and the $\Delta$-invariant part) is treated.

First we give the definition of the Mazur-Tate elements \cite[(1.1)--(1.2)]{MT87}.
Let $f_E$ be the newform of weight 2 associated to $E$ by the modularity theorem.
Recall that $\Omega^{\pm}$ are the real and imaginary N{\'{e}}ron periods of $E$.
 For $r \in \Q$, define $[r]^{\pm} \in \Q$ called the modular symbols by
 \[
 2\pi \int_0^{\infty} f_E(r+iy)dy = [r]^+ \Omega^+ + [r]^- \Omega^-.
 \]
 For a positive integer $M$, the Mazur-Tate element $\theta_M$ is defined by
 \[
 \theta_M = \sum_{a \in (\Z/M\Z)^{\times}} \left( \left[ a/M \right]^+ + \left[ a/M \right]^- \right)\sigma_a^{-1} \in \Q[\Gal(\Q(\mu_M)/\Q)].
 \]

\begin{rem}
As in Remark \ref{rem:812}, the convention of the Mazur-Tate elements in this paper and that in \cite{MT87} differ by the involution $\iota$.
However, this difference does not affect Conjecture \ref{conj:202} thanks to the functional equation \cite[(1.6.2)]{MT87} (which follows from \eqref{eq:72}).
\end{rem}

\begin{rem}\label{rem:a03}
Suppose $l^2 \nmid N$ holds for any prime divisor $l$ of $M$ (for example, $(M,N) = 1$ suffices).
Suppose also that $\ttilde{E}(\F_p)[p] = 0$ holds and $E[p]$ is irreducible as a $\Gal(\overline{\Q}/\Q)$-module.
Then we have $\theta_M \in \Z_p[\Gal(\Q(\mu_M)/\Q)]$.
This follows from \cite[Theorem 3.5]{Man72}, but this fact is unnecessary in this paper.
\end{rem}

\hidden{
\begin{proof}
Let $c$ be the Manin constant of $E$, namely, let $\pi: X_0(N) \to E$ be a modular parametrization and $\pi^* \omega_E = c f(q)\frac{dq}{q}$.
We show the following claim:
\begin{quote}
Let $r \in \Q$ satisfy the property that $l^2 \nmid N$ for any prime number $l$ with $\ord_l(r) < 0$.
Then we have $[r]^{\pm} \in \frac{1}{c(q+1-a_q)} \Z$ for any prime number $q \nmid N$.
\end{quote}
The irreducibility of $E[p]$ and $p \nmid N$ show $p \nmid c$, by the work of Mazur.
Moreover, $\ttilde{E}(\F_p)[p] = 0$ is equivalent to $p \nmid (p+1-a_p)$.
Thus the above claim implies the lemma.

We prove the claim.
The relation $T_qf_E = a_qf_E$ implies
\[
a_q \int_r^{i \infty} f_E(z)dz = \sum_{b=0}^{q-1} \int_{(r+b)/q}^{i \infty} f_E(z)dz + \int_{qr}^{i \infty} f_E(z)dz.
\]
Therefore
\[
(q+1-a_q) \int_r^{i \infty} f_E(z)dz = \sum_{b=0}^{q-1} \int_r^{(r+b)/q} f_E(z)dz + \int_r^{qr} f_E(z)dz.
\]
Multiplying $-2\pi i$ yields
\[
(q+1-a_q) 2\pi \int_0^{\infty} f_E(r+iy)dy = -\sum_{b=0}^{q-1} \int_r^{(r+b)/q} f_E(q)\frac{dq}{q} - \int_r^{qr} f_E(q)\frac{dq}{q}.
\]
By our assumption on $r$, the rational numbers $r, (r+b)/q, qr$ are contained in the same $\Gamma_0(N)$-orbit (see Manin).
More precisely, this holds since $r = u/v\delta$ with $(u, v\delta) = 1, \delta \mid N, (v, N/\delta) = 1$ and $q \equiv 1 \mod (\delta, N/\delta)$.
Hence we obtain
\[
c(q+1-a_q) \left( [r]^+ \Omega^+ + [r]^- \Omega^- \right) \in \Z \Omega^+ + \Z \Omega^-.
\]
This completes the proof.
\end{proof}
}

We review the properties of Mazur-Tate elements \cite[(1.3)--(1.4)]{MT87}.
For a {\it primitive} Dirichlet character $\psi$ modulo $M$, we have
\begin{equation}\label{eq:20}
\psi(\theta_M) = \tau(\psi^{-1}) \frac{L(E, \psi, 1)}{\Omega^{\sign(\psi)}}.
\end{equation}
For a prime divisor $l$ of $M$, let $z^{M}_{M/l}$ be the projection map $\Q[\Gal(\Q(\mu_M)/\Q)] \to \Q[\Gal(\Q(\mu_{M/l})/\Q)]$.
Let $\nu^{M}_{M/l}$ be the map $\Q[\Gal(\Q(\mu_{M/l})/\Q)] \to \Q[\Gal(\Q(\mu_M)/\Q)]$ induced by the multiplication by the norm element of $\Gal(\Q(\mu_M)/\Q(\mu_{M/l}))$.
Then we have
\begin{equation}\label{eq:21}
z^{M}_{M/l}(\theta_M) =
	\begin{cases}
		(a_l - \sigma_l^{-1} - \mathbf{1}_N(l) \sigma_l) \theta_{M/l} & (l^2 \nmid M) \\
		a_l \theta_{M/l} - \mathbf{1}_N(l) \nu^{M/l}_{M/l^2}(\theta_{M/l^2}) & (l^2 \mid M).
	\end{cases}
\end{equation}

In the rest of this section, we fix a positive integer $m$ which is relatively prime to $pN$.
Also suppose Assumption \ref{ass:04} holds for $K = \Q(\mu_m)$ in the ordinary case.
Recall the notations introduced just before Lemma \ref{lem:946}: $K_{m, n} = \Q(\mu_{mp^{n+1}}), R_{m,n} = \Z_p[\Gal(K_{m,n}/\Q)]$, and $\RR_m = \Z_p[[\Gal(K_{m,\infty}/\Q)]]$.
Also put $\theta_{m,n} = \theta_{mp^{n+1}}$.

The rough idea to prove Theorem \ref{thm:203} can be illustrated by the following diagram (the dotted lines represent certain connections).
\[
\xymatrix{
	\Sel_S^{\bullet}(E/K_{m, \infty}) \ar@{.}[r]^-{\textrm{EMC}} \ar@{.}[d] &
	\LL_S^{\bullet}(E/K_{m, \infty}) \ar@{.}[d]\\
	\Sel(E/K_{m, n}) \ar@{.}[r]_-{\textrm{MTC}} &
	\theta_{m,n}
}
\]
Here, EMC and MTC stand for, respectively, the equivariant main conjecture \eqref{eq:241} and the Mazur-Tate conjecture (Conjecture \ref{conj:202}).
We will establish the left side (algebraic) connection in Subsection \ref{subsec:821} and the right side (analytic) connection in Subsection \ref{subsec:875}.
As mentioned in Section \ref{sec:01}, both sides have difficulties.
Those results will prove Theorem \ref{thm:203} in Subsection \ref{subsec:874}.

\subsection{Algebraic Side}\label{subsec:821}

We begin with an elementary lemma.

\begin{lem}\label{lem:901}
Let $G$ be a finite abelian group and $\sigma \in G$ be an element.
Then $\sigma + 1$ is a unit as an element of $\F_p[G]$ if and only if the order of $\sigma$ is odd.
\end{lem}

\begin{proof}
Let $G'$ be the cyclic subgroup of $G$ generated by $\sigma$.
By the isomorphism $\F_p[G]/(\sigma+1) \simeq \F_p[G] \otimes_{\F_p[G']} \F_p[G']/(\sigma+1)$, we may assume that $G = G'$.
If we denote by $h$ the order of $\sigma$, we have isomorphisms 
\[
\F_p[G']/(\sigma+1) \simeq \F_p[x]/(x^h-1, x+1) \simeq \F_p/((-1)^h-1)
\]
where $x$ denotes an indeterminate corresponding to $\sigma$ and the last isomorphism sends $x$ to $-1$.
The final term vanishes if and only if $h$ is odd.
\end{proof}

\begin{defn}\label{defn:996}
As a special case of Definition \ref{defn:801}, we introduce the following notation.
For a prime divisor $l$ of $m$ and an integer $n \geq -1$ or $n = \infty$, let $K_{m, (l), n}$ be the inertia field of $l$ in $K_{m,n}/\Q$.
More concretely, we have $K_{m, (l), n} = K_{m/l^{e_l}, n}$ with $e_l = \ord_l(m) \geq 1$.
Let $\nu_{m, (l)} = \nu_{\Q(\mu_m), (l)} \in \RR_m$ be the norm element of $\Gal(K_{m, \infty}/ K_{m, (l), \infty})$.
Put $\RR_{m, (l)} = \Z_p[[\Gal(K_{m, (l), \infty}/\Q)]]$, which is a quotient ring of $\RR_m$.

For a (possibly empty) subset $T$ of $\prim(m)$, put $\nu_{m,(T)} = \prod_{l \in T} \nu_{m,(l)} \in \RR_m$.
Let $m_{(T)}$ be the maximal divisor of $m$ such that $\prim(m_{(T)}) = \prim(m) \setminus T$.
Then $\Q(\mu_m)_{(T)} = \Q(\mu_{m_{(T)}})$ and $(\RR_m)_{(T)} = \RR_{m_{(T)}}$ in the notation in Definition \ref{defn:801}.
Moreover, $\nu_{m, (T)}$ is precisely the norm element of $\Gal(K_{m, \infty}/K_{m_{(T)}, \infty})$.
\end{defn}

Let $\FF$ be the (quasi-)Fitting invariant defined by $\FF(X) = \Fitt_{\RR_m}(X)$.
Then we have the shift $\SF{-1}$ by Theorem \ref{thm:a01}.
This $\SF{-1}$ is already used in \cite[\S 5.4]{Kata} in the study of ideal class groups, namely the Galois representation $\Z_p(1)$.
In the following Proposition \ref{prop:510}, we operate a similar computation on $T_pE$.
This corresponds to \cite[Lemma 5.13]{Kata}, but inevitably, the computation becomes harder.
Recall the element $P_l \in \RR_{m, (l)}$ defined in Definition \ref{defn:991}.

\begin{prop}\label{prop:510}
Let $l$ be a prime divisor of $m$.
Let $\begin{pmatrix} x_1 & x_2 \\ x_3 & x_4 \end{pmatrix} \in \GL_2(\Z_p)$ be the presentation matrix of the action of the Frobenius $\sigma_l$ on $T_pE$ with respect to a basis of $T_pE$ over $\Z_p$.

(1) We have
\[
\SF{-1}(H^0(K_{m, \infty} \otimes \Q_l, T_pE)) = \left(1, \nu_{m,(l)} P_l^{-1} (\sigma_l - x_1, x_2, x_3, \sigma_l - x_4, l-1) \right).
\]

(2) Suppose $l \equiv 1 \mod p$ holds.
We have $(\sigma_l - x_1, x_2, x_3, \sigma_l - x_4) \neq \RR_{m,(l)}$ as an ideal of $\RR_{m,(l)}$ if and only if one of the following holds.
\begin{itemize}
\item[(i)] $a_l \equiv 2 \mod p$ and $\sharp \ttilde{E}_l(\F_l)[p] = p^2$.
\item[(ii)] $a_l \equiv -2 \mod p$, the residue degree of $l$ in $\Q(\mu_m)/\Q$ is even, and $\sharp \ttilde{E}_l(\F_{l^2})[p] = p^2$.
\end{itemize}
\end{prop}

\begin{proof}
(1) 
Similarly as in the proof of Lemma \ref{lem:a02}(2), we obtain an exact sequence
\[
0 \to \RR_{m, (l)}^{\oplus 2} \overset{\times D}{\to} \RR_{m, (l)}^{\oplus 2} \to H^0(K_{m, (l), \infty} \otimes \Q_l, T_pE) \to 0,
\]
where $D = \begin{pmatrix} \sigma_l - x_1 & -x_2 \\ -x_3 & \sigma_l - x_4 \end{pmatrix}$.
By snake lemma, this sequence induces the upper exact sequence of the following diagram
\begin{equation}\label{eq:61}
 \xymatrix{
0 \ar[r] &
H^0(K_{m, (l), \infty} \otimes \Q_l, T_pE) \ar[r] \ar@{=}[d] &
(\RR_{m, (l)} / P_l)^{\oplus 2} \ar[r]^-{\times D} \ar@{^{(}-_>}[d]_{\nu_{m,(l)}} &
(\RR_{m, (l)}/P_l)^{\oplus 2} \\
0 \ar[r] &
H^0(K_{m, \infty} \otimes \Q_l, T_pE) \ar[r] &
(\RR_m / \overl{P_l})^{\oplus 2} \ar[r] &
Y \ar[r] &
0
}
\end{equation}
We shall construct the other parts of this diagram.
The left vertical equality of \eqref{eq:61} follows from the assumption that $l$ is good for $E$.
The upper row can be regarded as a sequence of $\RR_m$-modules on which $\Gal(K_{m, \infty}/K_{m, (l), \infty})$ acts trivially.
Then $\nu_{m,(l)}: \RR_{m, (l)} \to \RR_m$ is a homomorphism of $\RR_m$-modules.
Take a lift $\overl{\sigma_l} \in \RR_m$ of $\sigma_l$ and put $\overl{P_l} = 1 - a_l l^{-1} \overl{\sigma_l} + l^{-1} \overl{\sigma_l}^2 \in \RR_{m}$, which is a lift of $P_l$.
Then $\nu_{m, (l)}: \RR_{m, (l)}/P_l \to \RR_{m}/\overl{P_l}$ is well-defined and injective.
Now we can define $Y$ so that \eqref{eq:61} is a commutative diagram of $\RR_m$-modules with exact rows.

It is easy to see that $Y$ does not contain a non-trivial finite submodule.
Therefore Theorem \ref{thm:a01} and the lower sequence of \eqref{eq:61} imply
\[
\SF{-1}(H^0(K_{m, \infty} \otimes \Q_l, T_pE)) = \overl{P_l}^{-2} \FF(Y).
\]
It remains to compute $\FF(Y)$.
By $(\det(D)) = (P_l)$, a similar proof as in Lemma \ref{lem:53}(2) shows that the sequence
\[
(\RR_{m, (l)} / P_l)^{\oplus 2} \overset{\times \ttilde{D}}{\to} (\RR_{m, (l)} / P_l)^{\oplus 2} \overset{\times D}{\to} (\RR_{m, (l)} / P_l)^{\oplus 2}
\]
is exact, where $\ttilde{D}$ is the adjugate matrix of $D$.
Therefore the diagram \eqref{eq:61} shows that $Y$ fits into an exact sequence
\[
(\RR_m / \overl{P_l})^{\oplus 2} \overset{\times \nu_{m,(l)} \ttilde{D}}{\to} (\RR_m / \overl{P_l})^{\oplus 2} \to Y \to 0.
\]
Hence $Y$ admits a presentation
\[
\RR_m^{\oplus 4} \overset{\times \mathbb{D}}{\to} \RR_m^{\oplus 2} \to Y \to 0
\]
over $\RR_m$, where $\mathbb{D}$ denotes the $4 \times 2$ matrix with the scalar matrix $\overl{P_l}$ in the upper $2 \times 2$ and $\nu_{m,(l)} \ttilde{D}$ in the lower $2 \times 2$.
Consequently we have
\begin{align}
\FF(Y) &= (\overl{P_l}^2, \nu_{m,(l)} \overl{P_l}  (\sigma_l-x_1, x_2, x_3, \sigma_l-x_4), \nu_{m,(l)}^2 \overl{P_l})\\
&= (\overl{P_l}^2, \nu_{m,(l)} \overl{P_l} (\sigma_l-x_1, x_2, x_3, \sigma_l-x_4, l-1)),
\end{align}
where the last equality follows from $\nu_{m,(l)}^2 = l^{e_l-1}(l-1) \nu_{m, (l)}$ with $e_l = \ord_l(m)$ (if $l \geq 3$; in the $l = 2$ case, the verification is slightly different, but the result is unchanged).
This completes the proof of (1).

(2) First we claim that $(\sigma_l + 1) \neq \RR_{m,(l)}$ if and only if the residue degree of $l$ in $\Q(\mu_m)/\Q$ is even.
By Nakayama's lemma, $(\sigma_l + 1) \neq \RR_{m,(l)}$ is equivalent to $(\sigma_l + 1) \neq \F_p[\Gal(K_{m,(l),0}/\Q)]$.
Thus by Lemma \ref{lem:901}, this is equivalent to that the residue degree of $l$ in $K_{m,(l),0}/\Q$ is even.
Since $l$ splits completely in $K_{m,0}/K_{m,-1}$ and is totally ramified in $K_{m,0}/K_{m,(l),0}$, we obtain the claim.

In the following, the congruences are considered modulo $p$.
First suppose $(\sigma_l - x_1, x_2, x_3, \sigma_l - x_4) \neq \RR_{m,(l)}$.
Then $x_1 - x_4 \equiv x_2 \equiv x_3 \equiv 0$.
Since $x_1 x_4 - x_2 x_3 = l \equiv 1$ and $x_1 + x_4 = a_l$, we obtain $x_1 \equiv x_4 \equiv \pm 1$ and $a_l \equiv \pm 2$.
If $a_l \equiv 2$, then $\sigma_l$ acts on $E[p]$ trivially and we get (i).
Suppose $a_l \equiv -2$ holds.
Then $\sigma_l^2$ acts on $E[p]$ trivially.
Also by the assumption $(\sigma_l +1) \neq \RR_{m, (l)}$, the above claim shows (ii).

Conversely, suppose either (i) or (ii) holds.
If (i) holds, then the second condition means $x_1 \equiv x_4 \equiv 1$ and $x_2 \equiv x_3 \equiv 0$.
Thus $(\sigma_l - x_1, x_2, x_3, \sigma_l - x_4, p) = (\sigma_l-1, p) \neq \RR_{m,(l)}$.
Suppose (ii) holds.
Then $\begin{pmatrix} x_1 & x_2 \\ x_3 & x_4 \end{pmatrix}^2 \equiv \begin{pmatrix} 1 & 0 \\ 0 & 1 \end{pmatrix}$ and $x_1 + x_4 = a_l \equiv -2$ easily imply $x_1 \equiv x_4 \equiv -1$ and $x_2 \equiv x_3 \equiv 0$.
Thus $(\sigma_l - x_1, x_2, x_3, \sigma_l - x_4, p) = (\sigma_l + 1, p) \neq \RR_{m,(l)}$ by the second condition of (ii) and the above claim.
\end{proof}

\begin{cor}\label{cor:518}
Let $\bullet \in \{ \emptyset, +, -, \sharp, \flat\}$.
Suppose that $\Sel_{\prim(m)}^{\bullet}(E/K_{m, \infty})$ is $\Lambda$-cotorsion.
Under the condition $(\star)$ in Theorem \ref{thm:203}, we have
\[
\FF(\Sel^{\bullet}(E/K_{m, \infty})^{\dual}) = \left( \prod_{l \mid m} \left(1, \nu_{m,(l)} P_l^{-1} \right) \right) \FF(\Sel_{\prim(m)}^{\bullet}(E/K_{m, \infty})^{\dual}).
\]
\end{cor}

\begin{proof}
By Theorem \ref{thm:76}, we can apply Theorem \ref{thm:a01} to the exact sequence \eqref{eq:04} and thus
\begin{equation}\label{eq:31}
\FF(\Sel^{\bullet}(E/K_{m, \infty})^{\dual}) = \left(\prod_{l \mid m} \SF{-1}(H^1(\Q_l, \T_m)) \right) \FF(\Sel_{\prim(m)}^{\bullet}(E/K_{m, \infty})^{\dual}),
\end{equation}
where $\T_m = T_pE \otimes \RR_m$ is a representation of $\Gal(\overline{\Q}/\Q)$ over $\RR_m$.
By Lemma \ref{lem:a02}(1) and Proposition \ref{prop:510}, the assumption $(\star)$ implies $\SF{-1}(H^1(\Q_l, \T_m)) = \left(1, \nu_{m,(l)} P_l^{-1} \right)$ for each $l \mid m$.
This completes the proof.
\end{proof}

The following is the main result of this subsection.

\begin{prop}\label{prop:511}
Under the condition $(\star)$ in Theorem \ref{thm:203}, we have the following.

(1) Suppose $p \nmid a_p$ holds.
For an element $\xi \in \RR$, if $\FF(\Sel_{\prim(m)}(E/K_{m, \infty})^{\dual}) \supset (\xi)$, then
\[
\Fitt_{R_{m,n}}(\Sel(E/K_{m,n})^{\dual}) \supset \left( \prod_{l \mid m} \left( 1, \nu_{m, (l)}P_l^{-1} \right) \right)(\xi \bmod \omega_n)
\]
holds for any $n \geq 0$.

(2) Suppose $a_p = 0$ holds.
For an element $\xi^{\pm} \in \RR$, if $\au^{\pm} \FF(\Sel^{\pm}_{\prim(m)}(E/K_{m, \infty})^{\dual}) \supset (\xi^{\pm})$, then
\[
\Fitt_{R_{m,n}}(\Sel(E/K_{m,n})^{\dual}) \supset \left( \prod_{l \mid m} \left( 1, \nu_{m, (l)} P_l^{-1} \right) \right) (\ttilde{\omega}_n^{\mp} \xi^{\pm} \bmod \omega_n)
\]
holds for any $n \geq 0$.
\end{prop}

\begin{proof}
We follow the proof of \cite[Theorem 1.14]{KK}.

(1) The restriction map $H^1(K_{m,n}, E[p^{\infty}]) \to H^1(K_{m, \infty}, E[p^{\infty}])$ is injective by $H^0(K_{m, \infty}, E[p^{\infty}]) = 0$.
Hence $\Sel(E/K_{m,n}) \to \Sel(E/K_{m, \infty})[\omega_n]$ is also injective.
By the functoriality of Fitting ideals, we obtain
\[
\Fitt_{R_{m,n}}(\Sel(E/K_{m,n})^{\dual}) 
\supset \FF(\Sel(E/K_{m, \infty})^{\dual})R_{m,n}.
\]
By Corollary \ref{cor:518}, we have
\[
\FF(\Sel(E/K_{m, \infty})^{\dual}) = \left( \prod_{l \mid m} \left(1, \nu_{m, (l)} P_l^{-1} \right) \right) \FF(\Sel_{\prim(m)}(E/K_{m, \infty})^{\dual})
\supset \left( \prod_{l \mid m} \left(1, \nu_{m, (l)} P_l^{-1} \right)\right) \xi.
\]
These prove the assertion.

(2) We make use of the exact sequence
\begin{equation}
0 \to \Sel^{\pm}(E/K_{m,n}) \to \Sel(E/K_{m,n}) \to \frac{E(k_{m,n}) \otimes (\Q_p / \Z_p)}{E^{\pm}(k_{m,n}) \otimes (\Q_p / \Z_p)}
\end{equation}
by \eqref{eq:109}.
By a basic property of Fitting ideals, the Pontryagin dual of this sequence implies
\begin{equation}\label{eq:66}
\Fitt_{R_{m,n}}(\Sel(E/K_{m,n})^{\dual}) 
\supset \Fitt_{R_{m,n}}\left(\left(\frac{E(k_{m,n}) \otimes (\Q_p / \Z_p)}{E^{\pm}(k_{m,n}) \otimes (\Q_p / \Z_p)}\right)^{\dual}\right) 
\Fitt_{R_{m,n}}(\Sel^{\pm}(E/K_{m,n})^{\dual}).
\end{equation}
Moreover, the first factor in the right hand side is equal to $(\au^{\pm}\ttilde{\omega}_n^{\mp})$ by Lemma \ref{lem:623}.
On the other hand, similarly as in (1), the injectivity of the restriction map $\Sel^{\pm}(E/K_{m,n}) \to \Sel^{\pm}(E/K_{m, \infty})$ and Corollary \ref{cor:518} imply
\begin{equation}\label{eq:65}
\au^{\pm} \Fitt_{R_{m,n}}(\Sel^{\pm}(E/K_{m,n})^{\dual}) 
\supset \left( \prod_{l \mid m} \left( 1, \nu_{m, (l)} P_l^{-1} \right) \right) (\xi^{\pm} \bmod \omega_n).
\end{equation}
This completes the proof.
\end{proof}

\subsection{Analytic Side}\label{subsec:875}

Recall that $m$ denotes an integer such that $(m, pN) = 1$.
For $m' \mid m$ and $n \geq n' \geq -1$, we put $\nu^{m,n}_{m',n'} = \nu^{mp^{n+1}}_{m'p^{n'+1}}$ and $z^{m,n}_{m',n'} = z^{mp^{n+1}}_{m'p^{n'+1}}$.
For an allowable root $\alpha$ of $t^2-a_pt+p$, we define the $\alpha$-stabilized Mazur-Tate element in $R_{m,n} \otimes \Q_p(\alpha)$ by
\begin{equation}\label{eq:150}
\vartheta_{m,n}^{\alpha} = 
	\begin{cases}
		\alpha^{-(1+n)} (\theta_{m,n} - \alpha^{-1} \nu^{m,n}_{m,n-1}(\theta_{m,n-1})) & (n \geq 0) \\
		(1-\alpha^{-1} \varphi)(1-\alpha^{-1}\varphi^{-1})\theta_{m,-1} & (n = -1).
	\end{cases}
\end{equation}
By \eqref{eq:21}, these elements satisfy $z^{m,n}_{m,n-1}(\vartheta_{m,n}^{\alpha}) = \vartheta_{m,n-1}^{\alpha}$ for $n \geq 0$ and, for a prime divisor $l$ of $m$,
\begin{equation}\label{eq:905}
z^{m,n}_{m/l, n}(\vartheta_{m,n}^{\alpha}) =
	\begin{cases}
		(a_l - \sigma_l^{-1} - \sigma_l) \vartheta_{m/l,n}^{\alpha} & (l^2 \nmid m) \\
		a_l \vartheta_{m/l,n}^{\alpha} - \nu^{m/l, n}_{m/l^2, n}(\vartheta_{m/l^2, n}^{\alpha}) & (l^2 \mid m).
	\end{cases}
\end{equation}
Moreover, by \eqref{eq:20}, we have
\begin{equation}\label{eq:55}
\psi(\vartheta_{m,n}^{\alpha}) = e_p(\alpha, \psi) \tau(\psi^{-1}) \frac{L(E, \psi, 1)}{\Omega^{\sign(\psi)}}
\end{equation}
for a character $\psi$ modulo $mp^{n+1}$ with $m_{\psi} = m$ (not necessarily $n_{\psi} = n$).

Put $\LL_{m,n}^{\alpha} = \LL_{\prim(m)}(E/K_{\infty}, \alpha) \bmod \omega_n \in R_{m,n} \otimes \Q_p(\alpha)$.
By \eqref{eq:71}, this element can be characterized by
\begin{equation}\label{eq:56}
\psi(\LL_{m,n}^{\alpha}) = e_p(\alpha, \psi) \tau_{\prim(m)}(\psi^{-1}) \frac{L_{\prim(m)}(E, \psi, 1)}{\Omega^{\sign(\psi)}}
\end{equation}
for {\it any} character $\psi$ modulo $mp^{n+1}$.
This characterization (or Lemma \ref{lem:992}) implies
\begin{equation}\label{eq:906}
z^{m,n}_{m/l, n}(\LL_{m,n}^{\alpha}) =
	\begin{cases}
		(- \sigma_l^{-1} P_l) \LL_{m/l,n}^{\alpha}  & (l^2 \nmid m) \\
		\LL_{m/l,n}^{\alpha}& (l^2 \mid m).
	\end{cases}
\end{equation}

We shall compare $\vartheta_{m,n}^{\alpha}$ with $\LL_{m,n}^{\alpha}$.
The difference of the compatible properties \eqref{eq:905} and \eqref{eq:906} makes it difficult to compare these two elements (such a difficulty does not appear in \cite{KK}).
We shall show the following incomplete comparison, which suffices for us.
Let $\prim_1(m)$ denote the set of prime divisors $l$ of $m$ such that $l^2 \nmid m$.
Recall the element $\nu_{m, (T)}$ defined in Definition \ref{defn:996}.

\begin{prop}\label{prop:513}
We have equalities
\[
\prod_{l^2 \mid m} \left(1-l^{-1} \nu^{m,n}_{m/l,n} \right) 
\left[ \vartheta_{m,n}^{\alpha} - \sum_{T \subset \prim_1(m)} \nu_{m, (T)} \left( \prod_{l \in T} l^{-1}(a_l - \sigma_l) \right) \LL_{m_{(T)}, n}^{\alpha}  \right] = 0
\]
and
\[
\prod_{l^2 \mid m} \left(1-l^{-1} \nu^{m,n}_{m/l,n} \right) 
\left[ \LL_{m,n}^{\alpha} - \sum_{T \subset \prim_1(m)} \nu_{m, (T)} \left( \prod_{l \in T} l^{-1}(\sigma_l - a_l) \right) \vartheta_{m_{(T)}, n}^{\alpha}  \right] = 0.
\]
Here, in the first product, $l$ runs over the prime divisors of $m$ such that $l^2 \mid m$.
\end{prop}

\begin{proof}
We compute in a similar manner as in Proposition \ref{prop:641}.
We show only the first equality since the second can be shown similarly.

It is enough to show that $\psi(\text{left hand side}) = 0$ for any character $\psi$ of $G_{m,n}$.
If there exists a prime divisor $l$ of $m$ such that $l^2 \mid m$ and $m_{\psi}l \mid m$, then this assertion is clear from $\psi(1 - l^{-1} \nu^{m,n}_{m/l,n}) = 0$.
Suppose such a divisor $l$ does not exist, namely, $\prim(m) \setminus \prim(m_{\psi}) \subset \prim_1(m)$.
We can compute
\begin{align}
\psi(\vartheta_{m,n}^{\alpha}) 
&= \prod_{l \mid m, l \nmid m_{\psi}} (a_l - \psi(l)^{-1} - \psi(l)) \psi(\vartheta_{m_{\psi}, n}^{\alpha})\\
&= \prod_{l \mid m, l \nmid m_{\psi}} \left[ (1-l^{-1}) (a_l - \psi(l)) + \psi(-\sigma_l^{-1} P_l) \right]\psi(\LL_{m_{\psi}, n}^{\alpha})\\
&= \sum_{T \subset \prim(m) \setminus \prim(m_{\psi})} \left( \prod_{l \in T} (1-l^{-1})(a_l - \psi(l))\right) 
\left(\prod_{l \mid m, l \nmid m_{\psi}, l \not\in T} \psi(-\sigma_l^{-1} P_l) \right) \psi(\LL_{m_{\psi}, n}^{\alpha})\\
&= \sum_{T \subset \prim(m) \setminus \prim(m_{\psi})} \left( \prod_{l \in T} l^{-1} \psi(\nu_{m, (l)})  (a_l - \psi(l))\right) 
\psi(\LL_{m_{(T)}, n}^{\alpha})\\
&= \psi \left[ \sum_{T \subset \prim(m) \setminus \prim(m_{\psi})} \nu_{m, (T)} \left( \prod_{l \in T} l^{-1} (a_l - \sigma_l)\right) \LL_{m_{(T)}, n}^{\alpha} \right].
\end{align}
Here, the first equality follows from \eqref{eq:905};
the second follows from an easy computation and \eqref{eq:55} and \eqref{eq:56};
the third is an expansion of the product;
the fourth follows from \eqref{eq:906}.
By the same reasoning as the final step in the proof of Proposition \ref{prop:641}, we can replace the range of $T$ by $T \subset \prim_1(m)$ in the final formula.
This completes the proof of $\psi(\text{left hand side}) = 0$.
\end{proof}

\begin{cor}\label{cor:514}
(1) We have 
\[
\vartheta_{m,n}^{\alpha} \in \sum_{m' \mid m} (\nu^{m,n}_{m',n} \LL_{m',n}^{\alpha})_{R_{m, n}},
\]
where $m'$ runs over all divisors of $m$.
Moreover, if $a_p = 0$, the coefficients of $\nu^{m,n}_{m',n}\LL_{m',n}^{\alpha}$ to express $\vartheta_{m,n}^{\alpha}$ can be given independent of $\alpha$.

(2) We have the equality
\[
\sum_{m' \mid m} (\nu^{m,n}_{m',n} \vartheta_{m',n}^{\alpha})_{R_{m,n}} = \sum_{m' \mid m} (\nu^{m,n}_{m',n} \LL_{m',n}^{\alpha})_{R_{m,n}}.
\]
\end{cor}

Here, the subscripts $R_{m,n}$ are attached to emphasize that we consider the generated $R_{m,n}$-submodules.
But in the following, we will often omit the subscripts when no confusion occurs.

\begin{proof}
(1) The final assertion on independence of $\alpha$ can be observed from the proof below and we do not mention any more.
We use induction on $m$.
By Proposition \ref{prop:513}, 
it is enough to show that 
\[
\nu^{m,n}_{m/l_0,n} \vartheta_{m,n}^{\alpha} \in \sum_{m' \mid m} (\nu^{m,n}_{m',n} \LL_{m',n}^{\alpha})
\]
for any $l_0^2 \mid m$.
Since \eqref{eq:905} implies 
\[
\nu^{m, n}_{m/l_0, n} \vartheta_{m,n}^{\alpha} 
= a_{l_0} \nu^{m, n}_{m/l_0, n} \vartheta_{m/l_0, n}^{\alpha} - \nu^{m, n}_{m/l_0^2, n} \vartheta_{m/l_0^2, n}^{\alpha},
\]
the induction hypothesis for $m/l_0$ and $m/l_0^2$ implies the consequence.

(2) The inclusion $\subset$ follows from (1), using induction on $m$.
The other inclusion can be shown similarly.
\end{proof}

\begin{prop}\label{prop:515}
(1) Suppose $p \nmid a_p$ holds.
Then, for $n \geq 0$, we have
\[
(\vartheta_{m,n}^{\alpha}) = (\theta_{m,n}, \nu^{m,n}_{m,n-1}(\theta_{m,n-1})).
\]

(2) Suppose $a_p = 0$ holds.
Then, for $n \geq 0$, we have
\[
\frac{1}{2} p^{[(n+2)/2]} (\vartheta_{m,n}^{\alpha} + \vartheta_{m,n}^{-\alpha})
 = 	\begin{cases}
	 	(-1)^{n/2} \nu^{m,n}_{m,n-1}(\theta_{m,n-1}) & (\text{$n$ is even}) \\
		(-1)^{(n+1)/2} \theta_{m,n} & (\text{$n$ is odd})
	\end{cases}
\]
and
\[
\frac{1}{2\alpha} p^{[(n+3)/2]} (\vartheta_{m,n}^{\alpha} - \vartheta_{m,n}^{-\alpha})
 = 	\begin{cases}
	 	(-1)^{(n+1)/2} \nu^{m,n}_{m,n-1}(\theta_{m,n-1}) & (\text{$n$ is odd}) \\
		(-1)^{(n+2)/2} \theta_{m,n} & (\text{$n$ is even}).
	\end{cases}
\]
\end{prop}

\begin{proof}
(1) This is asserted in \cite[\S 2.2]{KK} for $m = 1$ (and the $\Delta$-invariant part).
First we show the assertion for $n=0$, namely 
\begin{equation}\label{eq:59}
(\theta_{m,0} - \alpha^{-1} \nu^{m,0}_{m,-1}(\theta_{m, -1})) = (\theta_{m,0}, \nu^{m,0}_{m,-1}(\theta_{m,-1})).
\end{equation}
We can divide this assertion with respect to the decomposition $R_0 = R_0^{\Delta} \times \Rnt$.
The $\Rnt$-part of \eqref{eq:59} is clear, since $\nu^{m,0}_{m,-1}(\theta_{m, -1})$ does not contribute.
To verify the $R_0^{\Delta}$-part of \eqref{eq:59}, we can apply $z^{m,0}_{m,-1}$ and thus it is enough to show
\begin{equation}\label{eq:62}
((a_p-\varphi-\varphi^{-1}) \theta_{m,-1} - \alpha^{-1} (p-1) \theta_{m,-1})
= ((a_p-\varphi-\varphi^{-1}) \theta_{m,-1}, (p-1) \theta_{m,-1}).
\end{equation}
The right hand side of \eqref{eq:62} is generated by $\theta_{m,-1}$ (by the second element).
 Since $a_p \equiv \alpha$ modulo $p$, we can compute
 \[
 a_p - \varphi - \varphi^{-1} - \alpha^{-1}(p-1) \equiv \alpha^{-1}(\alpha-\varphi)(\alpha-\varphi^{-1}).
 \]
This element is a unit in $R_{m, -1}$ by Lemma \ref{lem:100}, which proves \eqref{eq:62}.
This proves the assertion for $n = 0$.

Secondly, we claim $\theta_{m, n-1} \in (z^{m,n}_{m,n-1}(\theta_{m,n}))$ for $n \geq 1$.
If $n = 1$, this claim $\theta_{m,0} \in (a_p\theta_{m,0} - \nu^{m,0}_{m,-1}(\theta_{m,-1}))$ can again be divided into $R_0^{\Delta}$ and $\Rnt$, and the assertion on $\Rnt$ is clear.
For the $R_0^{\Delta}$-part, by applying $z^{m,0}_{m,-1}$, it is enough to show
\begin{equation}\label{eq:153}
(a_p-\varphi - \varphi^{-1})\theta_{m,-1} \in (a_p(a_p-\varphi - \varphi^{-1})\theta_{m,-1} - (p-1) \theta_{m,-1}).
\end{equation}
By a similar computation as the proof of \eqref{eq:62}, we can show that the right hand side of \eqref{eq:153} is generated by $\theta_{m,-1}$.
Therefore \eqref{eq:153} holds.
If $n \geq 2$, by induction we may suppose that $\theta_{m, n-2} \in (z^{m,n-1}_{m,n-2}(\theta_{m,n-1}))$, which implies $\nu^{m,n-1}_{m,n-2}(\theta_{m, n-2}) \in (\nu^{m,n-1}_{m,n-2} \theta_{m,n-1})$.
Since $\nu^{m,n-1}_{m,n-2}$ is contained in the Jacobson radical of $R_{m, n-1}$, the element $z^{m,n}_{m,n-1}(\theta_{m,n}) = a_p\theta_{m,n-1}-\nu^{m,n-1}_{m,n-2}(\theta_{m,n-2})$ should be equal to $\theta_{m,n-1}$ up to a unit of $R_{m,n-1}$.
Thus we have $(z^{m,n}_{m,n-1}(\theta_{m,n})) = (\theta_{m,n-1})$.

Finally, for any $n \geq 1$, the second claim above implies $\nu^{m,n}_{m,n-1}(\theta_{m,n-1}) \in (\nu^{m,n}_{m,n-1}(\theta_{m,n}))$.
Then, again using the fact that $\nu^{m,n}_{m,n-1}$ is contained in the Jacobson radical of $R_{m, n}$, we see that the element $\vartheta^{\alpha}_{m,n} = \alpha^{-(1+n)} (\theta_{m,n} - \alpha^{-1} \nu^{m,n}_{m,n-1}(\theta_{m,n-1}))$ is equal to $\theta_{m,n}$ up to a unit of $R_{m,n}$.
Therefore
\[
(\vartheta_{m,n}^{\alpha}) = (\theta_{m,n}) = (\theta_{m,n}, \nu^{m,n}_{m,n-1}(\theta_{m,n-1})),
\]
where the second equality again follows from the second claim above.

(2) This is shown by a direct computation, using the definition \eqref{eq:150}.
\end{proof}

For $\bullet \in \{\emptyset, +, -\}$, put $\LL_{m,n}^{\bullet} = \LL_{\prim(m)}^{\bullet}(E/K_{m, \infty}) \bmod \omega_n \in R_{m,n} \otimes \Q_p$.

\begin{lem}\label{lem:976}
Suppose $a_p = 0$ holds.
Then we have 
\[
\frac{1}{2} p^{[(n+2)/2]} (\LL_{m,n}^{\alpha} + \LL_{m,n}^{-\alpha})
 = \ttilde{\omega}_n^+ \LL_{m, n}^-
 \]
and
\[
\frac{1}{2\alpha} p^{[(n+3)/2]} (\LL_{m,n}^{\alpha} - \LL_{m,n}^{-\alpha})
 = \ttilde{\omega}_n^- \LL_{m, n}^+
  \]
 for $n \geq 0$.
\end{lem}

\begin{proof}
This follows from \eqref{eq:125} and the observations $\log^- \equiv p^{-[(n+3)/2]}\ttilde{\omega}_n^- \mod \omega_n$ and $\log^+ \equiv p^{-[(n+2)/2]}\ttilde{\omega}_n^+ \mod \omega_n$.
\end{proof}

\begin{lem}\label{lem:970}
For $\bullet \in \{\emptyset, +, -\}$, we have
\[
\sum_{m' \mid m} (\nu^{m,n}_{m',n} \LL_{m',n}^{\bullet})
= \sum_{T \subset \prim(m)} (\nu_{m, (T)} \LL_{m_{(T)},n}^{\bullet})
= \left(\prod_{l \mid m} (1, \nu_{m, (l)}P_l^{-1}) \right) \LL_{m,n}^{\bullet}.
\]
\end{lem}

\begin{proof}
The second equality follows from Lemma \ref{lem:992} (cf. \eqref{eq:906}).
The inclusion $\supset$ of the first equality is trivial.
For the inverse inclusion, take any $m' \mid m$ and let $T = \prim(m) \setminus \prim(m')$.
Then we have 
\[
\nu^{m,n}_{m',n} \LL_{m',n}^{\bullet} 
= \nu_{m, (T)}\nu^{m_{(T)},n}_{m',n} \LL_{m', n}^{\bullet} 
= \nu_{m, (T)}\nu^{m_{(T)},n}_{m',n} \LL_{m_{(T)},n}^{\bullet}
\in (\nu_{m, (T)} \LL_{m_{(T)},n}^{\bullet}).
\]
\end{proof}

The following is the main result of this subsection.

\begin{prop}\label{prop:516}
(1) Suppose $p \nmid a_p$ holds.
Then, for $n \geq 0$, we have
\[
(\theta_{m,n}, \nu^{m,n}_{m,n-1}(\theta_{m,n-1}))
\subset \left( \prod_{l \mid m} \left( 1, \nu_{m, (l)} P_l^{-1} \right) \right) \LL_{m,n}.
\]

(2) Suppose $a_p = 0$ holds.
Then, for $n \geq 0$, we have
\[
(\theta_{m,n}, \nu^{m,n}_{m,n-1}(\theta_{m,n-1}))
\subset \left( \prod_{l \mid m} \left( 1, \nu_{m, (l)} P_l^{-1} \right) \right) (\ttilde{\omega}_n^+ \LL^-_{m,n}, \ttilde{\omega}_n^- \LL^+_{m,n}).
\]
\end{prop}

\begin{proof}
(1) We have
\[
(\theta_{m,n}, \nu^{m,n}_{m,n-1}(\theta_{m,n-1}))
= (\vartheta_{m,n}^{\alpha})
\subset \sum_{m' \mid m} (\nu^{m,n}_{m',n} \LL_{m',n}^{\alpha})
= \left( \prod_{l \mid m} \left( 1, \nu_{m, (l)} P_l^{-1} \right) \right) \LL_{m,n}
\]
by Proposition \ref{prop:515}(1), Corollary \ref{cor:514}(1), and Lemma \ref{lem:970}, respectively.

(2) This corresponds to \cite[Proposition 1.11]{KK}.
Similarly as in (1), we have
\begin{align}
(\theta_{m,n}, \nu^{m,n}_{m,n-1}(\theta_{m,n-1}))
&= \left(\frac{1}{2} p^{[(n+2)/2]} (\vartheta_{m,n}^{\alpha} + \vartheta_{m,n}^{-\alpha}), \frac{1}{2\alpha} p^{[(n+3)/2]} (\vartheta_{m,n}^{\alpha} - \vartheta_{m,n}^{-\alpha}) \right)\\
&\subset \sum_{m' \mid m} \nu^{m,n}_{m',n} \left(\frac{1}{2} p^{[(n+2)/2]} (\LL_{m',n}^{\alpha} + \LL_{m',n}^{-\alpha}), \frac{1}{2\alpha} p^{[(n+3)/2]} (\LL_{m',n}^{\alpha} - \LL_{m',n}^{-\alpha}) \right)\\
&= \sum_{m' \mid m} \nu^{m,n}_{m',n} \left(\ttilde{\omega}_n^-\LL^+_{m',n}, \ttilde{\omega}_n^+\LL^-_{m',n} \right)\\
&= \left( \prod_{l \mid m} \left( 1, \nu_{m, (l)} P_l^{-1} \right) \right) \left( \ttilde{\omega}_n^- \LL^+_{m,n}, \ttilde{\omega}_n^+ \LL^-_{m,n} \right)
\end{align}
by Proposition \ref{prop:515}(2), Corollary \ref{cor:514}(1), Lemma \ref{lem:976}, and Lemma \ref{lem:970}, respectively.
\end{proof}

\subsection{Conclusion of Proof}\label{subsec:874}

The following is our generalization of \cite[Theorem 1.14]{KK}.

\begin{thm}\label{thm:151}
Suppose $p \nmid a_p$ or $a_p = 0$ holds.
Let $m$ be a positive integer relatively prime to $pN$ such that Assumption \ref{ass:04} holds for $K = \Q(\mu_m)$ in the ordinary case.
Suppose the condition $(\star)$ in Theorem \ref{thm:203} holds.
Then, if 
\[
\au^{\bullet} \Fitt_{\RR}(\Sel_{\prim(m)}^{\bullet}(E/K_{m, \infty})^{\dual}) \supset (\LL_{\prim(m)}^{\bullet}(E/K_{m, \infty}))
\]
holds for any possible $\bullet \in \{\emptyset, +, -\}$, then we have 
 \[
 (\theta_{m,n}, \nu^{m,n}_{m,n-1}(\theta_{m,n-1})) \subset \Fitt_{R_{m,n}} (\Sel(E/K_{m,n})^{\dual})
 \]
for any $n \geq 0$.
\end{thm}

\begin{proof}
This follows from Propositions \ref{prop:511} and \ref{prop:516}.
\end{proof}

Note that the appearance of $\nu^{m,n}_{m,n-1}(\theta_{m,n-1})$ is observed by Kurihara \cite[Conjecture 0.3]{Kur02}, at least when we are concerned with $m = 1$ (and the $\Delta$-invariant part).

\begin{proof}[Proof of Theorem \ref{thm:203}]
Theorem \ref{thm:151} already implies the assertion for $M = mp^{n+1}$ with $n \geq 0$.
Also Theorem \ref{thm:151} for $n = 0$ implies $\nu^{m,0}_{m,-1}(\theta_{m,-1}) \in \Fitt_{R_{m,0}} (\Sel(E/K_{m,0})^{\dual})$.
Then applying $z^{m,0}_{m,-1}$ to it yields the assertion for $M = m$.
This completes the proof of Theorem \ref{thm:203}.
\end{proof}

\subsection{A Conjecture of Kurihara}\label{subsec:974}

We continue to consider an integer $m$ with $(m, pN) = 1$.
In the ordinary case, Kurihara proposed a conjecture \cite[Conjecture 10.1]{Kur03} that
\begin{equation}\label{eq:74}
\FF(\Sel(E/K_{m, \infty})^{\dual}) =  \sum_{m' \mid m} \left( \nu^{m,\infty}_{m', \infty} \vartheta_{m', \infty}^{\alpha} \right),
\end{equation}
where $\alpha$ is the unit root, $\vartheta_{m', \infty}^{\alpha} = (\vartheta_{m', n}^{\alpha})_n \in \RR_{m'}$, and $\nu^{m,\infty}_{m', \infty}$ has the obvious meaning.
This is a kind of equivariant conjecture.
Under certain hypotheses, \cite[Theorem 10.2]{Kur03} proves the inclusion $\supset$ of this conjecture.
Our method also reproves that assertion under our running hypotheses, as follows.

\begin{thm}\label{thm:903}
Suppose the hypotheses in Theorem \ref{thm:203} hold. 
For $\bullet \in \{\emptyset, +, -\}$, we have 
\[
 \sum_{T \subset \prim(m)} \left( \nu_{m, (T)} \LL_{\prim(m_{(T)})}^{\bullet}(E/K_{m_{(T)}, \infty}) \right) 
 \subset \FF(\Sel^{\bullet}(E/K_{m, \infty})^{\dual}).
\]
This is an equality if the equality \eqref{eq:241} holds.
\end{thm}

\begin{proof}
By Corollary \ref{cor:518}, we have
\begin{align}
\FF(\Sel^{\bullet}(E/K_{m, \infty})^{\dual}) 
&\supset  \left( \prod_{l \mid m} \left(1, \nu_{m,(l)} P_l^{-1} \right) \right) \LL_{\prim(m)}^{\bullet}(E/K_{m, \infty})\\
&= \sum_{T \subset \prim(m)} \left( \nu_{m, (T)} \LL_{\prim(m_{(T)})}^{\bullet}(E/K_{m_{(T)}, \infty}) \right),
\end{align}
which is an equality if \eqref{eq:241} holds.
\end{proof}

If $p \nmid a_p$, then Corollary \ref{cor:514}(2) and Lemma \ref{lem:970} imply
\[
 \sum_{m' \mid m} \left( \nu^{m,\infty}_{m', \infty} \vartheta_{m', \infty}^{\alpha} \right) 
 =  \sum_{T \subset \prim(m)} \left( \nu_{m, (T)} \LL_{\prim(m_{(T)})}(E/K_{m_{(T)}, \infty}) \right).
 \]
Therefore, in the ordinary case, Theorem \ref{thm:903} proves the one inclusion $\supset$ of \eqref{eq:74} in the situation of Theorem \ref{thm:203} (moreover, the equality \eqref{eq:74} is implied by the equality \eqref{eq:241}).
In particular, in the situation of Corollary \ref{cor:204}, the one inclusion $\supset$ of \eqref{eq:74} is confirmed.

Note that the hypotheses of \cite[Theorem 10.2]{Kur03} are slightly different from ours in Corollary \ref{cor:204}.
On the one hand, the conditions (b) and (c) are weakened in \cite{Kur03}.
On the other hand, \cite{Kur03} assumes that the $\mu$-invariant of $\Sel(E/K_{\infty})^{\dual}$ is zero, which is stronger than condition (e) as in Remark \ref{rem:a11}.

\renewcommand{\thesection}{\Alph{section}}
\setcounter{section}{0}

\section{Appendix: Properties of Fitting Ideals}\label{sec:108}

In this appendix, we collect auxiliary propositions about Fitting ideals, using the results in \cite{Kata}.

Let $G$ be a finite abelian group and put $\Lambda = \Z_p[[\X]], \RR = \Z_p[G][[\X]]$.
We denote by $\MM$ the category of finitely generated torsion $\RR$-modules, where torsionness means being torsion as a $\Lambda$-module.
Let $\PP$ be the subcategory of $\MM$ consisting of module with $\pd_{\RR} \leq 1$.
Let $\CC$ be the subcategory of $\MM$ consisting of modules which does not have nonzero finite submodules.
It is well-known (e.g. \cite[Proposition (5.3.19)(i)]{NSW}) that a finitely generated $\RR$-module $X$ does not contain a non-trivial finite submodule if and only if $\pd_{\Lambda}(X) \leq 1$.
Therefore $\PP \subset \CC$, meaning that any object of $\PP$ is an object of $\CC$.

\begin{defn}
Let $\Omega$ be a commutative monoid.

(1) A map $\FF: \MM \to \Omega$ is called a {\it Fitting invariant} if the following conditions hold:
\begin{itemize}
\item If $P \in \PP$, then $\FF(P) \in \Omega$ is invertible.
\item If $P \in \PP$ and $X, X' \in \MM$ fit into an exact sequence $0 \to X' \to X \to P \to 0$,
then $\FF(X) = \FF(P)\FF(X')$.
\end{itemize}

(2) A map $\FF: \CC \to \Omega$ is called a {\it quasi-Fitting invariant} if the following conditions hold:
\begin{itemize}
\item If $P \in \PP$, then $\FF(P) \in \Omega$ is invertible.
\item If $P \in \PP$ and $X, X' \in \CC$ fit into an exact sequence $0 \to X' \to X \to P \to 0$,
then $\FF(X) = \FF(P)\FF(X')$.
\item If $P \in \PP$ and $X, X' \in \CC$ fit into an exact sequence $0 \to P \to X \to X' \to 0$,
then $\FF(X) = \FF(P)\FF(X')$.
\end{itemize}
\end{defn}

This definition of Fitting invariants (resp. quasi-Fitting invariants) is given in \cite[Definition 2.4]{Kata} (resp. \cite[Definition 3.16]{Kata}).
A fundamental example is the following.

\begin{prop}\label{prop:613}
Let $\Omega$ be the monoid of fractional ideals of $\RR$ and we put $\FF(X) = \Fitt_{\RR}(X)$, the Fitting ideal of $X$.
Then $\FF: \MM \to \Omega$ is a Fitting invariant.
Moreover, the restriction $\FF|_{\CC}: \CC \to \Omega$ is a quasi-Fitting invariant.
\end{prop}

\begin{proof}
The first assertion is proved in \cite[Proposition 2.7]{Kata}.
Then the second follows, since \cite[Proposition 3.17]{Kata} proves that, in our present situation, any Fitting invariant gives rise to a quasi-Fitting invariant by restriction to $\CC$.
\end{proof}

Next we state the definition of {\it shifts} of (quasi-)Fitting invariants, introduced in \cite{Kata}.

\begin{thm}[{\cite[Theorem 2.6]{Kata}}]\label{thm:614}
Let $\FF: \MM \to \Omega$ be a Fitting invariant.
For $X \in \MM$ and $n \geq 0$, the following $\varSF{n}(X) \in \Omega$ is well-defined.
Take an exact sequence $0 \to Y \to P_1 \to \dots \to P_n \to X \to 0$ in $\MM$ with $P_i \in \PP$ for $1 \leq i \leq n$.
Then define
\[
\varSF{n}(X) = \left( \prod_{i=1}^n \FF(P_i)^{(-1)^i} \right) \FF(Y).
\] 
\end{thm}

\begin{thm}[{\cite[Corollary 3.21]{Kata}}]\label{thm:a01}
Let $\FF: \CC \to \Omega$ be a quasi-Fitting invariant.
Then there exists a unique family $\{\SF{n}: \MM \to \Omega\}_{n \in \Z}$ of maps satisfying the following.
Firstly, $\SF{0}: \MM \to \Omega$ is an extension of $\FF: \CC \to \Omega$.
Secondly, if $0 \to Y \to P_1 \to \dots \to P_d \to X \to 0$ is an exact sequence in $\MM$ with $P_i \in \PP$ for $1 \leq i \leq d$, then 
\[
\SF{n}(X) = \left( \prod_{i=1}^{d} \FF(P_i)^{(-1)^{n-d+i}} \right) \SF{n-d}(Y).
\]
\end{thm}

The rest of this section is devoted to algebraic propositions which we use in Section \ref{sec:84}.

\begin{defn}\label{defn:995}
Let $\II, \JJ$ be fractional ideals of $\RR$.
If $\II \subset \JJ$ and moreover the quotient $\JJ/\II$ is finite, then we write $\II \subset_{\fin} \JJ$ and $\JJ \supset_{\fin} \II$.
\end{defn}

\begin{lem}\label{lem:111}
Put $\FF = \Fitt_{\RR}$.
For $X \in \MM$ and a finite submodule $X'$ of $X$, we have $\FF(X) \subset_{\fin} \FF(X/X')$.
\end{lem}

\begin{proof}
It is well-known that
\[
\FF(X/X') \FF(X') \subset \FF(X) \subset \FF(X/X')
\]
in this case.
Since $\FF(X') \subset_{\fin} \RR$, the lemma follows.
\end{proof}

\begin{defn}\label{defn:977}
For an $\RR$-module $X$, we define $E^i(X) = \Ext_{\RR}^i(X, \RR)$ for $i \geq 0$.
Since $\Hom_{\RR}(X, \RR)$ is an $\RR$-module by $(af)(x) = a f(x)$ for $a \in \RR, f \in \Hom_{\RR}(X, \RR)$, and $x \in \RR$, the derived functor $E^i(X)$ also admits an $\RR$-module structure.
\end{defn}

\begin{prop}\label{prop:112}
Put $\FF = \Fitt_{\RR}$.
Let $0 \to X \to P_1 \to P_2 \to Y \to 0$ be an exact sequence in $\MM$.
Suppose that $P_i \in \PP$ holds for $i = 1,2$.
Then 
\[
\FF(P_1)\FF(Y) \subset_{\fin} \FF(P_2)\FF(E^1(X)).
\]
\end{prop}

\begin{proof}
Note that this is an equality if $Y \in \CC$ (see \cite[Lemma 5]{BG03} or \cite[Proposition 4.7]{Kata}).
To treat the case where $Y \not\in \CC$, we modify the argument of \cite{Kata}.

Define $\FF^*: \CC \to \Omega$ by $\FF^*(X) = \FF(E^1(X))$ for $X \in \CC$.
Then $\FF^*$ is again a quasi-Fitting invariant by the duality properties of $E^1$ on $\CC$ (see \cite[Proposition 3.11]{Kata}).
Hence Theorem \ref{thm:a01} yields maps $(\FF^*)^{ \langle n \rangle}: \MM \to \Omega$, which satisfies
\[
(\FF^*)^{ \langle 2 \rangle}(Y) = \FF^*(P_2) \FF^*(P_1)^{-1} (\FF^*)^{\langle 0 \rangle}(X).
\]
Observe that $\FF^*(P_i) = \FF(E^1(P_i)) = \FF(P_i)$ by \cite[Lemma 4.6]{Kata} and $(\FF^*)^{\langle 0 \rangle}(X) = \FF^*(X) = \FF(E^1(X))$ by $X \in \CC$.
Therefore we have 
\begin{equation}\label{eq:231}
(\FF^*)^{ \langle 2 \rangle}(Y) = \FF(P_2)\FF(P_1)^{-1}\FF(E^1(X)).
\end{equation}
Thus our goal is to show $\FF(Y) \subset_{\fin} (\FF^*)^{ \langle 2 \rangle}(Y)$ for $Y \in \MM$.

Take an element $f \in \Lambda \setminus \{0\}$ which annihilates $Y$.
By a presentation $\RR^b \overset{h}{\to} \RR^a \to Y \to 0$ of $Y$, construct an exact sequence
\begin{equation}\label{eq:86}
0 \to X' \to (\RR/f)^b \overset{\overline{h}}{\to} (\RR/f)^a \to Y \to 0.
\end{equation}
Then by \eqref{eq:231}, we have $(\FF^*)^{ \langle 2 \rangle}(Y) = f^{a-b}\FF(E^1(X'))$.
Let $Z$ be the image of $\overline{h}$ in the sequence \eqref{eq:86}, which is contained in $\CC$.
Then the two short exact sequences obtained by splitting \eqref{eq:86} induce the following commutative diagram with exact row and column.
\[
\xymatrix{
0 \ar[r] &
E^1(Y) \ar[r] &
E^1((R/f)^a) \ar[r] \ar[rd]_{\overline{h}^T} &
E^1(Z) \ar[r] \ar@{^{(}->}[d] &
E^2(Y) \ar[r] &
0\\
&&& E^1((\RR/f)^b) \ar@{->>}[d] & & \\
&&& E^1(X') &&
}
\]
Here, we identify $E^1(R/f) \simeq R/f$ naturally and the superscript $T$ denotes the transpose.
This diagram induces an exact sequence
\[
0 \to E^2(Y) \to \Cok(\overline{h}^T) \to E^1(X') \to 0.
\]
Since $E^2(Y)$ is a finite module, Lemma \ref{lem:111} implies that $\FF(\Cok(\overline{h}^T)) \subset_{\fin} \FF(E^1(X'))$.
Therefore
\[
(\FF^*)^{ \langle 2 \rangle}(Y) = f^{a-b} \FF(E^1(X')) \supset_{\fin} f^{a-b} \FF(\Cok(\overline{h}^T)) = \FF(\Cok(\overline{h})) = \FF(Y),
\]
which completes the proof.
\end{proof}

For a prime ideal $\qu$ of $\Lambda$ of height one, let $\Lambda_{\qu}$ be the localization of $\Lambda$ at $\qu$ and put $\RR_{\qu} = \Lambda_{\qu} \otimes_{\Lambda} \RR$.

\begin{defn}\label{defn:110}
For fractional ideals $\II, \JJ$ of $\RR$, we say $\II$ and $\JJ$ are commensurable if both $\II\cap \JJ \subset_{\fin} \II$ and $\II\cap \JJ \subset_{\fin} \JJ$ hold.
In that case we write $\II \sim_{\fin} \JJ$.
Equivalently, $\II \sim_{\fin} \JJ$ if and only if $\II \RR_{\qu} = \JJ \RR_{\qu}$ for any prime ideal $\qu$ of $\Lambda$ of height $1$.
It follows that $\sim_{\fin}$ is an equivalence relation.
\end{defn}

\begin{lem}\label{lem:109}
Let $f, g \in \RR$ be non-zero-divisors and $\II \subset \RR$ be an ideal containing a non-zero-divisor.
Suppose that either $\II \RR_{p\Lambda} = \RR_{p\Lambda}$ or the order of $G$ is relatively prime to $p$ holds.
Then $f\II \sim_{\fin} g\II$ implies $f\RR = g\RR$.
\end{lem}

\begin{proof}
At first we suppose $\II = \RR$ holds.
Consider the natural injective map $f\RR/(f\RR \cap g\RR) \hookrightarrow \RR/g\RR$.
On the one hand, the assumption $f\RR \sim_{\fin} g\RR$ implies that $f\RR/(f\RR \cap g\RR)$ is finite.
On the other hand, $\RR/g\RR$ does not contain non-trivial finite submodules since $\RR/g\RR \in \PP \subset \CC$.
Therefore we obtain $f\RR \subset g\RR$.
By symmetry we conclude $f\RR = g\RR$, which proves the lemma when $\II = \RR$.

In general, for any prime ideal $\qu$ of $\Lambda$ of height $1$, we claim that the ideal $\II \RR_{\qu}$ is invertible.
If $\qu \neq p\Lambda$ or the order of $G$ is relatively prime to $p$, then this is clear since $\RR_{\qu}$ is a product of principal ideal domains (and $\II$ contains a non-zero-divisor).
Otherwise, our assumption $\II\RR_{p\Lambda} = \RR_{p\Lambda}$ implies the claim.
Now suppose $f\II \sim_{\fin} g\II$ holds.
Then, for any $\qu$, we have $f\II \RR_{\qu} = g\II \RR_{\qu}$, which implies $f \RR_{\qu} = g \RR_{\qu}$ by the above claim.
Therefore $f\RR \sim_{\fin} g\RR$ and we conclude $f\RR = g\RR$ by the case $\II = \RR$, which is already established.
This completes the proof.
\end{proof}

\begin{rem}\label{rem:110}
The assumption $\II \RR_{p\Lambda} = \RR_{p\Lambda}$ is necessary when the order of $G$ is divisible by $p$.
For example, suppose $G$ is a cyclic group of order $p$ and consider $\II = (p, N_G)$, where $N_G$ is the norm element.
Then we can easily verify $(p+N_G)\II = p\II$ and $(p+N_G)\RR \neq p\RR$.
\end{rem}

\section*{Acknowledgments}

I would like to express my deepest gratitude to Takeshi Tsuji for his support during the research.
I also thank Chan-Ho Kim, Takahiro Kitajima, Masato Kurihara, Rei Otsuki, and Ryotaro Sakamoto.
They answered my questions and also gave suggestions.
This research was supported by JSPS KAKENHI Grant Number 17J04650, and by the Program for Leading Graduate Schools (FMSP) at the University of Tokyo.

{
\bibliographystyle{abbrv}
\bibliography{ellip}
}

\end{document}